\documentclass[11pt]{article}
\usepackage{amsmath, latexsym, amsfonts, amssymb, amsthm, amscd}
\usepackage[left=2.2cm, right=2.2cm, top=2.5cm, bottom=2.5cm]{geometry}
\usepackage{upquote}
\usepackage{color}
\usepackage{setspace}
\usepackage[labelfont=bf]{caption}

\usepackage[utf8]{inputenc}
\usepackage[T1]{fontenc}
\usepackage[english]{babel}
\usepackage{dsfont}
\usepackage{mathtools} 

\usepackage[svgnames]{xcolor}
\definecolor{webgreen}{HTML}{008000}

\usepackage{hyperref}
\hypersetup{colorlinks, breaklinks, urlcolor=Maroon, linkcolor=Maroon, citecolor=webgreen} 

\allowdisplaybreaks

\usepackage{bookmark}



\newtheorem{corollary}{Corollary}
\newtheorem{proposition}{Proposition}
\newtheorem{lemma}{Lemma}
\newtheorem{remark}{Remark}

\newtheorem{theorem}{Theorem}

\theoremstyle{definition}

\newtheorem{Condition}{Condition}

\theoremstyle{remark}
\newtheorem*{general*}{A general remark}

\usepackage{soul}
\usepackage[authormarkup=none]{changes}

\definechangesauthor[name={Cai},color=blue]{C}

\usepackage[title]{appendix}

\begin{document}
\title{The fluctuations of the giant cluster for percolation on random split trees}
\author{Gabriel Berzunza\footnote{ E-mail:
        \href{mailto:gabriel.berzunza-ojeda@math.uu.se}{gabriel.berzunza-ojeda@math.uu.se}},\,\,
    Xing Shi Cai \footnote{  E-mail: \href{mailto:xingshi.cai@math.uu.se}{xingshi.cai@math.uu.se}} \,  and
    Cecilia Holmgren\footnote{ E-mail: \href{mailto:cecilia.holmgren@math.uu.se}{cecilia.holmgren@math.uu.se}} \\ \vspace*{10mm}
{\small Department of Mathematics, Uppsala University, Sweden} }
\maketitle

\vspace{0.1in}

\begin{abstract} 
A split tree of cardinality $n$ is constructed by distributing $n$ ``balls'' in a subset of vertices of an infinite tree which encompasses many types of random trees such as $m$-ary search trees, quad trees, median-of-$(2k+1)$ trees, fringe-balanced trees, digital search trees and random simplex trees. 
In this work, we study Bernoulli bond percolation on arbitrary split trees of large but finite
cardinality $n$. We show for appropriate percolation regimes that depend on the cardinality $n$ of
the split tree that there exists a unique giant cluster, the fluctuations of the size of the giant cluster as \(n \to \infty\) are described by an infinitely divisible distribution that belongs to the class of stable
(asymmetric) Cauchy laws. This work generalizes the results for the random $m$-ary recursive trees in Berzunza
\cite{Ber2015}. Our approach is based on a remarkable decomposition of the
size of the giant percolation cluster as a sum of essentially independent random variables which may be useful for studying
percolation on other trees with logarithmic height; for instance in this work we study
also the case of regular trees.
\end{abstract}



\section{Introduction}

Consider a tree $T_{n}$ of large but finite size $n \in \mathbb{N}$ and perform Bernoulli
bond-percolation with parameter $p_{n} \in [0,1]$ that depends on the size of the graph. This means
that we remove each edge in $T_{n}$ with probability $1 - p_{n}$, independently of the other edges,
inducing a partition of the set of vertices into connected clusters. In particular, we are interested in the supercritical percolation regime, in the sense that with high
probability, there exists a giant cluster, that is of size comparable to that of the entire tree.
Bertoin \cite{Be1} established for several families of trees with $n$ vertices that the
supercritical regime corresponds to percolation parameters of the form $1-p_{n} = c/\ell(n) +
o(1/\ell(n))$ as $n \rightarrow \infty$, where $c > 0$ is fixed and $\ell(n)$ is an estimate of the
height of a typical vertex in the tree structure\footnote{For two sequences of real numbers
    $(A_{n})_{n \geq 1}$ and $(B_{n})_{n \geq 1}$ such that $B_{n} > 0$, we write $A_{n} = o(B_{n})$
    if $\lim_{n \rightarrow \infty} A_{n}/B_{n} = 0$. We also write $A_{n} = O(B_{n})$ if
    $\limsup_{n \rightarrow \infty} |A_{n}|/B_{n} < \infty$
}.  More precisely, Bertoin \cite{Be1} showed that under the previous regime the size $\Gamma_{n}$
of the cluster containing the root satisfies $\lim_{n \rightarrow \infty} n^{-1}\Gamma_{n} =
\Gamma(c)$ in law to some random variable $\Gamma(c) \not \equiv 0$. This includes, for instance,
important families of random trees with logarithmic height, such as random recursive trees,
preferential attachment trees, binary search trees where it is well-known that $\ell(n) = \ln n$;
see \cite{Drmota20092}, \cite[Section 4.4]{Durrett2010}. In those cases the random variable
$\Gamma(c)$ is a constant; see \cite{Be3}, \cite{Uribe2015}, \cite{Ber2015}. A different class of
example is the Cayley tree where $\ell(n) = \sqrt{n}$ and $\Gamma(c)$ is not a constant; see
\cite{Pitman1999}.

More recently, some authors have considered analyzing the fluctuations of the size of the largest percolation cluster as $n \rightarrow \infty$ for different families of trees with logarithmic height; see Schweinsberg \cite{Schweinsberg2012} and Bertoin \cite{Be2} for random recursive trees, Berzunza \cite{Ber2015} for $m$-ary random increasing trees (these include binary search trees) and preferential attachment trees.  
The motivation stems from the feature that the size of the giant cluster resulting from supercritical bond percolation on those trees has non-Gaussian fluctuations. Instead, they are described by an infinitely divisible distribution that belongs to the class of stable (asymmetric) Cauchy laws. This contrasts with analogous results on other random graphs where the asymptotic normality of the size of the giant clusters on supercritical percolation is established. We refer for instance to the works of Stepanov \cite{Ste1970}, Bollob\'{a}s and Riordan \cite{Bollo2012} and Seierstad \cite{Sei2013}.

The main purpose of this work is to investigate analogously the case of random split trees which were introduced by Devroye \cite{Luc1999}. The class of random split trees includes many families of trees that are frequently used in algorithm analysis, e.g., binary search trees \cite{Hoa1962}, $m$-ary search trees \cite{Pyke1965}, quad trees \cite{finkel1974}, median-of-$(2k+1)$ trees \cite{Walker1976}, fringe-balanced trees \cite{Luc1993}, digital search trees \cite{Coffman1970} and random simplex trees \cite[Example 5]{Luc1999}. Informally, a random split tree $T_{n}^{{\rm sp}}$ of ``size'' (or cardinality) $n$ is constructed as follows. Consider a rooted infinite $b$-ary tree with $b \in \mathbb{N}$ and where each vertex is a bucket of finite capacity $s \in \mathbb{N}$. We place $n$ balls at the root, and the balls individually trickle down the tree in a random fashion until no bucket is above capacity. Each vertex draws a split vector $\mathcal{V} = (V_{1}, \dots, V_{b})$ from a common distribution, where $V_{i}$ describes the probability that a ball passing through the vertex continues to the $i$-th child. We provide a precise description of this algorithm in Section \ref{sec1}. Finally, any vertex $u$ such that the sub-tree rooted as $u$ contains no balls is then removed, and we consider the resulting tree $T_{n}^{{\rm sp}}$. An important peculiarity of the split tree $T_{n}^{{\rm sp}}$ is that the number of vertices is random in general which makes the study of split trees usually challenging. It must also be pointed out that later we assume that $b < \infty$. However,  we believe that our approach can be applied to cases when $b= \infty$ with a little extra effort. The case $b=\infty$ includes uniform recursive trees and preferential attachment trees for which recently Janson \cite{janson2018} has shown can be viewed as special split trees with $b= \infty$. 

Loosely speaking, our main result shows that in the supercritical percolation regime the size of
the giant cluster has also non-Gaussian fluctuations where the ``size'' of $T_{n}^{{\rm sp}}$ can be
defined as either the number of vertices or the number of balls. We then show that the supercritical
regime corresponds to $1-p_{n} = c/\ln n$ with $c > 0$ fixed which agrees with the fact that split
trees belong to the family of trees with logarithmic height; see \cite{Luc1999}. Essentially, this
is Bertoin's \cite{Be1} criterion. Then, our main contribution establishes that the fluctuations of
the ``size'' (either number of vertices or balls) of the giant cluster are described by an
infinitely divisible distribution, the so-called continuous Luria-Delbr\"uck law. Finally, we show
that the approach developed in this work may be useful for studying percolation on other classes of
trees, such as for instance regular trees (see Section  \ref{sec6} below). 

We next introduce formally the family of random split trees and relevant background, which will enable
us to state our main result in Section \ref{sec5}.

\subsection{Random split trees} \label{sec1}

In this section, we introduce the split tree model with parameters $b, s, s_{0}, s_{1}, \mathcal{V}$ and $n$ introduced by Devroye \cite{Luc1999}. Some of the parameters are the branch factor $b \in \mathbb{N}$, the vertex capacity $s \in \mathbb{N}$, and the number of balls (or cardinality) $n \in \mathbb{N}$. The additional integers $s_{0}$ and $s_{1}$ are needed to describe the ball distribution process. They satisfy the inequalities
\begin{eqnarray} \label{eq3}
0 < s, \hspace*{3mm} 0 \leq s_{0} \leq s, \hspace*{3mm} 0 \leq b s_{1} \leq s + 1 - s_{0}.
\end{eqnarray}

\noindent The so-called random split
vector $\mathcal{V} = (V_{1}, \dots, V_{b})$ is a random non-negative vector with $\sum_{i=1}^{b} V_{i} = 1$ and $V_{i} \geq 0$, for $i=1, \dots, b$. 

Consider an infinite rooted $b$-ary tree $\mathbb{T}$, i.e., every vertex has $b$ children. We view each vertex of $\mathbb{T}$ as a bucket with capacity $s$ and assign to each vertex $u \in \mathbb{T}$ an independent copy $\mathcal{V}_{u} = (V_{u,1}, \dots, V_{u,b})$ of the random split vector $\mathcal{V}$. 

The split tree $T_{n}^{{\rm sp}}$ is constructed by distributing $n$ balls among the vertices of $\mathbb{T}$. For a vertex $u$, let $n_{u}$ be the number of balls stored in the sub-tree rooted at $u$. The tree
$T_{n}^{{\rm sp}}$ is then defined as the largest sub-tree of $\mathbb{T}$ such that $n_{u} > 0$ for
all $u \in T_{n}^{{\rm sp}}$. Let $u_{1}, \dots, u_{b}$ be the child vertices of $u$. Conditioning
on $n_{u}$ and $\mathcal{V}_{u}$, if $n_{u} \leq s$, then $n_{u_{i}} = 0$ for all $i=1, \dots, b$;
if $n_{u} > s$, then the cardinalities $(n_{u_{1}}, \dots, n_{u_{b}})$ of the $b$ sub-trees rooted
at $u_{1}, \dots, u_{b}$ are distributed as 
\begin{eqnarray*}
\text{Mult}(n_{u}-s_{0}-b s_{1}, V_{u,1}, \dots, V_{u,b} ) + (s_{1}, \dots, s_{1}),
\end{eqnarray*}

\noindent where $\text{Mult}$ denotes the multinomial distribution, and $b,s,s_{0}, s_{1}$ are integers satisfying (\ref{eq3}). 

It would be convenient to recall one more equivalent description of $T_{n}^{{\rm sp}}$ where one inserts data items into an initially empty data structure $\mathbb{T}$. Let $C(u)$ denote the number of balls in vertex $u$, initially setting $C(u) = 0$ for all $u$. We call $u$ a leaf if $C(u) > 0$ and $C(v) = 0$ for all children $v$ of $u$, and internal if $C(v) > 0$ for some strict descendant $v$ of $u$. Then $T_{n}^{{\rm sp}}$ is constructed recursively by distributing $n$ balls one at time to generate a subset of vertices of $\mathbb{T}$. The balls are labeled using the set $\{1, 2, \dots, n\}$ in the order of insertion. The $j$-th ball is added by the following procedure.
 \begin{enumerate}
 \item Insert $j$ to the root.
 \item While $j$ is at an internal vertex $u \in \mathbb{T}$, choose child $i$ with probability $V_{u,i}$ and move $j$ to child $i$.
 \item If $j$ is at a leaf $u$ with $C(u) < s$, then $j$ stays at $u$ and $C(u)$ increases by $1$. If $j$ is at a leaf with $C(u) = s$, then the balls at $u$ are distributed among $u$ and its children as follows. We select $s_{0} \leq s$ of the balls uniformly at random to stay at $u$. Among the remaining $s + 1 - s_{0}$ balls, we uniformly at random distribute $s_{1}$ balls to each of the $b$ children of $u$. Each of the remaining $s + 1 - s_{0} - bs_{1}$ balls is placed at a child vertex chosen independently at random according to the split vector assigned to $u$. This splitting process is repeated for any child which receives more than $s$ balls. 
\end{enumerate}
We stop once all $n$ balls have been placed in $\mathbb{T}$ and obtain $T_{n}^{{\rm sp}}$ by deleting all vertices $u \in \mathbb{T}$ such that the sub-tree rooted at $u$ contains no balls. Note that an internal vertex of $T_{n}^{{\rm sp}}$ contains exactly $s_{0}$ balls, while a leaf contains a random number in $\{1, . . . , s\}$. This description will be used in the Appendix \ref{sec3}.

\begin{remark}
Note that the number of vertices $N$ of $T_{n}^{{\rm sp}}$ is a random variable in general although the number of balls $n$ is deterministic. This is one of the
main challenges in the study of split trees. 
\end{remark}

\begin{remark} \label{remark2}
Depending on the choice of the parameters, several important data structures may be modeled. For instance, the binary search trees where $b=2$, $s = s_{0} = 1$, $s_{1} = 0$ and $\mathcal{V}$ is distributed as $(U, 1-U)$ for $U$ a random variable uniform on $[0,1]$. In this case $N=n$. Some other relevant (and
more complicated) examples of split trees are $m$-ary search trees, median-of-$(2k+1)$ trees, quad
trees, simplex trees; see for instance the original work of Devroye \cite{Luc1999} for details.
\end{remark}

\begin{remark}
    We can assume without loss of generality that the components of
    the split vector $\mathcal{V}$ are identically distributed by the random permutations explained in  \cite{Luc1999}. In particular, $\mathbb{E}[V_{1}] = 1/b$. 
\end{remark}

Two quantities deeply related to the structure of split trees are
\begin{eqnarray} \label{eq56}
\mu \coloneqq b \mathbb{E}[-V_{1} \ln V_{1}] \hspace*{5mm} \text{and} \hspace*{5mm} \sigma^{2} \coloneqq b \mathbb{E}[V_{1} \ln^{2} V_{1}] - \mu^{2}. 
\end{eqnarray}

\noindent Note that $\mu \in (0, \ln b)$ and $\sigma < \infty$. They were introduced first by Devroye \cite{Luc1999} in the
study of the height of $T_{n}^{{\rm sp}}$ as the number of balls increases.

In the study of split trees, the following condition is often assumed:

\begin{Condition} \label{Cond1}
Assume that $\mathbb{P}(V_{1} = 1) = \mathbb{P}(V_{1} = 0) = 0$.
\end{Condition}

In the present work, we use the so-called total path length of $T_{n}^{{\rm sp}}$ defined by
$\Psi(T_{n}^{{\rm sp}}) \coloneqq \sum_{i=1}^{n} D_{n}(i)$, where $D_{n}(j)$ denotes the height (or depth) of the ball labeled $j$ when all $n$ balls
have been inserted in $T_{n}^{{\rm sp}}$. Broutin and Holmgren \cite[Theorem
3.1]{Bro2012} have shown that under Condition \ref{Cond1} (and even for degenerate $V_{1}$),
\begin{eqnarray} \label{eq10}
\mathbb{E}[\Psi(T_{n}^{{\rm sp}})] =  \mu^{-1} n \ln n + \varpi(\ln n) n + o(n), 
\end{eqnarray}

\noindent where $\varpi: \mathbb{R} \rightarrow \mathbb{R}$ is a continuous periodic function of period 
\begin{eqnarray} \label{eq51}
d \coloneqq \sup \{ a \geq 0: \mathbb{P}(\ln V_{1} \in a \mathbb{Z}) = 1 \}. 
\end{eqnarray}

\noindent In particular, if the random variable $\ln V_{1}$ is non-lattice\footnote{The random
variable $\ln V_{1}$ is non-lattice when there is not $a \in \mathbb{R}$ such that $\ln V_{1} \in a
\mathbb{Z}$ almost surely.  The constant $d$ is called the span of the lattice when $d >0$ and $\ln
V_{1}$ is non-lattice when $d =0$.}, then $d = 0$ and the function $\varpi$ is a constant and we
write $\varsigma \equiv \varpi$. 

It is important to point out that the proof \cite[Theorem
3.1]{Bro2012} is missing some details for the case when $\ln V_{1}$ is lattice. The issue there is that the convergence (24) in \cite{Bro2012} only holds when the distribution of $\ln V_{1}$ is non-lattice. Nevertheless, a close look to the proof of \cite[Lemma 4.2]{Bro2012} and Lemma \ref{ExtraLemma} (ii) below show that result in \cite[Theorem
3.1]{Bro2012} is correct also in the lattice case.  

\begin{remark}
In binary search trees the function $\varpi$ equals to  $2 \gamma -4$ where $\gamma$ is the Euler's constant; see \cite{Hibb1962}. A similar result has been proven for random $m$-ary search trees \cite{Mah1986}, quad
trees \cite{Nein1999}, the random median of a $(2k + 1)$-tree \cite{Ro2001}, tries, and Patricia tries \cite{Bourdon2001}.
\end{remark}

An alternative notion of path length is the sum of all the heights of the
vertices in $T_{n}^{{\rm sp}}$, i.e., $\Upsilon(T_{n}^{{\rm sp}}) \coloneqq  \sum_{u \in T_{n}^{{\rm sp}}} d_{n}(u)$, where $d_{n}(u)$ denotes the height of the vertex $u \in T_{n}^{{\rm sp}}$. Recall that the height of a vertex is defined as the minimal number of edges of $T_{n}^{{\rm sp}}$ which are needed to connect it to the root.

\begin{Condition} \label{Cond2}
Suppose that $\ln V_{1}$ is non-lattice. Furthermore, for some $\alpha > 0$ and $\varepsilon > 0$, 
$\mathbb{E}[N] = \alpha n + O(n (\ln n)^{-1-\varepsilon})$.
\end{Condition}

\noindent Assuming that Condition \ref{Cond2} holds, Broutin and Holmgren \cite[Corollary 5.1]{Bro2012}  showed that 
\begin{eqnarray} \label{eq5}
\mathbb{E}[\Upsilon(T_{n}^{{\rm sp}})] =  \alpha \mu^{-1} n \ln n + \zeta n + o(n), \hspace*{4mm} \text{for some constant} \hspace*{3mm} \zeta \in \mathbb{R}.
\end{eqnarray}

\begin{remark}
Holmgren \cite[Theorem 1.1]{Holm2012} showed that if $\ln V_{1}$ is non-lattice, i.e., $d=0$, then
there exists a constant $\alpha >0$ such that $\mathbb{E}[N] = \alpha n + o(n)$ and furthermore
$Var(N) = o(n^{2})$. However, this result is not enough to deduce (\ref{eq5}) from (\ref{eq10}) and
the extra control in $\mathbb{E}[N]$ is needed; see \cite[Section 5.1]{Bro2012}. On the one hand,
Condition \ref{Cond2} is satisfied in many interesting cases. For instance, it holds for $m$-ary
search trees \cite{Mah1989}. Moreover, Flajolet et al.\ \cite{Flajolet2010} showed that for most
tries (where $s=1$ and $s_{0} = 0$ and as long as $\ln V_{1}$ is non-lattice) Condition \ref{Cond2}
holds. On the other hand, there are some special cases of random split trees that do not satisfy Condition \ref{Cond2}. For instance, tries with a fixed split vector $(1/b, \dots, 1/b)$, in which case $\ln V_{1}$ is lattice with $d = b$.  
\end{remark}

\begin{remark} \label{remark1}
It is important to mention that one can use Condition \ref{Cond2} to improve the result of Holmgren \cite[Theorem 1.1]{Holm2012} and obtain that $Var(N) = o (n^{2}\ln^{-2-2\varepsilon} n )$. We refer to \cite[Theorem 1.1]{Holm2012} and \cite[Remark 3.1]{Holm2012} for a proof.
\end{remark}

Finally, we recall and extend some results in \cite[Section 2]{Holm2012} and \cite[Section 4.2]{Bro2012} related to the application of renewal theory in the study of split-trees. For $k \geq 1$, set $S_{k} \coloneqq \sum_{j=1}^{k} - \ln V_{j}^{\prime}$ where $(V_{j}^{\prime}, j\geq 1)$ is a sequence of i.i.d.\ copies of $V_{1}$. Following the presentation in Holmgren \cite{Holm2012} (or  \cite[Section 4.2]{Bro2012}), for $k \geq 1$ and $t \in \mathbb{R}$, let $\vartheta_{k}(t) \coloneqq b^{k} \mathbb{P}(S_{1} \leq t)$ and define the renewal function 
\begin{eqnarray*}
U(t) =  \sum_{k=1}^{\infty} \vartheta_{k}(t).
\end{eqnarray*}

\noindent Observe that $U(t) = 0$, for $t < 0$. For $t \in \mathbb{R}$, let $\vartheta(t) = \vartheta_{1}(t)$ and observe that $U$ satisfies the following renewal equation 
\begin{eqnarray} \label{ExtraEq1}
U(t) = \vartheta(t) + (U \ast {\rm d} \vartheta)(t), \hspace*{3mm} \text{where} \hspace*{3mm}  (U \ast {\rm d} \vartheta)(t) = \int_{0}^{t}U(t-z) {\rm d} \vartheta(z),  \hspace*{4mm} \text{for} \hspace*{2mm} t \geq 0.
\end{eqnarray}

\begin{lemma} \label{ExtraLemma}
Assume that $\mathbb{P}(V_{1} = 1) = \mathbb{P}(V_{1} = 0) = 0$. The renewal function $U$ satisfies the following. 
\begin{itemize}
\item[(i)] Suppose that $\ln V_{1}$ is non-lattice. Then,
\begin{eqnarray*} 
 U(t)= \left(\frac{1}{\mu} + o(1) \right)e^{t}, \hspace*{3mm} \text{as} \hspace*{2mm} t \rightarrow \infty. 
\end{eqnarray*}

\item[(ii)] Suppose that the distribution of $\ln V_{1}$ is lattice with span $d$ defined in (\ref{eq51}). Then,
\begin{eqnarray*} 
 U(d \lfloor t \rfloor)= \left(\frac{d}{\mu} \frac{1}{1-e^{-d}}+ o(1) \right)e^{d \lfloor t \rfloor}, \hspace*{3mm} \text{as} \hspace*{2mm} t \rightarrow \infty.
\end{eqnarray*}
\end{itemize}
\end{lemma}

\begin{proof}
Part (i) follows from Holmgren \cite[Lemma 2.1]{Holm2012}. To prove part (ii), we use the lattice version of the key renewal theorem. Observe that ${\rm d} \vartheta(t)$ is not a probability measure. Following Holmgren \cite{Holm2012} (or \cite[Section 4.2]{Bro2012}), one can define another (``tilted'') measure ${\rm d} \omega(t) = e^{-t} {\rm d} \vartheta(t)$ which indeed is a probability measure. Furthermore, ${\rm d} \omega(t)$ is lattice with period $d$. The renewal equation (\ref{ExtraEq1}) can then be written as
\begin{eqnarray*}
\hat{U}(t) = \hat{\vartheta}(t) + (\hat{U} \ast {\rm d} \omega)(t), \hspace*{3mm} \text{where} \hspace*{3mm}  \hat{U}(t) = e^{-t}U(t) \hspace*{3mm} \text{and} \hspace*{3mm} \hat{\vartheta}(t)  = e^{-t} \vartheta(t),
\end{eqnarray*}

\noindent for $t \geq 0$. On the other hand, $\sum_{k=0}^{\infty} \hat{\vartheta}(kd) = (1-e^{-d})^{-1}$. Therefore, (ii) follows from \cite[Proposition 4.1, Chapter V]{asmussen2003}. 
\end{proof}

In \cite[Section 4.2]{Bro2012}, the second-order behavior of the renewal function $U$ is also studied. More precisely, \cite[Lemma 4.2]{Bro2012} establishes that under Condition \ref{Cond1} (and even for degenerate $V_{1}$) one has that
\begin{eqnarray} \label{eq33}
\int_{0}^{t} e^{-z}\left(U(z)  -\mu^{-1}e^{z} \right) {\rm d}z = \frac{\sigma^{2}-\mu^{2}}{2\mu^{2}} - \mu^{-1} + \phi(t) +  o(1), \hspace*{4mm} \text{as} \hspace*{2mm} t \rightarrow \infty,
\end{eqnarray}

\noindent  where $\phi: \mathbb{R} \rightarrow \mathbb{R}$ is a continuous $d$-periodic function with $d$ defined in (\ref{eq51}). Moreover, if $d=0$ then $\phi \equiv 0$; see Holmgren \cite[Corollary 2.2]{Holm2012}  for the non-lattice case.

\subsection{Main results} \label{sec5}

In this section, we present the main results of this work. Let $T_{n}^{{\rm sp}}$ be a split tree with $n$ balls. We then perform Bernoulli bond percolation with parameter
\begin{eqnarray} \label{eq9}
p_{n} = 1 -\frac{c}{\ln n}, 
\end{eqnarray}

\noindent where $c > 0$ is fixed. We write $\hat{G}_{n}$ for the size, i.e., the number of balls, of the percolation cluster that contains the root. Our first result shows that this choice of the percolation parameter corresponds precisely to the supercritical regime we are interested in.  

\begin{lemma} \label{lemma4}
Suppose that Condition \ref{Cond1} holds. In the regime (\ref{eq9}), we have that
\begin{eqnarray*} 
\lim_{n \rightarrow \infty} n^{-1} \hat{G}_{n}  =  e^{- \frac{c}{ \mu} }, \hspace*{5mm} \text{in
probability}.
\end{eqnarray*}

\noindent Moreover, the root cluster is the unique giant component, i.e., $\lim_{n \rightarrow \infty} n^{-1} \hat{G}_{n}^{2{\rm nd}} = 0$ in probability, where $\hat{G}_{n}^{2{\rm nd}}$ denotes the number of balls of the second largest percolation cluster. 
\end{lemma}

Alternatively, let \(G_{n}\) be the number of vertices in the root cluster. Then we have the
similar result:
\begin{lemma} \label{lemma3}
Suppose that Conditions \ref{Cond1} and \ref{Cond2} hold. In the regime (\ref{eq9}), we have that
\begin{eqnarray} \label{eq2}
\lim_{n \rightarrow \infty} n^{-1} G_{n} = \alpha e^{- \frac{c}{ \mu} }, \hspace*{5mm} \text{in probability}, 
\end{eqnarray}

\noindent where $\alpha >0$ is the constant in Condition \ref{Cond2}. Moreover, the root cluster is the unique giant component, i.e., $\lim_{n \rightarrow \infty} n^{-1} G_{n}^{2{\rm nd}} = 0$ in probability, where $G_{n}^{2{\rm nd}}$ denotes the number of vertices of the second largest percolation cluster. 
\end{lemma}

Lemma \ref{lemma4} and Lemma \ref{lemma3} are
a direct consequence of the results of Bertoin \cite{Be1} which
provides a simple
characterization of tree families and percolation regimes which yield giant clusters;  details of their proofs
are given in Section \ref{secproflemma}.

The results in Lemma \ref{lemma4} and Lemma \ref{lemma3} 
can be viewed as the law of large numbers for the  ``size'' of the giant cluster, and it is then natural to investigate the fluctuations of $\hat{G}_{n}$ and $G_{n}$. To give a precise statement, recall that a real-valued random variable $Z$ has the so-called continuous Luria-Delbr\"uck law\footnote{The name of this distribution had its origin in a series of classic experiments in evolutionary biology pionneered by Luria and Delbr\"uck \cite{luria1943} in order to study ``random mutation'' versus ``directed adaptation'' in the context of bacteria becoming resistant to a previously lethal agent. We refer also to \cite{Moohle2005}.} when its characteristic function is given by
\begin{eqnarray*}
\mathbb{E}\left[e^{itZ} \right] = \exp\left( -\frac{\pi}{2} |t| - it \ln |t| \right), \hspace*{5mm} t \in \mathbb{R}. 
\end{eqnarray*}

\noindent This distribution arises in limit theorems for sums of positive i.i.d.\ random variables
in the domain of attraction of a completely asymmetric Cauchy process; see e.g., Geluk and de Haan
\cite{Geluk2000}. In the context of percolation on large trees, it was observed first by
Schweinsberg \cite{Schweinsberg2012} (see also Bertoin \cite{Be2} for an alternative approach) in
relation with the fluctuations of the size (number of vertices) of the giant cluster for
supercritical percolation on random recursive trees. More precisely, let $T_{n}^{{\rm rec}}$ be a random recursive tree with $n$ vertices and denote by $G_{n}^{{\rm rec}}$ the size (number of vertices) of the largest percolation cluster after performing percolation 
with parameter $p_{n}$ as in (\ref{eq9}); In \cite{Be3}, it has been proven that this yields also to the supercritical regime in $T_{n}^{{\rm rec}}$, i.e., $\lim_{n \rightarrow \infty} n^{-1}G_{n}^{{\rm rec}} = e^{-c}$ in probability. Then,
\begin{eqnarray*}
\left( n^{-1} G_{n}^{{\rm rec}} -e^{-c} \right) \ln n - ce^{-c} \ln \ln n \xrightarrow[ ]{d} -ce^{-c}(Z + \ln c),
\end{eqnarray*} 

\noindent where $\xrightarrow[ ]{d}$ means convergence in distribution as $n \rightarrow \infty$. More recently, Berzunza \cite{Ber2015} has shown for preferential attachment trees and $m$-ary random increasing trees (the latter includes the case of binary search trees) that the fluctuations of the size of the giant component in the percolation regime (\ref{eq9}) are also described by the continuous Luria-Delbr\"uck distribution.

On the other hand, the continuous Luria-Delbr\"uck distribution has been further observed in several weak limit theorems for the number of cuts required to isolate the root of a tree; see the original work of Meir and Moon \cite{Meir1970}. For random recursive tree (Drmota et al.\ \cite{Drmota2009}, Iksanov and Möhle \cite{Iksanov2007}), random binary search tree (Holmgren \cite{Holmgren2010}) and split trees (Holmgren \cite{Holm2011}). We refer to \cite{cai20192} and \cite{cai2018} for a generalization of the Meir and Moon cutting model where similar results appears.  

We now state the central results of this work. 

\begin{theorem} \label{Theo2}
Suppose that Condition \ref{Cond1} holds and that $\ln V_{1}$ is non-lattice. As $n \rightarrow \infty$, there is the convergence in distribution 
\begin{eqnarray*}
 \left(\frac{\hat{G}_{n}}{n} -e^{-\frac{c}{\mu}}  \right)\ln n - \frac{c}{\mu}e^{-\frac{c}{\mu}} \ln \ln n   \xrightarrow[ ]{d} - \frac{c}{\mu} e^{- \frac{c}{\mu}} \left( Z + \ln \left( \frac{c}{\mu} \right) + \varsigma \mu  + \frac{(\mu^{2}- \sigma^{2})(c+\mu)}{2\mu^{2}}  - \gamma + 1 \right),
\end{eqnarray*}

\noindent where $\mu$ and $\sigma^{2}$ are the constants defined in (\ref{eq56}), $\varpi \equiv \varsigma$ (a constant) is defined in (\ref{eq10}), $\gamma$ is the Euler constant and the variable $Z$ has the continuous Luria-Delbr\"uck distribution. 
\end{theorem}

Similarly, we obtain that the fluctuations of $G_{n}$ are also described by $Z$. 

\begin{theorem} \label{Theo1}
Suppose that Condition \ref{Cond1} and \ref{Cond2} hold. As $n \rightarrow \infty$, there is the convergence in distribution
\begin{eqnarray*}
 \left(\frac{G_{n}}{n} - \alpha e^{-\frac{c}{\mu}}  \right)\ln n - \frac{c\alpha}{\mu}e^{-\frac{c}{\mu}} \ln \ln n \xrightarrow[ ]{d} - \frac{c \alpha}{\mu} e^{- \frac{c}{\mu}} \left( Z + \ln \left( \frac{c}{\mu} \right) + \frac{ \zeta  \mu}{\alpha} + \frac{(\mu^{2}- \sigma^{2})(c+\mu)}{2\mu^{2}}- \gamma + 1 \right),
\end{eqnarray*}

\noindent where $\mu$ and $\sigma^{2}$ are the constants defined in (\ref{eq56}), $\alpha$ is defined in Condition {\ref{Cond2}}, $\zeta$ is defined in (\ref{eq5}), $\gamma$ is the Euler constant and the variable $Z$ has the continuous Luria-Delbr\"uck distribution. 
\end{theorem}

We also show that Theorem \ref{Theo2} can essentially be extended to the case when $\ln V_{1}$ is lattice. More precisely, we consider the following additional condition. Write $y = \lfloor y \rfloor + \{ y \}$ for the decomposition of a real number $y$ as the sum of its integer and fractional parts. 

\begin{Condition} \label{NewCond1}
Let $T_{n}^{\rm sp}$ be a split tree with cardinality $n$ and span $d >0$ defined in (\ref{eq51}). Furthermore, suppose that $\{d^{-1} \ln \ln n\} \rightarrow \varrho \in [0,1)$, as $n \rightarrow \infty$. 
\end{Condition}

We  introduce for every $\varrho \in [0,1)$ and $c, d, x > 0$,
\begin{eqnarray*}
\bar{\Xi}_{\varrho}^{c,d}(x) = \frac{c}{\mu}\frac{d}{1-e^{-d}} e^{d \lfloor \varrho -d^{-1}\ln x - d^{-1}c/\mu \rfloor - d \varrho}, 
\end{eqnarray*} 

\noindent where $\mu$ is the constant defined in (\ref{eq56}).
The function $\bar{\Xi}_{\varrho}^{c,d}$ decreases as $x \rightarrow \infty$ and it can be viewed as the tail of a measure $\Xi_{\varrho}^{c,d}$ on $(0, \infty)$. It is not difficult to see that this measure fulfills the integral condition $\int_{(0, \infty)} (1 \wedge x^{2}) \Xi_{\varrho}^{c,d}({\rm d} x) < \infty$. This enables us to introduce a L\'evy process without negative jumps $Z_{\varrho}^{c,d} = (Z_{\varrho}^{c,d}(t))_{t \geq 0}$ with Laplace exponent
\begin{eqnarray*}
\Phi_{\varrho}^{c,d}(a) = \int_{(0, \infty)} (e^{-ax} - 1 + ax \mathds{1}_{\{x < 1 \}}) \Xi_{\varrho}^{c,d}({\rm d} x),
\end{eqnarray*} 

\noindent i.e., $\mathbb{E}[e^{-aZ_{\varrho}^{c,d}(t)}] = e^{t \Phi_{\varrho}^{c,d}(a)}$, for  $a \geq 0$. 

\begin{theorem} \label{NewTheo2}
Suppose that Condition \ref{Cond1} holds and that $T_{n}^{\rm sp}$ satisfies Condition \ref{NewCond1}. For any constant $\theta > 0$,  as $n \rightarrow \infty$, there is the convergence in distribution 
\begin{align*}
&  \left(\frac{\hat{G}_{n}}{n} -e^{-\frac{c}{\mu}}  \right)\ln n - \frac{c}{\mu}e^{-\frac{c}{\mu}} \ln \ln n + ce^{-\frac{c}{\mu}} \left( \varpi(\ln n) - \phi \left(\ln \left( \theta^{-1} e^{-\frac{c}{\mu}} \ln n \right) \right) \right) \\
& \hspace*{15mm} \xrightarrow[ ]{d} -Z_{\varrho}^{c,d}(1)  - \frac{c}{\mu} e^{- \frac{c}{\mu}} \left(  \frac{c}{\mu}  + \frac{(\mu^{2}- \sigma^{2})(c+\mu)}{2\mu^{2}} \right),
\end{align*}

\noindent where $\mu$ and $\sigma^{2}$ are the constants defined in (\ref{eq56}), $\varpi$ is the function defined in (\ref{eq10}), $\phi$ is the function defined in (\ref{eq33}), $\varrho$ is defined in Condition \ref{NewCond1} and $\gamma$ is the Euler constant.
\end{theorem}

\begin{remark}
Following Bertoin \cite{Be1}, we point out that Lemmas \ref{lemma4} and \ref{lemma3} still hold whenever the percolation parameter satisfies $p_{n} = 1 -c \ln^{-1} n+ o ( \ln^{-1} n)$, where $c > 0$ is fixed, which still falls in the supercritical regime; see \cite[Theorem 1]{Be1}. However, to obtain similar results to those in Theorems \ref{Theo2}, \ref{Theo1} and \ref{NewTheo2} one needs to know more information of the $o( \ln^{-1} n)$ term.
\end{remark}

It is important to remark that the constants appearing in our main results can be computed explicitly for some types of trees. For example, if $T_{n}^{{\rm bst}}$ is a binary search tree with $n$ vertices (recall Remark \ref{remark2}), then $N =n$, $\alpha =1$, $\mu = 1/2$, $\sigma^{2} = 1/4$,  $\zeta = \varsigma = 2 \gamma -4$ and $\phi \equiv 0$; see for example \cite{Hibb1962}. Moreover, the result in Theorem \ref{Theo2} (or Theorem \ref{Theo1}) applied to $T_{n}^{{\rm bst}}$ coincides with Berzunza \cite[Theorem 1.1]{Ber2015}. The value of the constant can also be computed, for instance, for quad trees or for $m$-ary search trees; we refer to \cite{Nein1999} and \cite{Mah1986}, respectively, for details.

The approach used by Schweinsberg \cite{Schweinsberg2012} for recursive trees relies on its
connection with the Bolthausen-Sznitman coalescent founded by Goldschmidt and Martin \cite{Golds2005}
and the estimation of the rate of decrease of the number of blocks in such coalescent process. The
alternative approach of Bertoin \cite{Be2} makes use of the special properties of recursive trees
(namely the splitting property) and more specifically of a coupling due to Iksanov and Möhle
\cite{Iksanov2007} connecting the Meir and Moon \cite{Meir1970} algorithm for the isolation of the
root with a certain random walk in the domain of attraction of the completely asymmetric Cauchy
process. This clearly fails for split-trees. On the other hand, the basic idea of Berzunza
\cite{Ber2015} for the case of $m$-ary random increasing trees and preferential attachment trees is
based in the close relation of these trees with Markovian branching processes and the dynamical
incorporation of percolation as neutral mutations. Roughly speaking, this yields to the analysis of
the asymptotic behavior of branching processes subject to rare neutral mutations. The relationship
between percolation on trees and branching process with mutations was first observed by Bertoin and
Uribe Bravo \cite{Uribe2015}. Recently, Holmgren and Janson \cite{HoJa2017} have shown that some
kinds of split trees (but not all) can be related to genealogical trees of general age-dependent
branching processes (or Crump-Mode-Jagers processes), for instance, $m$-ary search trees and
median-of-$(2\ell + 1)$ trees. Furthermore, Berzunza \cite{Berzunza2018} has proven the existence of
a giant percolation cluster for appropriate regimes of such genealogical trees via a similar
relationship with a general branching process with mutations. However, the branching processes with
mutations in \cite{Berzunza2018} is in general not Markovian due to the nature of the
Crump-Mode-Jagers processes; see \cite{Ja1975}. This makes the idea of \cite{Ber2015} difficult to
implement since there the Markov property is crucial. We thus have to use here a fairly different
route.

The method used here is inspired in the original technique developed by Janson \cite{Jans2004} to study the number of cuts needed to isolate the root of complete binary trees with the cutting-down procedure of Meir and Moon \cite{Meir1970}. Holmgren \cite{Holmgren2010, Holm2011} has successfully extended this method to study the same quantity as in \cite{Jans2004} for split trees. Informally speaking, we approximate $\hat{G}_{n}$ (resp. $G_{n}$) by the sum of the ``sizes'' of the percolation clusters of the sub-trees rooted at vertices that are at a distance around $ \ln \ln n$ from the root. There are approximately $b^{\ln \ln n}$ clusters, but we only consider those that are still connected to the root of $T_{n}^{{\rm sp}}$ after performing percolation for the regime $p_{n}$ as in (\ref{eq9}). The number of balls (or number of vertices) between the root of $T_{n}^{{\rm sp}}$ and the the vertices at height $\ln \ln n$ is equal to $O(\ln n)$ and thus they do not contribute to the fluctuations of $\hat{G}_{n}$ (resp. $G_{n}$). We then analyze carefully the ``sizes'' of percolation clusters at distances close to $ \ln \ln n$ from the root, and essentially, we view $\hat{G}_{n}$ (resp. $G_{n}$) as a sum of independent random variables. This will allow us to apply a classical limit theorem for the convergence of triangular arrays to get our main result. Therefore, we conclude that most of the random fluctuations can be explained by the ``sizes'' of percolation clusters at distances close to $\ln \ln n$ from the root of $T_{n}^{{\rm sp}}$ and that they are still connected to the root. It is important to point out, as well as an inspiration, that this phenomenon has been observed by Bertoin \cite[Section 3]{Be2} in a similar setting where he studied the fluctuations of the number of vertices at height $\ln \ln n$ which has been disconnected from the root in $b$-regular trees after performing supercritical percolation. The fluctuations in this setting are described by a L\'evy process without negative jumps that also appears in \cite{Jans2004}.

The rest of this paper is organized as follows: 
We start by proving Lemma
\ref{lemma4} and Lemma \ref{lemma3} in Section \ref{secproflemma}. In Section \ref{sec2}, we then focus on the proof of Theorem \ref{Theo2} and Theorem \ref{NewTheo2}. Section \ref{sec4} is devoted to the proof
of Theorem \ref{Theo1} which follows essentially from Theorem \ref{Theo2}. In Section
\ref{sec6}, we briefly point out that the present approach also applies to study the fluctuations of
the size of the giant cluster for percolation on regular trees. The appendices provide details on 
some technical results that are used in the proofs of the main result but that we decided to postpone for a better understanding of our approach. In particular, Appendix \ref{sec3} is dedicated to investigate the asymptotic behavior of distances between uniformly chosen vertices and uniformly chosen balls in $T_{n}^{{\rm sp}}$ which may be of independent interest.

\section{Proof of Lemma \ref{lemma4} and Lemma \ref{lemma3}} \label{secproflemma}

Lemma \ref{lemma4} and Lemma \ref{lemma3} are a merely consequence of the results of Bertoin \cite{Be1} after mild modifications.

\begin{proof}[Proof of Lemma \ref{lemma4}]
The result follows from exactly the same argument as the proof of \cite[Corollary 1 and Proposition 1]{Be1} by using Lemma \ref{lemma5}, Corollary \ref{cor2} in Appendix \ref{sec3} and by taking into account that the size is defined as the number of balls instead of the number of vertices.
\end{proof} 

\begin{proof}[Proof of Lemma \ref{lemma3}]
The result follows from a simple application of \cite[Corollary 1]{Be1}. Note that conditions ($\mathbf{H}_{k}$) and ($\mathbf{H}_{k}^{\prime}$), for $k=1,2$, in \cite[Corollary 1]{Be1} are verified in Lemma \ref{lemma1} and Corollary \ref{cor2} in Appendix \ref{sec3}. Therefore, in the percolation regime (\ref{eq9}), we have that $\lim_{n \rightarrow \infty} N^{-1} G_{n} =  e^{-\frac{c}{\mu} }$,  in probability. On the other hand, Conditions \ref{Cond1} and \ref{Cond2} imply that $\lim_{n \rightarrow \infty} N/n = \alpha$, in probability. This establishes (\ref{eq2}) in Lemma \ref{lemma3}. The uniqueness of the giant component follows from \cite[Proposition 1]{Be1} by noticing that the condition there is satisfied as a consequence of Lemma \ref{lemma1} and Corollary \ref{cor2} in Appendix \ref{sec3}, that is, we have the joint convergence
\begin{eqnarray*}
\lim_{n \rightarrow \infty} \frac{1}{\ln n} (d_{n}(u_{1}), d_{n}(u_{1}, u_{2}))=\left(1/\mu,
    2/\mu \right), \hspace*{5mm} \text{in probability},
\end{eqnarray*} 

\noindent where $u_{1}, u_{2}$ are two i.i.d.\ uniform random vertices in $T_{n}^{{\rm sp}}$,  $d_{n}(u_{1})$ denotes the height of $u_{1}$ and $d_{n}(u_{1}, u_{2})$ is the number of edges of $T_{n}^{{\rm sp}}$ which are needed to connect the root and the vertices $u_{1}$ and $u_{2}$.
\end{proof}

\section{Proof of Theorem \ref{Theo2}} \label{sec2}

This section is devoted to the proofs of Theorem \ref{Theo2} and Theorem \ref{NewTheo2} along the lines explained at the end of Section \ref{sec5}. The starting point is Lemma \ref{lemma6} where we estimate the number of balls of the percolation clusters of sub-trees rooted at vertices that are around height $\ln \ln n$. We continue with Lemmas \ref{lemma8}, \ref{lemma9} and \ref{lemma10} that allow us to approximate $\hat{G}_{n}$ as essentially a sum of independent random variables. Finally, we establish Theorem \ref{Theo3} that shows that the conditions of \cite[Theorem 15.28]{Kall2002}, a classical limit theorem for triangular arrays, are fulfilled which allow us to conclude with the proof of Theorem \ref{Theo2}.

For a vertex $v \in T_{n}^{{\rm sp}}$ that is at height $d_{n}(v) = j$, it is not difficult to see
from the definition of random split trees in Section \ref{sec1} that conditioning on the split
vectors, we have 
\begin{eqnarray} \label{eq19}
    \text{binomial}\left(n, \prod_{k=1}^{j} W_{v,k}  \right) - sj \leq_{\text{st}} n_{v}
    \leq_{\text{st}}
\text{binomial}\left(n, \prod_{k=1}^{j} W_{v,k} \right) + s_{1}j,
\end{eqnarray}

\noindent where \(\le_{\text{st}}\) denotes \emph{stochastically dominated by and}
$(W_{v,k}, k=1, \dots, j)$ are i.i.d.\ random variables on $[0,1]$ given by the split vectors
associated with the vertices in the unique path from $v$ to the root; This property has been used in
\cite{Luc1999} and \cite{Holm2012}. In particular $W_{v,k} = V_{1}$ in distribution. We deduce the
following important estimates.
\begin{eqnarray} \label{eq17}
E[n_{v}] \leq n  \prod_{k=1}^{j} \mathbb{E}[W_{v,k}] + s_{1}j = n b^{-j} + s_{1}j, 
\end{eqnarray}

\noindent where we have used $\mathbb{E}[W_{v,k}] = \mathbb{E}[V_{1}] = 1/b$. Moreover, 
\begin{eqnarray} \label{eq7}
E[n_{v}^{2}]  \leq  n^{2}  \prod_{k=1}^{j} \mathbb{E}[W_{v,k}^{2}] + n\left(\prod_{k=1}^{j} \mathbb{E}[W_{v,k}] -  \prod_{k=1}^{j} \mathbb{E}[W_{v,k}^{2}] \right) + 2s_{1}jn \prod_{k=1}^{j} \mathbb{E}[W_{v,k}] + s_{1}^{2}j^{2}. 
\end{eqnarray}

\noindent Notice that $\mathbb{E}[W_{v,k}^{2}] = \mathbb{E}[V_{1}^{2}] < 1/b$.

We use the notation $\log_{b}x = \ln x/\ln b$ for the logarithm with base $b$ of $x > 0$, and we
write $m_{n} =  \lfloor \beta \log_{b} \ln n \rfloor$ for some constant $\beta >
-2/(1+\log_{b}\mathbb{E}[V_{1}^{2}])$. We further assume that $n$ is large enough such that $0 <
m_{n} < \ln n$. For $1 \leq i \leq b^{m_{n}}$, let $v_{i}$ be a vertex in $T_{n}^{{\rm sp}}$ at height
$m_{n}$ and let $n_{i}$ be the number of balls stored at the sub-tree rooted at $v_{i}$. In
particular,
\begin{eqnarray} \label{eq41}
E[n_{i}^{2}]  = n^{2} \mathbb{E}^{m_{n}}[V_{1}^{2}] +o 
( n^{2} \ln^{-k}n ) ,
\end{eqnarray}

\noindent for an arbitrary $k \geq 0$.

We denote by $\hat{C}_{n,i}$ the number of balls of the sub-tree of $T_{n}^{{\rm sp}}$ rooted at
$v_{i}$ after Bernoulli bond-percolation with parameter $p_{n}$. Clearly, $(\hat{C}_{n,i}, 1 \leq i \leq b^{m_{n}})$ are conditionally independent random variables given $(n_{i}, 1 \leq i \leq b^{m_{n}})$. We write $\mathbb{E}_{n_{i}}[\hat{C}_{n,i}] \coloneqq\mathbb{E}[\hat{C}_{n,i} | n_{i}]$, i.e., it is the conditional expected value of $\hat{C}_{n,i}$ given $n_{i}$. 

In the sequel, we shall often use the following notation $A_{n} = B_{n}+ o_{\rm p}(f(n))$, where $A_{n}$ and $B_{n}$ are two sequences of real random variables and $f: \mathbb{N} \rightarrow (0, \infty)$ a function, to indicate that $\lim_{n \rightarrow \infty} |A_{n} - B_{n}|/f(n) = 0$ in probability.  

\begin{lemma} \label{lemma6}
Suppose that Condition \ref{Cond1} is fulfilled. For $1 \leq i \leq b^{m_{n}}$, we have that
\begin{eqnarray*}
\mathbb{E}_{n_{i}}[\hat{C}_{n,i} ] = n_{i}e^{-\frac{c}{\mu}\frac{\ln n_{i}}{\ln n} } -  \frac{c^{2}\mu^{2} - c^{2} \sigma^{2}}{2\mu^{3}} \frac{n_{i} \ln n_{i}}{\ln^{2} n} e^{-\frac{c}{\mu}\frac{\ln n_{i}}{\ln n} } -  c  \frac{ n_{i} \varpi(\ln n_{i})}{\ln n}  e^{- \frac{c}{\mu}\frac{\ln n_{i}}{\ln n} } +  o\left(\frac{n_{i}}{\ln n} \right),
\end{eqnarray*}

\noindent where $\varpi:\mathbb{R} \rightarrow \mathbb{R}$ is the function in (\ref{eq10}). 
\end{lemma}

\begin{proof}
For $1 \leq i \leq b^{m_{n}}$, let $T_{i}$ be the sub-tree of $T_{n}^{{\rm sp}}$ rooted at the vertex $v_{i}$ at height $m_{n}$. Let $b_{i}$ be
an uniformly chosen ball in $T_{i}$. Let $D_{n_{i}}(b_{i})$ be the height of $b_{i}$ in the sub-tree $T_{i}$. We have the following key observation made by Bertoin \cite[Proof of Theorem 1]{Be1},
\begin{eqnarray} \label{eq13}
\mathbb{E}_{n_{i}}\left[ n_{i}^{-1} \hat{C}_{n,i} \right] = \mathbb{E}_{n_{i}}\left[ p_{n}^{D_{n_{i}}(b_{i})} \right].
\end{eqnarray}

\noindent In words, the left-hand side can be interpreted as the probability that $b_{i}$ belongs to
the percolation cluster containing the root of $T_{i}$, i.e., $v_{i}$, while the right-hand side can be interpreted as the probability that no edge has been removed in the path between $b_{i}$ and $v_{i}$.

We assume for a while that
\begin{align} \label{eq14}
& \mathbb{E}_{n_{i}}\left[ p_{n}^{D_{n_{i}}(b_{i})} \right] \nonumber \\
& \hspace*{5mm} = \mathbb{E}_{n_{i}}\left[  p_{n}^{\frac{ \ln n_{i}}{\mu}}\left(1 +  \left(D_{n_{i}}(b_{i}) - \frac{ \ln n_{i}}{\mu} \right)\ln p_{n} + \frac{1}{2}\left(D_{n_{i}}(b_{i}) - \frac{ \ln n_{i}}{\mu} \right)^{2}\ln^{2} p_{n}\right) \right] + o\left(\frac{1}{\ln n} \right).
\end{align} 

\noindent We next note from our assumption (\ref{eq9}) in the percolation parameter that
\begin{eqnarray} \label{eq16}
\ln p_{n} = -\frac{c}{\ln n} +o\left(\frac{1}{\ln n} \right) \hspace*{5mm} \text{and} \hspace*{5mm}
p_{n}^{\frac{ \ln n_{i}}{\mu}} = e^{-\frac{c}{\mu}\frac{\ln n_{i}}{\ln n} } - \frac{c^{2}}{2\mu}
\frac{\ln n_{i}}{\ln^{2} n} e^{-\frac{c}{\mu}\frac{\ln n_{i}}{\ln n} }+  o\left(\frac{1}{\ln n}
\right).
\end{eqnarray}

\noindent We have used that $\ln n_{i} \leq \ln n$. Then it follows from Lemma \ref{lemma5} (i)-(ii)
in Appendix \ref{sec3} and a couple of lines of calculations that
\begin{eqnarray*}
\mathbb{E}_{n_{i}}\left[ p_{n}^{D_{n_{i}}(b_{i})} \right] = e^{-\frac{c}{\mu}\frac{\ln n_{i}}{\ln n} } - \frac{c^{2} \mu^{2} - c^{2} \sigma^{2}}{2\mu^{3}} \frac{\ln n_{i}}{\ln^{2} n} e^{-\frac{c}{\mu}\frac{\ln n_{i}}{\ln n} } -  c   \frac{\varpi(\ln n_{i})}{\ln n} e^{-\frac{c}{\mu}\frac{\ln n_{i}}{\ln n} }  + o\left(\frac{1}{\ln n} \right).
\end{eqnarray*} 

\noindent Therefore, the result in Lemma \ref{lemma6} follows from the identity  (\ref{eq13}) and the above estimation.

Now, we focus on establishing (\ref{eq14}). From the inequality 
\begin{align*}
& \left| p_{n}^{D_{n_{i}}(b_{i})} - p_{n}^{\frac{ \ln n_{i}}{\mu}}\left(1 +  \left(D_{n_{i}}(b_{i}) - \frac{ \ln n_{i}}{\mu} \right)\ln p_{n} + \frac{1}{2}\left(D_{n_{i}}(b_{i}) - \frac{ \ln n_{i}}{\mu} \right)^{2}\ln^{2} p_{n}\right) \right| \\
& \hspace*{100mm} \leq \left|\left( D_{n_{i}}(b_{i}) - \frac{ \ln n_{i}}{\mu} \right) \ln p_{n} \right|^{3}, 
\end{align*}

\noindent we conclude that it is enough to show that 
\begin{eqnarray*}
\mathbb{E}_{n_{i}} \left[ \left|\left( D_{n_{i}}(b_{i}) - \frac{ \ln n_{i}}{\mu} \right) \ln p_{n} \right|^{3}\right] = o\left(\frac{1}{\ln n} \right)
\end{eqnarray*}

\noindent in order to obtain (\ref{eq14}). But this follows from Lemma \ref{lemma5} (iii) in Appendix \ref{sec3} and (\ref{eq16}).
\end{proof}

Let $\eta_{n,i}$ be the total number of edges on the branch from $v_{i}$ to the root which have been deleted after percolation with parameter $p_{n}$. Notice that the random variable $\eta_{n,i}$ has the binomial distribution with parameters $(m_{n}, 1-p_{n})$. But the random variables $(\eta_{n,i}, 1 \leq i \leq b^{m_{n}})$ are not independent. On the other hand, we remark that $\eta_{n,i} = 0$ if and only if the vertex $v_{i}$ is still connected to the root. 

\begin{lemma} \label{lemma8}
Suppose that Condition \ref{Cond1} is fulfilled. We have for $\beta > -2/(1+\log_{b}\mathbb{E}[V_{1}^{2}])$ that
\begin{eqnarray*}
\hat{G}_{n} =   \sum_{i=1}^{b^{m_{n}}} \mathbb{E}_{n_{i}}[\hat{C}_{n,i}] \mathds{1}_{\left \{ \eta_{n,i} = 0 \right \}}  + o_{\rm p}\left(\frac{n}{ \ln n}\right).
\end{eqnarray*}
\end{lemma}

\begin{proof}
We denote by $\hat{C}_{n,0}$ the number of balls in the vertices of $T_{n}^{{\rm sp}}$ at height less or equal to $m_{n}-1$ that are connected to
the root after percolation with parameter $p_{n}$. Then, it should be plain that
\begin{eqnarray*}
\hat{G}_{n} = \hat{C}_{n,0} + \sum_{i=1}^{b^{m_{n}}} \hat{C}_{n,i} \mathds{1}_{\left \{ \eta_{n,i} = 0 \right \}}.
\end{eqnarray*}

\noindent We observe that the sequences of random variables $(\eta_{n,i}, 1 \leq i \leq b^{m_{n}})$ and $(\hat{C}_{n,i}, 1 \leq i \leq b^{m_{n}})$ are independent. Furthermore, the sequence of random variables $(\eta_{n,i}, 1 \leq i \leq b^{m_{n}})$ and $(n_{i}, 1 \leq i \leq b^{m_{n}})$ are also independent. Let $\mathcal{F}_{n}$ be the $\sigma$-field generated by $(\eta_{n,i}, 1 \leq i \leq b^{m_{n}})$ and $(n_{i}, 1 \leq i \leq b^{m_{n}})$. We also note that $\mathbb{E}[\hat{C}_{n,i} | \mathcal{F}_{n}] = \mathbb{E}_{n_{i}}[\hat{C}_{n,i}]$. By conditioning on the $\sigma$-field $\mathcal{F}_{n}$ and taking expectation, we obtain that
\begin{eqnarray*}
\mathbb{E} \left[ \left( \hat{G}_{n} - \hat{C}_{n,0} - \sum_{i=1}^{b^{m_{n}}} \mathbb{E}_{n_{i}}[\hat{C}_{n,i} ] \mathds{1}_{\left \{ \eta_{n,i} = 0 \right \}} \right)^{2} \right] & = & \mathbb{E}\left[ \sum_{i=1}^{b^{m_{n}}} \mathbb{E}_{n_{i}} \left[ \left(  \hat{C}_{n,i} - \mathbb{E}_{n_{i}}[\hat{C}_{n,i} ] \right)^{2}\right] \mathds{1}_{\left \{ \eta_{n,i} = 0 \right \}}  \right] \\
& = &   \sum_{i=1}^{b^{m_{n}}} \mathbb{E} \left[ \left(  \hat{C}_{n,i} - \mathbb{E}_{n_{i}}[\hat{C}_{n,i} ] \right)^{2}\right] \mathbb{P}\left ( \eta_{n,i} = 0 \right ). 
\end{eqnarray*}

\noindent Since $\mathbb{P}( \eta_{n,i} = 0) \leq 1$ and $\mathbb{E} [ (  \hat{C}_{n,i} - \mathbb{E}_{n_{i}}[\hat{C}_{n,i} ] )^{2} ] \leq 2\mathbb{E}[n_{i}^{2}]$, beacuse $\hat{C}_{n,i} \leq n_{i}$, we deduce that
\begin{eqnarray*}
\mathbb{E} \left[ \left( \hat{G}_{n} - \hat{C}_{n,0} - \sum_{i=1}^{b^{m_{n}}} \mathbb{E}_{n_{i}}[\hat{C}_{n,i} ] \mathds{1}_{\left \{ \eta_{n,i} = 0 \right \}} \right)^{2} \right] \leq 2 \sum_{i=1}^{b^{m_{n}}} \mathbb{E}[n_{i}^{2}].
\end{eqnarray*}

\noindent Since $\beta > -2/(1+\log_{b}\mathbb{E}[V_{1}^{2}])$, we obtain from the estimate (\ref{eq41}) that 
\begin{eqnarray*}
\mathbb{E} \left[ \left( \hat{G}_{n} - \hat{C}_{n,0} - \sum_{i=1}^{b^{m_{n}}} \mathbb{E}_{n_{i}}[\hat{C}_{n,i} ] \mathds{1}_{\left \{ \eta_{n,i} = 0 \right \}} \right)^{2} \right] = o\left(\frac{n^{2}}{ \ln^{2} n}\right).
\end{eqnarray*}
 
\noindent The above implies together with Chebyshev's inequality that
\begin{eqnarray*}
\hat{G}_{n} = \hat{C}_{n,0} + \sum_{i=1}^{b^{m_{n}}} \mathbb{E}_{n_{i}}[\hat{C}_{n,i} ] \mathds{1}_{\left \{ \eta_{n,i} = 0 \right \}} + o_{\rm p}\left(\frac{n}{ \ln n}\right).
\end{eqnarray*}

\noindent Finally, the statement follows easily after noticing that $0 \leq \hat{C}_{n,0} < b^{m_{n}+1} = o\left(\frac{n}{ \ln n}\right)$.
\end{proof}

Next, we combine Lemma \ref{lemma6} and \ref{lemma8}.

\begin{lemma} \label{lemma9}
Suppose that Condition \ref{Cond1} is fulfilled.  We have for $\beta > -2/(1+\log_{b}\mathbb{E}[V_{1}^{2}])$ that
\begin{equation*} 
    \begin{aligned}
    \hat{G}_{n} =
    & -  e^{- \frac{c}{\mu}  } \sum_{i=1}^{b^{m_{n}}} n_{i} \mathds{1}_{\left \{ \eta_{n,i} \geq 1
        \right \}}  + \sum_{i=1}^{b^{m_{n}}} n_{i} e^{-\frac{c}{\mu}\frac{\ln n_{i}}{\ln n} }  \\
    & -  ce^{-\frac{c}{\mu} }  \sum_{i=1}^{b^{m_{n}}} \frac{n_{i} \varpi(\ln n_{i})}{\ln n} -    \frac{c^{2}\mu^{2}-c^{2}\sigma^{2}}{2 \mu^{3}} e^{- \frac{c}{\mu}  }  \frac{n}{\ln n}    + o_{\rm p}\left(\frac{n}{\ln n} \right).
    \end{aligned}
\end{equation*}

\noindent where $\varpi:\mathbb{R} \rightarrow \mathbb{R}$ is the function in (\ref{eq10}). 
\end{lemma}

\begin{proof}
We remark that the two sequences of random variables $(\eta_{n,i}, 1 \leq i \leq b^{m_{n}})$ and $(n_{i}, 1 \leq i \leq b^{m_{n}})$ are independent. Recall that the random variable $\eta_{n,i}$ has the binomial distribution with parameters $(m_{n}, 1-p_{n})$. Hence 
\begin{eqnarray} \label{eq29}
1 - \mathbb{P}\left(\eta_{n,i} = 0 \right ) = \mathbb{P}\left(\eta_{n,i} \geq 1 \right ) = 1-p_{n}^{m_{n}} = O \left(\frac{\ln \ln n}{\ln n} \right).
\end{eqnarray}

\noindent Since $\sum_{i=1}^{b^{m_{n}}} n_{i} \leq n$ and $\mathbb{P}( \eta_{n,i} = 0) \leq 1$, we obtain that
\begin{eqnarray*}
\mathbb{E}\left[ \sum_{i=1}^{b^{m_{n}}} \frac{n_{i}}{\ln n} \mathds{1}_{\left \{ \eta_{n,i} = 0 \right \}}  \right] = \frac{1}{\ln n}\sum_{i=1}^{b^{m_{n}}}  \mathbb{E}[ n_{i} ] \mathbb{P}\left(\eta_{n,i} = 0 \right ) \leq \frac{n}{\ln n}.
\end{eqnarray*}

\noindent Thus Lemma \ref{lemma6} and Lemma \ref{lemma8} imply that
\begin{eqnarray} \label{eq27}
\hat{G}_{n}  = \sum_{i=1}^{b^{m_{n}}} \left( n_{i}-  \frac{c^{2} \mu^{2}- c^{2}\sigma^{2}}{2 \mu^{3}}  \frac{n_{i}\ln n_{i}}{\ln^{2} n} -  c \frac{ n_{i} \varpi(\ln n_{i})}{\ln n}   \right) e^{-\frac{c}{\mu}\frac{\ln n_{i}}{\ln n} }  \mathds{1}_{\left \{ \eta_{n,i} = 0 \right \}}  + o_{\rm p}\left(\frac{n}{\ln n} \right).
\end{eqnarray}

\noindent By the estimation (\ref{eq29}) and the fact that $\sum_{i=1}^{b^{m_{n}}} n_{i} \leq n$, we get that
\begin{eqnarray*}
\mathbb{E}\left[ \left| \sum_{i=1}^{b^{m_{n}}} \frac{n_{i}\ln n_{i}}{\ln^{2} n} e^{- \frac{c}{\mu}\frac{\ln n_{i}}{\ln n} } \mathds{1}_{\left \{ \eta_{n,i} = 0 \right \}} - \sum_{i=1}^{b^{m_{n}}} \frac{n_{i} \ln n_{i}}{\ln^{2} n} e^{- \frac{c}{\mu}\frac{\ln n_{i}}{\ln n} } \right| \right]  \leq  \frac{1}{\ln n}\sum_{i=1}^{b^{m_{n}}}  \mathbb{E}[ n_{i} ] \mathbb{P}\left(\eta_{n,i} \geq 1 \right ) = o\left(\frac{n}{\ln n} \right)
\end{eqnarray*}

\noindent and
\begin{equation*}
    \begin{aligned}
        & \mathbb{E}\left[ \left| \sum_{i=1}^{b^{m_{n}}} \frac{n_{i} \varpi(\ln n_{i})}{\ln n} e^{-
                    \frac{c}{\mu}\frac{\ln n_{i}}{\ln n} } \mathds{1}_{\left \{ \eta_{n,i} = 0
                    \right \}} - \sum_{i=1}^{b^{m_{n}}} \frac{n_{i} \varpi(\ln n_{i})}{\ln n} e^{-
                    \frac{c}{\mu}\frac{\ln n_{i}}{\ln n} } \right| \right]  \\
        & \leq  \frac{K}{\ln n}\sum_{i=1}^{b^{m_{n}}}  \mathbb{E}[ n_{i} ] \mathbb{P}\left(\eta_{n,i} \geq 1 \right ) = o\left(\frac{n}{\ln n} \right),
    \end{aligned}
\end{equation*}

\noindent for some constant $K >0$ such that $|\varpi(x)| \leq K$ for $x \in \mathbb{R}$; recall that $\varpi$ in (\ref{eq10}) is a continuous function with period $d \geq 0$.  
The previous two estimates together with Markov's inequality imply that
\begin{eqnarray} \label{eq48}
\sum_{i=1}^{b^{m_{n}}} \frac{n_{i}\ln n_{i}}{\ln^{2} n} e^{- \frac{c}{\mu}\frac{\ln n_{i}}{\ln n} }
\mathds{1}_{\left \{ \eta_{n,i} = 0 \right \}} = \sum_{i=1}^{b^{m_{n}}} \frac{n_{i}\ln
    n_{i}}{\ln^{2} n} e^{- \frac{c}{\mu}\frac{\ln n_{i}}{\ln n} } +   o_{\rm p}\left(\frac{n}{\ln n}
\right),
\end{eqnarray}

\noindent and
\begin{eqnarray} \label{eq25}
\sum_{i=1}^{b^{m_{n}}} \frac{n_{i} \varpi(\ln n_{i})}{\ln n} e^{- \frac{c}{\mu}\frac{\ln n_{i}}{\ln n} } \mathds{1}_{\left \{ \eta_{n,i} = 0 \right \}} = \sum_{i=1}^{b^{m_{n}}} \frac{n_{i}\varpi(\ln n_{i})}{\ln n} e^{- \frac{c}{\mu}\frac{\ln n_{i}}{\ln n} } +   o_{\rm p}\left(\frac{n}{\ln n} \right).
\end{eqnarray}

We observe that for large enough $k \geq 1$,
\begin{eqnarray*}
\sum_{i=1}^{b^{m_{n}}} n_{i} \mathds{1}_{\{ n_{i} \leq nb^{-k m_{n}} \}} \leq b^{-m_{n}(k-1)}n = o \left( \frac{n}{\ln^{k-1} n} \right).
\end{eqnarray*}

\noindent Then, by using the inequality $|e^{-x}-e^{-y} | \leq | x- y|$ for $x,y \in\mathbb{R}_{+}$, we have that
\begin{eqnarray} \label{eq30}
\mathbb{E}\left[ \left| \sum_{i=1}^{b^{m_{n}}} n_{i} e^{- \frac{c}{\mu}\frac{\ln n_{i}}{\ln n} }  - \sum_{i=1}^{b^{m_{n}}} n_{i} e^{- \frac{c}{\mu}} \right| \right] & \leq & \frac{c}{\mu} \frac{1}{\ln n}\mathbb{E}\left[  \sum_{i=1}^{b^{m_{n}}} n_{i} (\ln n- \ln n_{i})  \right] \nonumber \\
&  =  & \frac{c}{\mu} \frac{1}{\ln n} \mathbb{E}\left[  \sum_{i=1}^{b^{m_{n}}} n_{i} (\ln n- \ln n_{i}) \mathds{1}_{\{ n_{i} > nb^{-k m_{n}} \}}  \right] +o\left( \frac{n}{\ln^{k-1}n} \right) \nonumber \\
& = &  O\left( \frac{n \ln \ln n}{\ln n} \right),
\end{eqnarray}

\noindent where we have used that $\sum_{i=1}^{b^{m_{n}}} n_{i} \leq n$ in order to obtain the last estimation. The above implies
\begin{eqnarray} \label{eq26}
\frac{1}{\ln n}\sum_{i=1}^{b^{m_{n}}} n_{i} e^{- \frac{c}{\mu}\frac{\ln n_{i}}{\ln n} } = e^{- \frac{c}{\mu} } \frac{1 }{\ln n} \sum_{i=1}^{b^{m_{n}}} n_{i}  +  o_{\rm p}\left( \frac{n}{\ln n} \right). 
\end{eqnarray}

\noindent Similarly, we deduce from (\ref{eq29})  and (\ref{eq30})
\begin{eqnarray*}
\mathbb{E}\left[ \left| \sum_{i=1}^{b^{m_{n}}} n_{i} e^{- \frac{c}{\mu}\frac{\ln n_{i}}{\ln n} } \mathds{1}_{\left \{ \eta_{n,i} \geq 1 \right \}} - \sum_{i=1}^{b^{m_{n}}} n_{i} e^{- \frac{c}{\mu}} \mathds{1}_{\left \{ \eta_{n,i} \geq 1 \right \}} \right| \right] & \leq & \mathbb{E}\left[ \left| \sum_{i=1}^{b^{m_{n}}} n_{i} e^{- \frac{c}{\mu}\frac{\ln n_{i}}{\ln n} }  - \sum_{i=1}^{b^{m_{n}}} n_{i} e^{- \frac{c}{\mu}} \right| \right] \mathbb{P}\left(\eta_{n,1} \geq 1  \right) \\
& = & o\left( \frac{n}{\ln n} \right),
\end{eqnarray*}
\begin{eqnarray*}
\mathbb{E}\left[ \left| \sum_{i=1}^{b^{m_{n}}} \frac{n_{i} \ln n_{i}}{\ln^{2} n} e^{- \frac{c}{\mu}\frac{\ln n_{i}}{\ln n} } - \frac{1}{\ln n}\sum_{i=1}^{b^{m_{n}}} n_{i} e^{- \frac{c}{\mu}}  \right| \right] = o \left( \frac{n}{\ln n} \right)
\end{eqnarray*}

\noindent and 
\begin{eqnarray*}
\mathbb{E}\left[ \left| \sum_{i=1}^{b^{m_{n}}} \frac{n_{i} \varpi(\ln n_{i})}{\ln n} e^{- \frac{c}{\mu}\frac{\ln n_{i}}{\ln n} } - \sum_{i=1}^{b^{m_{n}}} \frac{n_{i} \varpi(\ln n_{i})}{\ln n} e^{- \frac{c}{\mu}}  \right| \right] = o \left( \frac{n}{\ln n} \right).
\end{eqnarray*}

\noindent As a consequence of the previous three estimates, we deduce from an application of the Markov's inequality that
\begin{eqnarray} \label{eq49}
\sum_{i=1}^{b^{m_{n}}} n_{i} e^{- \frac{c}{\mu}\frac{\ln n_{i}}{\ln n} } \mathds{1}_{\left \{ \eta_{n,i} \geq 1 \right \}} = e^{- \frac{c}{\mu}} \sum_{i=1}^{b^{m_{n}}} n_{i}  \mathds{1}_{\left \{ \eta_{n,i} \geq 1 \right \}} + o_{\rm p}\left( \frac{n}{\ln n} \right),
\end{eqnarray}

\begin{eqnarray} \label{eq50}
\sum_{i=1}^{b^{m_{n}}} \frac{n_{i} \ln n_{i}}{\ln^{2} n} e^{- \frac{c}{\mu}\frac{\ln n_{i}}{\ln
        n} }  = e^{- \frac{c}{\mu}}  \frac{1}{\ln n} \sum_{i=1}^{b^{m_{n}}} n_{i}  + o_{\rm p}\left(
    \frac{n}{\ln n} \right),
\end{eqnarray}

\noindent and 
\begin{eqnarray} \label{eq32}
\sum_{i=1}^{b^{m_{n}}} \frac{n_{i} \varpi(\ln n_{i})}{\ln n} e^{- \frac{c}{\mu}\frac{\ln n_{i}}{\ln n} }  = e^{- \frac{c}{\mu}}   \sum_{i=1}^{b^{m_{n}}} \frac{n_{i} \varpi(\ln n_{i})}{\ln n}  + o_{\rm p}\left( \frac{n}{\ln n} \right).
\end{eqnarray}

By applying the estimations (\ref{eq48}), (\ref{eq25}), (\ref{eq26}), (\ref{eq49}), (\ref{eq50}) and (\ref{eq32}) into the expression in (\ref{eq27}), we obtain that
\begin{align*} 
\hat{G}_{n} & = -e^{- \frac{c}{\mu}} \sum_{i=1}^{b^{m_{n}}} n_{i} \mathds{1}_{\left \{ \eta_{n,i} \geq 1 \right \}} + \sum_{i=1}^{b^{m_{n}}} n_{i} e^{-\frac{c}{\mu}\frac{\ln n_{i}}{\ln n} } \\
& \hspace*{10mm} -  \frac{c^{2}\mu^{2}-c^{2}\sigma^{2}}{2 \mu^{3}} e^{-\frac{c}{\mu} }   \frac{1}{\ln n}  \sum_{i=1}^{b^{m_{n}}} n_{i} - ce^{-\frac{c}{\mu}} \sum_{i=1}^{b^{m_{n}}} \frac{n_{i} \varpi(\ln n_{i})}{\ln n} + o_{\rm p}\left(\frac{n}{\ln n} \right);
\end{align*}

\noindent notice also that $\mathds{1}_{\left \{ \eta_{n,i} \geq 1 \right \}} = 1-\mathds{1}_{\left \{ \eta_{n,i} =0 \right \}}$. Finally, our claim in Lemma \ref{lemma9} follows by showing that
\begin{eqnarray} \label{eq31}
 \frac{1}{\ln n} \sum_{i=1}^{b^{m_{n}}} n_{i} = \frac{n}{\ln n} + o_{\rm p}\left(\frac{n}{\ln n} \right). 
\end{eqnarray}

In this direction, we notice that $\sum_{i=1}^{b^{m_{n}}} n_{i} = n - \hat{C}(n)$, where $\hat{C}(n)$ denotes the number of balls of the vertices of $T_{n}^{{\rm sp}}$ at distance less or equal to $m_{n} - 1$ from the root. It should be clear that $0 \leq  \hat{C}(n) < \max(s,s_{0}) b^{m_{n} +1} = o(n)$, which implies (\ref{eq31}).
\end{proof}

We refine the result of Lemma \ref{lemma9}.

\begin{lemma} \label{lemma10}
Suppose that Condition \ref{Cond1} is fulfilled.  We have for $\beta > -2/(1+\log_{b}\mathbb{E}[V_{1}^{2}])$ that
\begin{align*}
\hat{G}_{n} & = - e^{-\frac{c}{\mu} }   \sum_{1 \leq d_{n}(v) \leq m_{n} } n_{v}  \varepsilon_{v} +   \sum_{d_{n}(v)=m_{n}} n_{v} e^{-\frac{c}{\mu}\frac{\ln n_{v}}{\ln n} }  \\
& \hspace*{10mm} - c e^{-\frac{c}{\mu} }  \sum_{ d_{n}(v)=m_{n}} \frac{n_{v} \varpi(\ln n_{v})}{\ln n}   - \frac{c^{2}\mu^{2}-c^{2}\sigma^{2}}{2 \mu^{3}}  e^{-\frac{c}{\mu} }   \frac{n}{\ln n}    + o_{\rm p}\left(\frac{n}{\ln n} \right).
\end{align*} 

 \noindent where $\varpi:\mathbb{R} \rightarrow \mathbb{R}$ is the function in (\ref{eq10}) and $(\varepsilon_{v}, 1 \leq  d_{n}(v) \leq m_{n})$ is a sequence of i.i.d.\ Bernoulli random variables with parameter $1-p_{n}$. 
\end{lemma}

\begin{proof}
Our claim follows from Lemma \ref{lemma9} by showing that 
\begin{eqnarray} \label{eq37}
e^{-\frac{c}{\mu} }   \sum_{i=1}^{b^{m_{n}}} n_{i}  \mathds{1}_{\left \{ \eta_{n,i} \geq 1 \right \}} =  e^{-\frac{c}{\mu} }   \sum_{1 \leq d_{n}(v) \leq m_{n}} n_{v} \varepsilon_{v} + o_{\rm p} \left(\frac{n}{\ln n} \right).
\end{eqnarray}

Recall that the sequences of random variables $(\eta_{n,i}, 1 \leq i \leq b^{m_{n}})$ and $(n_{i}, 1 \leq i \leq b^{m_{n}})$ are independent. It should be obvious that
\begin{eqnarray} \label{eq34}
\mathbb{E}\left[ e^{-\frac{c}{\mu} }   \sum_{i=1}^{b^{m_{n}}} n_{i} \mathds{1}_{\left \{ \eta_{n,i} \geq 1 \right \}} \right] = \left(1- p_{n}^{m_{n}}\right) e^{-\frac{c}{\mu} }   \sum_{i=1}^{b^{m_{n}}} \mathbb{E}\left[n_{i} \right]. 
\end{eqnarray}

\noindent Next consider the vertices $v_{i,0}, v_{i,1}, \dots, v_{i,m_{n}}=v_{i}$ along the path from the root $v_{i,0}$ of $T_{n}^{{\rm sp}}$ to the vertex $v_{i}$ at height $m_{n}$. For $j=1, \dots, m_{n}$, we associate to each consecutive pair of vertices $(v_{i,j-1}, v_{i,j})$ the edge that is between them (where $v_{i,j}$ is a vertex at height $j$ on $T_{n}^{{\rm sp}}$). Define the event $E_{i,j} \coloneqq \{ \text{the edge $(v_{i,j-1}, v_{i,j})$ has been removed after percolation} \}$ and write $\varepsilon_{i,j} \coloneqq  \mathds{1}_{E_{i,j}}$. So, $(\varepsilon_{i,j}, 1 \leq j \leq m_{n} )$ is a sequence of i.i.d.\ Bernoulli random variables with parameter $1-p_{n}$ and
\begin{eqnarray} \label{eq36}
\eta_{n,i} = \sum_{j=1}^{m_{n}} \varepsilon_{i,j}.
\end{eqnarray}

\noindent Then
\begin{eqnarray} \label{eq35}
\mathbb{E}\left[e^{-\frac{c}{\mu} }   \sum_{i=1}^{b^{m_{n}}} n_{i} \eta_{n,i} \right] = m_{n} \left(1-p_{n} \right) e^{-\frac{c}{\mu} }   \sum_{i=1}^{b^{m_{n}}} \mathbb{E}\left[n_{i}  \right]. 
\end{eqnarray}

\noindent Since 
\begin{eqnarray*}
e^{-\frac{c}{\mu} }   \sum_{i=1}^{b^{m_{n}}} n_{i} \mathds{1}_{\left \{ \eta_{n,i} \geq 1 \right \}} \leq e^{-\frac{c}{\mu} }   \sum_{i=1}^{b^{m_{n}}} n_{i}  \eta_{n,i},
\end{eqnarray*}

\noindent we deduce from (\ref{eq34})  and (\ref{eq35}) that
\begin{eqnarray*}
\mathbb{E} \left[e^{-\frac{c}{\mu} }   \sum_{i=1}^{b^{m_{n}}} n_{i}  \eta_{n,i} -  e^{-\frac{c}{\mu} }   \sum_{i=1}^{b^{m_{n}}} n_{i} \mathds{1}_{\left \{ \eta_{n,i} \geq 1 \right \}}\right]  \leq  \left( m_{n}(1-p_{n})  -(1-p_{n}^{m_{n}})  \right) e^{-\frac{c}{\mu}} \sum_{i=1}^{b^{m_{n}}} \mathbb{E}[n_{i}] = o \left(\frac{n}{\ln n} \right), 
\end{eqnarray*}

\noindent where we have used that $\sum_{i=1}^{b^{m_{n}}} n_{i} \leq n$ and our assumption (\ref{eq9}) in the percolation parameter. Therefore, the identity (\ref{eq36}) implies that
\begin{eqnarray} \label{eq39}
e^{-\frac{c}{\mu} }   \sum_{i=1}^{b^{m_{n}}} n_{i} \mathds{1}_{\left \{ \eta_{n,i} \geq 1 \right \}} =  e^{-\frac{c}{\mu} }   \sum_{i=1}^{b^{m_{n}}} \sum_{j=1}^{m_{n}} n_{i}  \varepsilon_{i,j} + o_{\rm p} \left(\frac{n}{\ln n} \right).
\end{eqnarray}

Finally, let $P(v_{i})$ denote the unique path from the root $v_{i,0}$ of $T_{n}^{{\rm sp}}$ to
$v_{i}$, i.e., the unique sequence of vertices $v_{i,0}, v_{i,1}, \dots, v_{i,m_{n}} = v_{i}$. For $v = v_{i,j} \in P(v_{i}) \setminus \{v_{i,0} \}$, write $\varepsilon_{v}$ instead of $\varepsilon_{i,j}$. We observe that
\begin{eqnarray} \label{eq40}
 e^{-\frac{c}{\mu}}  \sum_{i=1}^{b^{m_{n}}}  \sum_{j=1}^{m_{n}} n_{i}\varepsilon_{i,j} & =  & e^{-\frac{c}{\mu}}  \sum_{i=1}^{b^{m_{n}}} n_{i}  \sum_{v \in P(v_{i}) \setminus \{v_{i,0} \}}  \varepsilon_{v} \nonumber \\
& = & e^{-\frac{c}{\mu}}  \sum_{1 \leq  d_{n}(v) \leq m_{n}}  \varepsilon_{v} \sum_{i:v\in P(v_{i}) \setminus \{v_{i,0} \}} n_{i} \nonumber \\
& = & e^{-\frac{c}{\mu}}  \sum_{1 \leq  d_{n}(v) \leq m_{n}} n_{v}  \varepsilon_{v} + o_{\rm p} \left(\frac{n}{\ln n} \right),
\end{eqnarray}

\noindent because $n_{v}-sb^{m_{n}} \leq \sum_{i:v\in P(v_{i}) \setminus \{v_{i,0} \}} n_{i} \leq n_{v}$. 

Therefore, the estimation (\ref{eq37}) follows by combining (\ref{eq39}) and (\ref{eq40}).
\end{proof}

Following the original idea of Janson \cite{Jans2004} and subsequently used by Holmgren \cite{Holmgren2010, Holm2011} (where the number of random cuts required to isolate the root of a tree
was studied), we express $\hat{G}_{n}$ as a sum of triangular arrays. We write
\begin{eqnarray} \label{eqnew1}
\xi_{v} \coloneqq  e^{-\frac{c}{\mu} }  \frac{ \ln n}{n} n_{v} \varepsilon_{v}, \hspace*{5mm} \text{for} \hspace*{3mm} v \in T_{n}^{{\rm sp}} \hspace*{5mm} \text{such that} \hspace*{5mm} d_{n} \leq m_{n},
\end{eqnarray}

\noindent where $(\varepsilon_{v}, 1 \leq d_{n}(v) \leq m_{n})$ is a sequence of i.i.d.\ Bernoulli random variables with parameter $1-p_{n}$. We also write $\xi_{i}^{\prime} \coloneqq  -\alpha_{n}/n $ for $i \in \mathbb{N}$, where
\begin{align*}
\alpha_{n} & \coloneqq  \frac{\ln n}{n} \sum_{d_{n}(v)=m_{n}} n_{v} e^{-\frac{c}{\mu}\frac{\ln n_{v}}{\ln n} } - ce^{-\frac{c}{\mu} }  \sum_{d_{n}(v)=m_{n}} \frac{n_{v} \varpi(\ln n_{v})}{n} \\
& \hspace*{10mm} -e^{-\frac{c}{\mu}} \ln n -\frac{c}{\mu} e^{-\frac{c}{\mu}} \ln \ln n + ce^{-\frac{c}{\mu}}\varpi(\ln n) - ce^{-\frac{c}{\mu}}\phi \left(\ln \left( \theta^{-1} e^{-\frac{c}{\mu}} \ln n \right) \right) - \frac{c^{2}\mu^{2}-c^{2}\sigma^{2}}{2 \mu^{3}}  e^{-\frac{c}{\mu} }
\end{align*}

\noindent for any constant $\theta > 0$. By normalizing $\hat{G}_{n}$, Lemma \ref{lemma10} gives that
\begin{align*}
&  \left(n^{-1} \hat{G}_{n}-e^{-\frac{c}{\mu}}  \right)\ln n - c \mu^{-1} e^{-\frac{c}{\mu}} \ln \ln n + ce^{-\frac{c}{\mu}}\left( \varpi(\ln n)   - \phi \left(\ln \left( \theta^{-1} e^{-\frac{c}{\mu}} \ln n \right) \right)\right)  \\
& \hspace*{10mm} =  - \sum_{1 \leq  d_{n}(v) \leq m_{n}} \xi_{v} - \sum_{i =1}^{n} \xi_{i}^{\prime} + o_{\rm p}(1).
\end{align*}

\noindent Recall that the cardinalities $(n_{v}, 1 \leq  d_{n}(v) \leq m_{n})$ are not independent random variables and thus the sequence $(\xi_{v}, 1 \leq  d_{n}(v) \leq m_{n}) \cup (\xi^{\prime}_{i}, i \in \mathbb{N})$ is not a triangular array. However, conditional on $\mathcal{F}_{m_{n}}$, the $\sigma$-field generated by $(n_{v}, 1 \leq d_{n}(v) \leq m_{n})$, the sequence $(\xi_{v}, 1 \leq d_{n}(v) \leq m_{n}) \cup (\xi^{\prime}_{i}, i \in \mathbb{N})$ is a triangular array where $(\xi^{\prime}_{i}, i \in \mathbb{N})$ is a deterministic sequence.  

Finally, the proofs of Theorem \ref{Theo2} and Theorem \ref{NewTheo2} are going to be completed via a classical theorem for convergence of sums of triangular arrays to infinitely divisible distributions; see e.g. \cite[Theorem 15.28]{Kall2002}. In this direction, we need the following result. For the sake of simplicity, we introduce the following notation. For any constants $\theta, x >0$,
\begin{eqnarray*}
\Delta_{n,1}\coloneqq  \sum_{1 \leq  d_{n}(v) \leq m_{n}} \mathbb{P}(\xi_{v} \geq x| \mathcal{F}_{m_{n}}), \hspace*{5mm} \Delta_{n,2} \coloneqq \sum_{1 \leq d_{n}(v) \leq m_{n}} \mathbb{E} \left[ \xi_{v} \mathds{1}_{\left \{  \xi_{v} \leq \theta \right\} }| \mathcal{F}_{m_{n}} \right] - \alpha_{n}, 
\end{eqnarray*}

\begin{eqnarray*}
 \hspace*{5mm} \text{and} \hspace*{5mm}   \Delta_{n,3} \coloneqq \sum_{1 \leq d_{n}(v) \leq m_{n}} Var \left( \xi_{v} \mathds{1}_{\left \{  \xi_{v} \leq \theta \right\} }| \mathcal{F}_{m_{n}} \right).
\end{eqnarray*}

\noindent For $\theta>0$ and $x \geq 0$, we also define the function 
\begin{eqnarray*}
\psi_{\theta}(x) = 1 - \frac{\theta x}{1-e^{-x}} e^{x \lfloor \varrho -x^{-1}\ln \theta - x^{-1}c/\mu \rfloor - x \varrho +c/\mu}
\end{eqnarray*}

\noindent such that $\psi_{\theta}(0)=0$.  

\begin{theorem} \label{Theo3}
Recall that $m_{n} = \lfloor\beta \log_{b} \ln n \rfloor$. Suppose that Condition \ref{Cond1} holds. Furthermore, if 
$\ln V_{1}$ is lattice with span $d$ defined in (\ref{eq51}), we also assume that Condition \ref{NewCond1} holds for some $\varrho \in [0,1)$. For any constant $\theta >0$ and large enough $\beta$, the following statements hold as $n \rightarrow \infty$,
\begin{itemize}
\item[(i)] $\displaystyle \sup_{1 \leq  d_{n}(v) \leq m_{n}}\mathbb{P}\left(\xi_{v} \geq x| \mathcal{F}_{m_{n}} \right) \xrightarrow[]{a.s.} 0$ for every $x >0$.

\item[(ii)] For every $x >0$,
\begin{eqnarray*}
\Delta_{n,1} \xrightarrow[]{\mathbb{P}} \nu([x,\infty) ):= \left\{ \begin{array}{lcl}
              \frac{c}{\mu}e^{-\frac{c}{\mu}} \frac{1}{x} & \mbox{  if } & \ln V_{1} \hspace*{2mm} \text{is non-lattice}, \\
              \frac{c}{\mu} \frac{d}{1-e^{-d}} e^{d \lfloor \varrho -d^{-1}\ln x - d^{-1}c/\mu \rfloor - d \varrho}  & \mbox{  if } & \ln V_{1} \hspace*{2mm} \text{is lattice}. \\
              \end{array}
    \right.
\end{eqnarray*}

\item[(iii)] $ \displaystyle \Delta_{n,2}  \xrightarrow[]{\mathbb{P}} \left(\frac{2c\mu + c\mu^{2} - c\sigma^{2}  -\mu \sigma^{2} + \mu^{3} }{2\mu^{2}}  + \ln \theta + \psi_{\theta}(d) \right) \frac{c}{\mu} e^{-\frac{c}{\mu}} $.

\item[(iv)] $ \displaystyle  \Delta_{n,3}  \xrightarrow[]{\mathbb{P}} \theta  \left(1 + \psi_{\theta}(d) \right) \frac{c}{\mu}e^{-\frac{c}{\mu}}$.
\end{itemize}
\end{theorem}

The proof of this theorem is rather technical and postponed until the Appendix \ref{appendix1}. 

\begin{proof}[Proof of Theorem \ref{Theo2}]
We apply \cite[Theorem 15.28]{Kall2002} with the constants
\begin{eqnarray*}
a=0 \hspace*{5mm} \text{and} \hspace*{5mm} b = \left(\frac{2c\mu + c\mu^{2} - c\sigma^{2}  -\mu \sigma^{2} + \mu^{3} }{2\mu^{2}}  \right) \frac{c}{\mu} e^{-\frac{c}{\mu}} 
\end{eqnarray*}

\noindent to the sequence $\left( Z_{n} \coloneqq \sum_{1 \leq d_{n}(v) \leq m_{n}} \xi_{v} + \sum_{i =1}^{n} \xi_{i}^{\prime}, n \geq 1 \right)$ conditioned on $\mathcal{F}_{m_{n}}$. We observe that $\alpha_{n}/n \rightarrow 0$ as $n \rightarrow \infty$. Thus, Theorem \ref{Theo3} (i) implies that conditioned on $\mathcal{F}_{m_{n}}$ the variables $(\xi_{v} ,  1 \leq d_{n}(v) \leq m_{n}) \cup (\xi_{i}^{\prime}, i \geq 1)$ form a null array. Theorem \ref{Theo3} (ii) shows that $
\nu({\rm d}x) = c \mu^{-1} e^{-\frac{c}{\mu}} x^{-2}$, for $x >0$. Hence
\begin{eqnarray*}
\int_{0}^{\theta}x^{2} \nu({\rm d}x) = c \mu^{-1} e^{-\frac{c}{\mu}} \theta \hspace*{4mm} \text{and} \hspace*{4mm} \int_{\theta}^{1}x \nu({\rm d}x) = - c \mu^{-1}e^{-\frac{c}{\mu}} \ln \theta  \hspace*{4mm}  \text{for} \hspace*{4mm} \theta >0.
\end{eqnarray*}

\noindent Thus the right-hand side of Theorem \ref{Theo3} (iii) and (iv) can be written as
\begin{eqnarray*}
b- \int_{\theta}^{1}x \nu({\rm d}x) \hspace*{4mm}   \text{and} \hspace*{4mm} a +\int_{0}^{\theta}x^{2} \nu({\rm d}x), \hspace*{4mm}  \text{for} \hspace*{4mm} \theta >0,
\end{eqnarray*}

\noindent respectively. Therefore \cite[Theorem 15.28]{Kall2002} implies that there is the convergence in distribution $Z_{n} \xrightarrow[ ]{d} W$ conditioned on $\mathcal{F}_{m_{n}}$, where $W$ has a weakly $1$-stable distribution with characteristic function given by
\begin{eqnarray*}
\mathbb{E}[e^{itW}] = \exp \left(ibt + \int_{0}^{\infty} \left(e^{it x} - 1 -itx \mathds{1}_{\{x < 1 \}} \right) \nu({\rm d} x)\right).
\end{eqnarray*}
 
\noindent This expression can be simplified to show that $W$ is equal in distribution to 
\begin{eqnarray*}
\frac{c}{\mu} e^{- \frac{c}{\mu}} \left( Z + \ln \left( \frac{c}{\mu} \right) + \frac{(\mu^{2}- \sigma^{2})(c+\mu)}{2\mu^{2}}  - \gamma + 1 \right),
\end{eqnarray*}

\noindent where $\gamma$ is the Euler constant and the variable $Z$ has the continuous Luria-Delbr\"uck distribution; see, e.g., \cite[Section XVII.3]{Fe1971}. Finally, we notice that the conditioning does not affect the distribution of $W$. Then it follows that the convergence $Z_{n} \xrightarrow[ ]{d} W$ holds also unconditioned; We refer to \cite[pages 407-409]{Holmgren2010} for a formal proof of this fact where a general argument is provided for a sequence with a similar structure as $\left( Z_{n}, n \geq 1 \right)$.  Therefore, the proof of Theorem \ref{Theo2} is completed.
\end{proof}

\begin{proof}[Proof of Theorem \ref{NewTheo2}]
The proof follows along the lines of the proof of Theorem \ref{Theo2}. Details are left to the reader.
\end{proof}

\section{Proof of Theorem \ref{Theo1}} \label{sec4}

In this section, we deduce Theorem \ref{Theo1} from Theorem \ref{Theo2} by showing that $\frac{n}{\ln n} G_{n}$ and $\frac{\alpha n}{\ln n}\hat{G}_{n}$ are close enough as $n \rightarrow \infty$. We start by recalling some notation from Section \ref{sec2}. Remember that we write $m_{n} =  \lfloor \beta \log_{b} \ln n \rfloor$, for some constant $\beta > 0$, and that we assume that $n$ is large enough such that $0 < m_{n} < \ln n$. For $1 \leq i \leq b^{m_{n}}$, recall also that we let $v_{i}$ be a vertex in $T_{n}^{{\rm sp}}$ at height $m_{n}$ and we let $n_{i}$ be the number of balls stored at the sub-tree rooted at $v_{i}$. We further let $N_{i}$ be the (random) number of vertices at the sub-tree rooted at $v_{i}$. 

We denote by $C_{n,i}$ the number of vertices of the sub-tree of $T_{n}^{{\rm sp}}$ rooted at $v_{i}$ after percolation with parameter $p_{n}$. Clearly, $(C_{n,i}, 1 \leq i \leq b^{m_{n}})$ are conditionally independent random variables given $(n_{i}, 1 \leq i \leq b^{m_{n}})$. We write $\mathbb{E}_{n_{i}}[C_{n,i}] \coloneqq\mathbb{E}[C_{n,i} | n_{i}]$, i.e., it is the conditional expected value of $C_{n,i}$ given $n_{i}$.

We have the following estimation of  $C_{n,i}$ that corresponds to Lemma \ref{lemma6}.

\begin{lemma} \label{lemma15}
Suppose that Condition \ref{Cond1} and \ref{Cond2} are fulfilled. For $1 \leq i \leq d^{m_{n}}$, we have that
\begin{eqnarray*}
\mathbb{E}_{n_{i}}[C_{n,i} ] = \alpha n_{i}e^{-\frac{c}{\mu}\frac{\ln n_{i}}{\ln n} } -  \alpha  \frac{c^{2} \mu^{2}-c^{2}\sigma^{2}}{2\mu^{3}} \frac{n_{i} \ln n_{i}}{\ln^{2} n} e^{-\frac{c}{\mu}\frac{\ln n_{i}}{\ln n} } -  c\zeta   \frac{ n_{i}}{\ln n}  e^{- \frac{c}{\mu}\frac{\ln n_{i}}{\ln n} } +  o\left(\frac{n_{i}}{\ln n} \right),
\end{eqnarray*}

\noindent where $\zeta \in \mathbb{R}$ is the constant in (\ref{eq5}). 
\end{lemma}

\begin{proof}
For $1 \leq i \leq d^{m_{n}}$, let $T_{i}$ be the sub-tree of $T_{n}^{{\rm sp}}$ rooted at the vertex $v_{i}$ at height $m_{n}$. Let $u_{i}$ be a vertex in $T_{i}$ with the uniform distribution on the set of vertices of the sub-tree $T_{i}$. Let $d_{n_{i}}(u_{i})$ be the height of $u_{i}$. We have the following key observation made by Bertoin \cite[Proof of Theorem 1]{Be1},
\begin{eqnarray} \label{eq70}
\mathbb{E}_{n_{i}}\left[ N_{i}^{-1} C_{n,i} \right] = \mathbb{E}_{n_{i}}\left[ p_{n}^{d_{n_{i}}(u_{i})} \right].
\end{eqnarray}

\noindent In words, the left-hand side can be interpreted as the probability that $u_{i}$ belongs to the percolation cluster containing the root of $T_{i}$, i.e., $v_{i}$, while the right-hand side can be interpreted as the probability that no edge has been removed in the path between $u_{i}$ and $v_{i}$. Then a similar computation as in the proof of Lemma 
\ref{lemma6} together with Lemma \ref{lemma1} (i)-(iii) in Appendix \ref{sec3} shows that  
\begin{eqnarray} \label{eq71}
\mathbb{E}_{n_{i}}\left[ p_{n}^{d_{n_{i}}(u_{i})} \right] = e^{-\frac{c}{\mu}\frac{\ln n_{i}}{\ln n} } - \frac{c^{2}\mu^{2}- c^{2}\sigma^{2}}{2\mu^{3}} \frac{\ln n_{i}}{\ln^{2} n} e^{-\frac{c}{\mu}\frac{\ln n_{i}}{\ln n} } -  \frac{c \zeta}{\alpha} \frac{1}{\ln n} e^{-\frac{c}{\mu}\frac{\ln n_{i}}{\ln n} }  + o\left(\frac{1}{\ln n} \right).
\end{eqnarray} 

\noindent On the other hand, we note that $C_{n,i} \leq N_{i}$. Hence Condition \ref{Cond2} and Remark \ref{remark1} imply that
\begin{eqnarray*}
\left| \mathbb{E}_{n_{i}}\left[ N_{i}^{-1} C_{n,i} \right] - \mathbb{E}_{n_{i}}\left[ \mathbb{E}_{n_{i}}^{-1}[N_{i}] C_{n,i} \right] \right| \leq \mathbb{E}_{n_{i}}^{-1}[N_{i}] \mathbb{E}\left[ \left| N_{i} -  \mathbb{E}_{n_{i}}[N_{i}]\right| \right] = o\left( \ln^{-1}n\right).
\end{eqnarray*}

\noindent By making use of Condition \ref{Cond2} one more time, we deduce that
\begin{eqnarray} \label{eq72}
 \mathbb{E}_{n_{i}}\left[ N_{i}^{-1} C_{n,i} \right] = \alpha^{-1} \mathbb{E}_{n_{i}}\left[ n_{i}^{-1} C_{n,i} \right] + o\left( \ln^{-1}n\right).
\end{eqnarray}

\noindent Therefore, our claim follows from the combination of (\ref{eq70}), (\ref{eq71}) and (\ref{eq72}).\end{proof}

Recall that $\eta_{n,i}$ denotes the total number of edges on the branch from $v_{i}$ to the root which has been deleted after percolation with parameter $p_{n}$. The next result is analogous of Lemma \ref{lemma8}.

\begin{lemma}  \label{lemma16}
Suppose that Condition \ref{Cond1} and \ref{Cond2} are fulfilled. We have for $\beta > -2/(\log_{b}\mathbb{E}[V_{1}^{2}]+1)$ that
\begin{eqnarray*}
G_{n} =   \sum_{i=1}^{b^{m_{n}}} \mathbb{E}_{n_{i}}[C_{n,i}] \mathds{1}_{\left \{ \eta_{n,i} = 0 \right \}}  + o_{\rm p}\left(\frac{n}{ \ln n}\right).
\end{eqnarray*}
\end{lemma}

\begin{proof}
The proof follows from a very similar argument as the proof of Lemma \ref{lemma8}. 
\end{proof}

Finally, we show that $\frac{n}{\ln n} G_{n}$ and $\frac{\alpha n}{\ln n}\hat{G}_{n}$ possess the same asymptotic behavior. 

\begin{lemma}  \label{lemma17}
Suppose that Condition \ref{Cond1} and \ref{Cond2} is fulfilled.  We have for $\beta > -2/(1+\log_{b}\mathbb{E}[V_{1}^{2}]+1)$ that
\begin{eqnarray*} 
G_{n} = \alpha \hat{G}_{n} + c \alpha (\varsigma  -\zeta \alpha^{-1}) e^{-\frac{c}{\mu}} \frac{n}{\ln n} + o_{\rm p}\left(\frac{n}{\ln n} \right).
\end{eqnarray*}

 \noindent where $\zeta \in \mathbb{R}$ is the constant defined in (\ref{eq5}) and 
 $\varsigma \in \mathbb{R}$ is the constant value of the $d$-periodic function $\varpi$  in (\ref{eq10}) when $d=0$. 
\end{lemma}

\begin{proof}
We deduce from Lemma \ref{lemma15}, Lemma \ref{lemma16} and equation (\ref{eq27}) that
\begin{eqnarray*}
G_{n} = \alpha \hat{G}_{n}  - c \alpha \sum_{i=1}^{b^{m_{n}}} \left(   \frac{\zeta}{\alpha}  \frac{n_{i}}{\ln n}  - \frac{n_{i} \varpi(\ln n_{i})}{\ln n} \right) e^{-\frac{c}{\mu}\frac{\ln n_{i}}{\ln n}} \mathds{1}_{\left \{ \eta_{n,i} = 0 \right \}} + o_{\rm p}\left(\frac{n}{\ln n} \right).
\end{eqnarray*}

\noindent By Condition \ref{Cond2}, the random variable $\ln V_{1}$ is non-lattice and thus the function $\varpi$ is a constant equal to $\varsigma$. Hence 
\begin{eqnarray*}
G_{n} = \alpha \hat{G}_{n}  - c \alpha ( \zeta \alpha^{-1} - \varsigma) \sum_{i=1}^{b^{m_{n}}} \frac{n_{i}}{\ln n}   e^{-\frac{c}{\mu}\frac{\ln n_{i}}{\ln n}} \mathds{1}_{\left \{ \eta_{n,i} = 0 \right \}} + o_{\rm p}\left(\frac{n}{\ln n} \right).
\end{eqnarray*}

\noindent Furthermore, the estimations (\ref{eq25}),  (\ref{eq26}) and  (\ref{eq31}) allow us to deduce that 
\begin{eqnarray*}
 \sum_{i=1}^{b^{m_{n}}} \frac{n_{i}}{\ln n} e^{-\frac{c}{\mu} \frac{\ln n_{i}}{\ln n}} \mathds{1}_{\left \{ \eta_{n,i} = 0 \right \}} = e^{-\frac{c}{\mu}} \frac{n}{\ln n} + o_{\rm p}\left(\frac{n}{\ln n} \right).
\end{eqnarray*}

\noindent Therefore, the result follows clearly by combining the previous two estimates. 
\end{proof}

We are now in the position to prove Theorem \ref{Theo1}.

\begin{proof}[Proof of Theorem \ref{Theo1}]
By normalizing $G_{n}$, Lemma \ref{lemma17} gives that
\begin{eqnarray*}
 \left(n^{-1} G_{n} - \alpha e^{-\frac{c}{\mu}}  \right)\ln n - \frac{c \alpha}{\mu}e^{-\frac{c}{\mu}} \ln \ln n = \alpha \left(n^{-1} \hat{G}_{n} -e^{-\frac{c}{\mu}}  \right)\ln n - \frac{c \alpha}{\mu}e^{-\frac{c}{\mu}} \ln \ln n + c \alpha (\varsigma  -\zeta \alpha^{-1}) e^{-\frac{c}{\mu}}  +  o_{\rm p}(1).
\end{eqnarray*}

\noindent Therefore, the result in Theorem \ref{Theo1} follows from a simple application of Theorem \ref{Theo2}.
\end{proof}

\section{Percolation on \texorpdfstring{$b$}{b}-regular trees} \label{sec6}

In this section, we point out that the approach developed in the proof of Theorem \ref{Theo2} can be also applied to study percolation on other classes of trees. We focus here on the case of rooted complete regular $b$-ary trees $T_{h}^{{\rm reg}}$ with height $h \in \mathbb{N}$ and $b \geq 2$ a fixed integer (i.e., each vertex has exactly out-degree $b$). We note that there are $b^{k}$ vertices at distance $k = 0, 1, \dots, h$ from the root and a total of $n_{h} = (b^{h+1}-1)/(b-1)$
vertices. We perform Bernoulli bond percolation with parameter
\begin{eqnarray*}
p_{h} = e^{-c/h},
\end{eqnarray*}

\noindent where $c > 0$ is fixed. It is not difficult to show that this choice of the percolation parameter corresponds precisely to the supercritical regime, i.e., there exists a (unique) giant cluster such that $\lim_{h \rightarrow \infty} n_{h}^{-1} G_{h}^{{\rm reg}} = e^{- c }$, in probability, where $G_{h}^{{\rm reg}}$ denotes the size (i.e., the number of vertices) of the cluster that contains the root. We refer to \cite[Section 3]{Be1} for details. We are interested in the fluctuations of $G_{h}^{{\rm reg}}$. Recall that we write $y = \lfloor y \rfloor + \{ y \}$ for the decomposition of a real number $y$ as the sum of its integer and fractional parts. We introduce for every $\rho \in [0,1)$ and $x > 0$,
\begin{eqnarray*}
\bar{\Lambda}_{\rho }(x) = \frac{b^{- \rho + \lfloor \rho  - \log_{b}x \rfloor +1}}{b-1}.
\end{eqnarray*} 

\noindent This function decreases as $x \rightarrow \infty$ and it can be viewed as the tail of a measure $\Lambda_{\rho}$ on $(0, \infty)$. Furthermore, it is not difficult to see that this measure fulfills the integral condition $\int_{(0, \infty)} (1 \wedge x^{2}) \Lambda_{\rho }({\rm d} x) < \infty$. This enables us to introduce a L\'evy process without negative jumps $L_{\rho } = (L_{\rho }(t))_{t \geq 0}$ with Laplace exponent
\begin{eqnarray*}
\Psi_{\rho }(a) = \int_{(0, \infty)} (e^{-ax} - 1 + ax \mathds{1}_{\{x < 1 \}}) \Lambda_{\rho }({\rm d} x), \hspace*{4mm} \text{for} \hspace*{2mm} a \geq 0.
\end{eqnarray*} 

We stress that the same process arises in the study of percolation on rooted complete regular $b$-ary trees. More precisely, Bertoin \cite[Theorem 3.1]{Be2} has proven that the fluctuations of the number of vertices at height $h$ which has been disconnected from the root after percolation are described by $L_{\rho}$. Furthermore, a similar process appears in relation with limit theorems for the number of random records on a complete binary tree; see Janson \cite{Jans2004}. 

We state the following analogue of Theorem \ref{Theo1}. 
 
\begin{theorem} \label{Theo4}
In the regime where $h \rightarrow \infty$ with $\{\log_ {b} h \} \rightarrow \rho  \in [0,1)$, there is the convergence in distribution
\begin{eqnarray*}
\left(\frac{G_{h}^{{\rm reg}}}{n_{h}}- e^{-c} \right)h - ce^{-c} \log_{b} h \xrightarrow[ ]{d} - e^{-c} \left( L_{\rho }(c) + c \rho  - \frac{c}{b-1} \right).
\end{eqnarray*}
\end{theorem}

We now prepare the ground for the proof of Theorem \ref{Theo4}. The strategy is the same as the one used in the proof of Theorem \ref{Theo2}. We write $m_{h} = 2 \lfloor \log_{b} h \rfloor$ and assume that $h$ is large enough such that $0 < m_{h} < h$. For $1 \leq i \leq b^{m_{h}}$, let $v_{i}$ be the $b^{m_{h}}$ vertices at height $m_{h}$. We notice that the number of vertices of the sub-tree of $T_{h}^{{\rm reg}}$ rooted at $v_{i}$ is given by $n_{h,i} = (b^{h-m_{h}+1}-1)/(b-1)$. We denote by $C_{h,i}$ the number of vertices of the sub-tree of $T_{h}^{{\rm reg}}$ rooted at $v_{i}$ after percolation with parameter $p_{h}$. Clearly, $(C_{h,i}, 1 \leq i \leq b^{m_{h}})$ is a sequence of independent and identically distributed random variables. 

\begin{lemma} \label{lemma18}
For $1 \leq i \leq b^{m_{h}}$, we have that
\begin{eqnarray*}
\mathbb{E}[C_{h,i} ] = n_{h,i} e^{-c} + n_{h,i}h^{-1} (b-1)^{-1}ce^{-c} + n_{h,i} m_{h} h^{-1} c e^{-c} +o(n_{h,i} h^{-1}).
\end{eqnarray*}
\end{lemma}

\begin{proof}
For $1 \leq i \leq b^{m_{h}}$, let $T_{h, i}$ be the sub-tree of $T_{h}^{{\rm reg}}$ rooted at the vertex $v_{i}$. Let $u_{i}$ denote a uniform chosen vertex in $T_{h, i}$ and write $d_{h}(u_{i})$ for its height in $T_{h, i}$. It should be obvious that $\mathbb{P}( d_{h}(u_{i}) = k ) = b^{k} n_{h,i}^{-1}$, for $k \in \{ 0, 1, \dots, h-m_{h} \}$. By the key observation made by Bertoin \cite[Proof of Theorem 1]{Be1}, we have that
\begin{eqnarray*}
\mathbb{E} \left[ n_{h,i}^{-1} C_{h,i} \right] & = & \mathbb{E} \left[ e^{-ch^{-1}d_{h}(u_{i})}\right] = \sum_{k=0}^{h-m_{h}} e^{-c h^{-1}k } \mathbb{P}(d_{h}(u_{i}) = k)  \\
& = & \frac{b^{h-m_{h}}}{n_{h,i}} e^{-c\frac{h-m_{h}}{h}} \sum_{k=0}^{h - m_{h}} e^{ch^{-1}k}b^{-k} \\
& = & \frac{b^{h-m_{h}}}{n_{h,i}} e^{-c\frac{h-m_{h}}{h}} \left( \frac{b}{b-1} + \frac{cb}{h(b-1)^{2}} + o(h^{-1})\right).
\end{eqnarray*}

\noindent Recall that $n_{h,i} = (b^{h-m_{h}+1}-1)/(b-1)$. Therefore, after some simple computations we obtain that
\begin{eqnarray*}
\mathbb{E} \left[ n_{h,i}^{-1} C_{h,i} \right] & = &  e^{-c\frac{h-m_{h}}{h}} \left( 1 +c(b-1)^{-1}h^{-1} \right) + o(h^{-1})
\end{eqnarray*}

\noindent from which our claim follows.
\end{proof}

Let $\eta_{h,i}$ be the total number of edges on the branch from $v_{i}$ to the root which have been deleted after percolation with parameter $p_{h}$. Notice that the random variable $\eta_{h,i}$ has the binomial distribution with parameters $(m_{h}, 1-p_{h})$. But the random variables $(\eta_{h,i}, 1 \leq i \leq b^{m_{h}})$ are not independent. On the other hand,  $\eta_{h,i} = 0$ if and only if the vertex $v_{i}$ is still connected to the root of $T_{h}^{{\rm reg}}$. 

\begin{lemma} \label{lemma19}
We have that
\begin{eqnarray*}
G_{h}^{{\rm reg}} =  - n_{h,1} e^{-c}  \sum_{i=1}^{b^{m_{h}}} \mathds{1}_{\left \{ \eta_{h,i} \geq 1 \right \}} + n_{h} e^{-c} +  n_{h}h^{-1} (b-1)^{-1}ce^{-c} + n_{h}m_{h}h^{-1}ce^{-c} + o_{\rm p}(n_{h}h^{-1}).
\end{eqnarray*}
\end{lemma}

\begin{proof}
We denote by $C_{h,0}$ the number of vertices of the tree $T_{h}^{{\rm reg}}$ at height less or equal to $m_{h}-1$ that are connected to the root after percolation with parameter $p_{h}$. Then, it should be plain that
\begin{eqnarray*}
G_{h}^{{\rm reg}} = C_{h,0} + \sum_{i=1}^{b^{m_{h}}} C_{h,i}  \mathds{1}_{\left \{ \eta_{h,i} = 0 \right \}}.
\end{eqnarray*}

\noindent We observe that the sequences of random variables $(\eta_{h,i}, 1 \leq i \leq b^{m_{h}})$ and $(C_{h,i}, 1 \leq i \leq b^{m_{h}})$ are independent. By conditioning first on the value of the random variables $(\eta_{h,i}, 1 \leq i \leq b^{m_{h}})$ and then taking expectation, we obtain that
\begin{eqnarray*}
\mathbb{E} \left[ \left( G_{h}^{{\rm reg}} - C_{h,0} - \sum_{i=1}^{b^{m_{h}}} \mathbb{E} \left[C_{h,i} \right]  \mathds{1}_{\left \{ \eta_{h,i} = 0 \right \}} \right)^{2} \right] & = & \mathbb{E}\left[ \sum_{i=1}^{b^{m_{h}}} \mathbb{E} \left[ \left(  C_{h,i} - \mathbb{E}[C_{h,i} ] \right)^{2}\right] \mathds{1}_{\left \{ \eta_{h,i} = 0 \right \}}  \right] \\
& = &   \sum_{i=1}^{b^{m_{h}}} \mathbb{E} \left[ \left(  C_{h,i} - \mathbb{E}[C_{h,i} ] \right)^{2}\right]  \mathbb{P}\left ( \eta_{h,i} = 0 \right ). 
\end{eqnarray*}

\noindent On the one hand, $\mathbb{P}( \eta_{h,i} = 0) \leq 1$. On the other hand, Bertoin \cite[Section 3]{Be1} has proven in \cite[Proof of Corollary 1]{Be1} that $\mathbb{E} [ (  C_{h,i} - \mathbb{E}[C_{h,i} ] )^{2}] = o(n_{h,i}^{2})$. Thus,
\begin{eqnarray*}
\mathbb{E} \left[ \left( G_{h}^{{\rm reg}} - C_{h,0} - \sum_{i=1}^{b^{m_{h}}} \mathbb{E} \left[C_{h,i} \right]  \mathds{1}_{\left \{ \eta_{h,i} = 0 \right \}} \right)^{2} \right]  = \sum_{i=1}^{b^{m_{h}}} o(n_{h,i}^{2}) = o(n_{h}^{2}h^{-2}).
\end{eqnarray*}

\noindent The above estimate and Chebyshev's inequality imply that
\begin{eqnarray*}
 G_{h}^{{\rm reg}} = C_{h,0} + \mathbb{E}[C_{h,1} ]\sum_{i=1}^{b^{m_{h}}}  \mathds{1}_{\left \{ \eta_{h,i} = 0 \right \}} + o_{{\rm p}}(n_{h} h^{-1})
\end{eqnarray*}

\noindent since $(C_{h,i}, 1 \leq i \leq b^{m_{h}})$ is a sequence of i.i.d.\ random variables. Moreover, we notice that $0 \leq C_{h,0} < b^{m_{h}+1} = o(n_{h} h^{-1})$. Hence 
\begin{eqnarray} \label{eq90}
G_{h}^{{\rm reg}}=  \mathbb{E}[C_{h,1} ]\sum_{i=1}^{b^{m_{h}}}  \mathds{1}_{\left \{ \eta_{h,i} = 0 \right \}} + o_{{\rm p}}(n_{h} h^{-1}).
\end{eqnarray}

\noindent We note that 
\begin{eqnarray} \label{eq91}
 \sum_{i=1}^{b^{m_{h}}} \mathds{1}_{  \left \{ \eta_{h,i} = 0 \right \} } = b^{m_{h}} - \sum_{i=1}^{b^{m_{h}}} \mathds{1}_{  \left \{ \eta_{h,i} \geq 1 \right \} }.
\end{eqnarray}

\noindent Finally, our claim follows by combining (\ref{eq91}) and Lemma \ref{lemma18} into (\ref{eq90}).
\end{proof}

We can now complete the proof of Theorem \ref{Theo4}. 

\begin{proof}[Proof of Theorem \ref{Theo4}]
From Lemma \ref{lemma19} we deduce that 
\begin{align*}
& \left( \frac{G_{h}^{{\rm reg}}}{n_{h}}- e^{-c} \right)h - ce^{-c} \log_{b} h   \\
& \hspace*{10mm} =  -\frac{n_{h,1} h}{n_{h}} e^{-c} \sum_{i=1}^{b^{m_{h}}} \mathds{1}_{\left \{ \eta_{h,i} \geq 1 \right \}} + ce^{-c}\lfloor \log_{b} h \rfloor - ce^{-c} \{ \log_{b} h \}  + \frac{c}{b-1}e^{-c} + o_{{\rm p}}(1).
\end{align*}

\noindent Since $n_{h}^{-1} n_{h,1} = b^{-m_{h}} + o(b^{-m_{h}})$ and 
\begin{eqnarray*}
\mathbb{E} \left[\sum_{i=1}^{b^{m_{h}}} \mathds{1}_{\left \{ \eta_{h,i} \geq 1 \right \}} \right] = \sum_{i=1}^{b^{m_{h}}} \mathbb{P}\left (\eta_{h,i} \geq 1 \right ) = b^{m_{h}}(1-e^{-cm_{h}h^{-1}}),
\end{eqnarray*}

\noindent we conclude by the Markov inequality that
\begin{align*}
& \left( \frac{G_{h}^{{\rm reg}}}{n_{h}}- e^{-c} \right)h - ce^{-c} \log_{b} h  \\
& \hspace*{10mm} = - h b^{-m_{h}} e^{-c} \sum_{i=1}^{b^{m_{h}}} \mathds{1}_{\left \{ \eta_{h,i} \geq 1 \right \}} + ce^{-c}\lfloor \log_{b} h \rfloor - ce^{-c} \{ \log_{b} h \}  + \frac{c}{b-1}e^{-c} + o_{{\rm p}}(1).
\end{align*}

\noindent Our claim follows by \cite[Corollary 3.4]{Be2} that establishes the convergence in distribution
 \begin{eqnarray*}
 h b^{-m_{h}}  \sum_{i=1}^{b^{m_{h}}} \mathds{1}_{\left \{ \eta_{h,i} \geq 1 \right \}} - c\lfloor \log_{b} h \rfloor \xrightarrow[ ]{d} L_{\rho }(c),
\end{eqnarray*}

\noindent in the regime where $h \rightarrow \infty$ with $\{\log_ {b} h \} \rightarrow \rho \in [0,1)$.
\end{proof}

\begin{remark}
One could have finished the proof of Theorem \ref{Theo4} along the same lines as for Theorem \ref{Theo2}, i.e., by using a classical limit result for triangular arrays. But for the sake of avoiding repetition, we decided to directly apply a result proven by Bertoin \cite{Be2} which is enough for our purpose. 
\end{remark}

\paragraph{Acknowledgements.}
This work is supported by the Knut and Alice Wallenberg
Foundation, a grant from the Swedish Research Council and The Swedish Foundations' starting grant from Ragnar S\"oderbergs Foundation.


\begin{thebibliography}{10}

\bibitem{asmussen2003}
S.~Asmussen, \emph{Applied probability and queues}, second ed., Applications
  of Mathematics (New York), vol.~51, Springer-Verlag, New York, 2003,
  Stochastic Modelling and Applied Probability. \MR{1978607}


\bibitem{Be1}
J.~Bertoin, \emph{Almost giant clusters for percolation on large trees with
  logarithmic heights}, J. Appl. Probab. \textbf{50} (2013), no.~3, 603--611.
  \MR{3102504}

\bibitem{Be2}
J.~Bertoin, \emph{On the non-{G}aussian fluctuations of the giant cluster for
  percolation on random recursive trees}, Electron. J. Probab. \textbf{19}
  (2014), no. 24, 15. \MR{3174836}

\bibitem{Be3}
J.~Bertoin, \emph{Sizes of the largest clusters for supercritical percolation on
  random recursive trees}, Random Structures Algorithms \textbf{44} (2014),
  no.~1, 29--44. \MR{3143589}

\bibitem{Uribe2015}
J.~Bertoin and G.~Uribe~Bravo, \emph{Supercritical percolation on large
  scale-free random trees}, Ann. Appl. Probab. \textbf{25} (2015), no.~1,
  81--103. \MR{3297766}

\bibitem{Ber2015}
G.~Berzunza, \emph{Yule processes with rare mutation and their applications to
  percolation on {$b$}-ary trees}, Electron. J. Probab. \textbf{20} (2015), no.
  43, 23. \MR{3339863}

\bibitem{Berzunza2018}
G.~Berzunza, \emph{The existence of a giant cluster for percolation on large
  crump-mode-jagers trees}, To appears in Advances in Applied Probability, arXiv:1806.10686 (2020).

\bibitem{Bollo2012}
B.~Bollob\'{a}s and O.~Riordan, \emph{Asymptotic normality of the size of the
  giant component in a random hypergraph}, Random Structures Algorithms
  \textbf{41} (2012), no.~4, 441--450. \MR{2993129}

\bibitem{Bourdon2001}
J.~Bourdon, \emph{Size and path length of {P}atricia tries: dynamical sources
  context}, Random Structures Algorithms \textbf{19} (2001), no.~3-4, 289--315,
  Analysis of algorithms (Krynica Morska, 2000). \MR{1871557}

\bibitem{Bro2012}
N.~Broutin and C.~Holmgren, \emph{The total path length of split trees}, Ann.
  Appl. Probab. \textbf{22} (2012), no.~5, 1745--1777. \MR{3025680}

\bibitem{cai2018}
X.~S. Cai and C.~Holmgren, \emph{Cutting resilient networks---complete binary
  trees}, Electron. J. Combin. \textbf{26} (2019), no.~4, Paper 4.43, 28.
  \MR{4045395}

\bibitem{cai20192}
X.~S. Cai, C.~Holmgren, L.~Devroye, and F.~Skerman, \emph{{$k$}-cut on paths
  and some trees}, Electron. J. Probab. \textbf{24} (2019), Paper No. 53, 22.
  \MR{3968715}

\bibitem{Coffman1970}
E.~G. Coffman~Jr and J.~Eve, \emph{File structures using hashing functions},
  Communications of the ACM \textbf{13} (1970), no.~7, 427--432.

\bibitem{Luc1993}
L.~Devroye, \emph{On the expected height of fringe-balanced trees}, Acta
  Inform. \textbf{30} (1993), no.~5, 459--466. \MR{1236537}

\bibitem{Luc1999}
L.~Devroye, \emph{Universal limit laws for depths in random trees}, SIAM J.
  Comput. \textbf{28} (1999), no.~2, 409--432. \MR{1634354}

\bibitem{Drmota20092}
M.~Drmota, \emph{Random trees}, SpringerWienNewYork, Vienna, 2009, An interplay
  between combinatorics and probability. \MR{2484382}

\bibitem{Drmota2009}
M.~Drmota, A.~Iksanov, M.~Moehle, and U.~Roesler, \emph{A limiting distribution
  for the number of cuts needed to isolate the root of a random recursive
  tree}, Random Structures Algorithms \textbf{34} (2009), no.~3, 319--336.
  \MR{2504401}

\bibitem{Durrett2010}
R.~Durrett, \emph{Random graph dynamics}, Cambridge Series in Statistical and
  Probabilistic Mathematics, vol.~20, Cambridge University Press, Cambridge,
  2010. \MR{2656427}

\bibitem{Fe1971}
W.~Feller, \emph{An introduction to probability theory and its applications.
  {V}ol. {II}}, Second edition, John Wiley \& Sons, Inc., New
  York-London-Sydney, 1971. \MR{0270403}

\bibitem{finkel1974}
R.~A. Finkel and J.~L. Bentley, \emph{Quad trees a data structure for retrieval
  on composite keys}, Acta informatica \textbf{4} (1974), no.~1, 1--9.

\bibitem{Flajolet2010}
P.~Flajolet, M.~Roux, and B.~Vall\'{e}e, \emph{Digital trees and memoryless
  sources: from arithmetics to analysis}, 21st {I}nternational {M}eeting on
  {P}robabilistic, {C}ombinatorial, and {A}symptotic {M}ethods in the
  {A}nalysis of {A}lgorithms ({A}of{A}'10), Discrete Math. Theor. Comput. Sci.
  Proc., AM, Assoc. Discrete Math. Theor. Comput. Sci., Nancy, 2010,
  pp.~233--260. \MR{2735344}

\bibitem{Geluk2000}
J.~L. Geluk and L.~de~Haan, \emph{Stable probability distributions and their
  domains of attraction: a direct approach}, Probab. Math. Statist. \textbf{20}
  (2000), no.~1, Acta Univ. Wratislav. No. 2246, 169--188. \MR{1785245}

\bibitem{Golds2005}
C.~Goldschmidt and J.~B. Martin, \emph{Random recursive trees and the
  {B}olthausen-{S}znitman coalescent}, Electron. J. Probab. \textbf{10} (2005),
  no. 21, 718--745. \MR{2164028}

\bibitem{Hibb1962}
T.~N. Hibbard, \emph{Some combinatorial properties of certain trees with
  applications to searching and sorting}, J. Assoc. Comput. Mach. \textbf{9}
  (1962), 13--28. \MR{0152155}

\bibitem{Hoa1962}
C.~A.~R. Hoare, \emph{Quicksort}, Comput. J. \textbf{5} (1962), 10--15.
  \MR{0142216}

\bibitem{Ceciliarxiv20102}
C.~{Holmgren}, \emph{{A Weakly 1-Stable Limiting Distribution for the Number of
  Random Records and Cuttings in Split Trees}}, arXiv e-prints (2010),
  arXiv:1005.4590.

\bibitem{Ceciliarxiv2010}
C.~{Holmgren}, \emph{{Novel Characteristics of Split Trees by use of Renewal
  Theory}}, arXiv e-prints (2010), arXiv:1005.4594.

\bibitem{Holmgren2010}
C.~Holmgren, \emph{Random records and cuttings in binary search trees}, Combin.
  Probab. Comput. \textbf{19} (2010), no.~3, 391--424. \MR{2607374}

\bibitem{Holm2011}
C.~{Holmgren}, \emph{A weakly 1-stable distribution for the number of random records
  and cuttings in split trees}, Adv. in Appl. Probab. \textbf{43} (2011),
  no.~1, 151--177. \MR{2761152}

\bibitem{Holm2012}
C.~{Holmgren}, \emph{Novel characteristic of split trees by use of renewal theory},
  Electron. J. Probab. \textbf{17} (2012), no. 5, 27. \MR{2878784}

\bibitem{HoJa2017}
C.~Holmgren and S.~Janson, \emph{Fringe trees, {C}rump-{M}ode-{J}agers
  branching processes and {$m$}-ary search trees}, Probab. Surv. \textbf{14}
  (2017), 53--154. \MR{3626585}

\bibitem{Iksanov2007}
A.~Iksanov and M.~M\"{o}hle, \emph{A probabilistic proof of a weak limit law
  for the number of cuts needed to isolate the root of a random recursive
  tree}, Electron. Comm. Probab. \textbf{12} (2007), 28--35. \MR{2407414}


\bibitem{Ja1975}
P.~Jagers, \emph{Branching processes with biological applications},
  Wiley-Interscience [John Wiley \& Sons], London-New York-Sydney, 1975, Wiley
  Series in Probability and Mathematical Statistics---Applied Probability and
  Statistics. \MR{0488341}

\bibitem{Jans2004}
S.~Janson, \emph{Random records and cuttings in complete binary trees},
  Mathematics and computer science. {III}, Trends Math., Birkh\"{a}user, Basel,
  2004, pp.~241--253. \MR{2090513}

\bibitem{janson2018}
S.~Janson, \emph{Random recursive trees and preferential attachment trees are
  random split trees}, Combinatorics, Probability and Computing (2018), 1--19.

\bibitem{Kall2002}
O.~Kallenberg, \emph{Foundations of modern probability}, second ed.,
  Probability and its Applications (New York), Springer-Verlag, New York, 2002.
  \MR{1876169}

\bibitem{luria1943}
S.~E. Luria and M.~Delbr{\"u}ck, \emph{Mutations of bacteria from virus
  sensitivity to virus resistance}, Genetics \textbf{28} (1943), no.~6, 491.

\bibitem{Mah1986}
H.~M. Mahmoud, \emph{On the average internal path length of {$m$}-ary search
  trees}, Acta Inform. \textbf{23} (1986), no.~1, 111--117. \MR{845626}

\bibitem{Mah1989}
H.~M. Mahmoud and B.~Pittel, \emph{Analysis of the space of search trees under
  the random insertion algorithm}, J. Algorithms \textbf{10} (1989), no.~1,
  52--75. \MR{987097}

\bibitem{Meir1970}
A.~Meir and J.~W. Moon, \emph{Cutting down random trees}, J. Austral. Math.
  Soc. \textbf{11} (1970), 313--324. \MR{0284370}

\bibitem{Moohle2005}
M.~M\"{o}hle, \emph{Convergence results for compound {P}oisson distributions
  and applications to the standard {L}uria-{D}elbr\"{u}ck distribution}, J.
  Appl. Probab. \textbf{42} (2005), no.~3, 620--631. \MR{2157509}

\bibitem{Nein1999}
R.~Neininger and L.~R\"{u}schendorf, \emph{On the internal path length of
  {$d$}-dimensional quad trees}, Random Structures Algorithms \textbf{15}
  (1999), no.~1, 25--41. \MR{1698407}

\bibitem{Pitman1999}
J.~Pitman, \emph{Coalescent random forests}, J. Combin. Theory Ser. A
  \textbf{85} (1999), no.~2, 165--193. \MR{1673928}

\bibitem{Pyke1965}
R.~Pyke, \emph{Spacings. ({W}ith discussion.)}, J. Roy. Statist. Soc. Ser. B
  \textbf{27} (1965), 395--449. \MR{0216622}

\bibitem{Ro2001}
U.~R\"{o}sler, \emph{On the analysis of stochastic divide and conquer
  algorithms}, Algorithmica \textbf{29} (2001), no.~1-2, 238--261, Average-case
  analysis of algorithms (Princeton, NJ, 1998). \MR{1887306}

\bibitem{Schweinsberg2012}
J.~Schweinsberg, \emph{Dynamics of the evolving {B}olthausen-{S}znitman
  coalecent}, Electron. J. Probab. \textbf{17} (2012), no. 91, 50. \MR{2988406}

\bibitem{Sei2013}
T.~G. Seierstad, \emph{On the normality of giant components}, Random Structures
  Algorithms \textbf{43} (2013), no.~4, 452--485. \MR{3124692}

\bibitem{Ste1970}
V.~E. Stepanov, \emph{Phase transitions in random graphs}, Teor. Verojatnost. i
  Primenen. \textbf{15} (1970), 200--216. \MR{0270407}

\bibitem{Walker1976}
A.~Walker and D.~Wood, \emph{Locally balanced binary trees}, The Computer
  Journal \textbf{19} (1976), no.~4, 322--325.

\end{thebibliography}

\providecommand{\bysame}{\leavevmode\hbox to3em{\hrulefill}\thinspace}
\providecommand{\MR}{\relax\ifhmode\unskip\space\fi MR }
\providecommand{\MRhref}[2]{%
  \href{http://www.ams.org/mathscinet-getitem?mr=#1}{#2}
}
\providecommand{\href}[2]{#2}


\begin{appendices}

\section{Distances in split trees} \label{sec3}

The purpose of this section is to establish some general results on the distribution of the
distances between uniform chosen vertices and uniformly chosen balls in $T_{n}^{{\rm sp}}$ when $n \rightarrow \infty$. The results can be seen as a complement (or extension) of those of Devroye
\cite{Luc1999} and Holmgren \cite{Holm2012}. Let $H_{n}$  be the height of $T_{n}^{{\rm sp}}$, i.e., the maximal distance between the root and any leaf in $T_{n}^{{\rm sp}}$. We deduce the following moment estimate for $H_{n}$. For $y \in \mathbb{R}$, recall that $\lceil y \rceil$ denotes the least integer greater than or equal to $y$. Similarly, $\lfloor y \rfloor$ denotes the greatest integer less than or equal to $y$

\begin{lemma}  \label{lemma7}
Assume that Condition \ref{Cond1} is fulfilled. For all $r > 0$, we have that $\sup_{n \geq 1} \mathbb{E}[H_{n}^{r}] \ln^{-r} n < \infty$.
\end{lemma}

\begin{proof}
We claim that for all $r > 0$ there exists $c_{r} >0$ such that 
\begin{eqnarray} \label{eq42}
\lim_{n \rightarrow \infty} n^{r}\mathbb{P}(H_{n} \geq (3s_{1}+4)\lfloor c_{r} \ln n \rfloor) = 0.
\end{eqnarray}

\noindent Then, the bound $H_{n} \leq n$ implies that 
\begin{eqnarray*}
\mathbb{E}[H_{n}^{r}] \leq (3s_{1} + 4) c_{r} \ln^{r} n + n^{r} \mathbb{P}(H_{n} \geq (3s_{1}+4)\lfloor c_{r} \ln n \rfloor)
\end{eqnarray*}

\noindent which combined with (\ref{eq42}) allows us to conclude with the proof of Lemma \ref{lemma7}. 

Therefore, it only remains to prove the claim in (\ref{eq42}).  Devroye \cite{Luc1999} has shown that for integers $0 \leq k^{\prime} \leq k$ and $l = k^{\prime}(s_{1}+1)$ such that $s_{1}k^{\prime} < l$, and real numbers $t, t^{\prime} > 0$, we have that
\begin{eqnarray} \label{eq43}
\mathbb{P}(H_{n} \geq k + 3l) \leq 2b^{-k} + b^{k}(ne)^{t}b^{2kt/l}m(t)^{k} + b^{k}(s_{1}(k-k^{\prime}+1)e)^{t^{\prime}} b^{2kt^{\prime}/l}m(t^{\prime})^{k^{\prime}},
\end{eqnarray}

\noindent where $m(t) = \mathbb{E}[V_{1}^{t}]$ for $t >0$; see proof of \cite[Theorem 1]{Luc1999} for details. Then consider the estimate in (\ref{eq43}) with $k = k^{\prime} =  \lfloor c_{r} \ln n \rfloor$ and $l = k^{\prime}(s_{1}+1)$. Then choose $t, t^{\prime} \geq 0$ large enough such that $b m(t) < 1$ and $b m(t^{\prime}) < 1$. This is possible because $\mathbb{P}(V_{1} = 1) =0$ by Condition \ref{Cond1}, and thus, $m(t) \rightarrow 0$ as $t \rightarrow \infty$; see \cite[Lemma 1]{Luc1999}. Finally, (\ref{eq42}) follows immediately by taking $c_{r} > \max(r/\ln b, -(r+t)/\ln (bm(t)), -r/\ln (bm(t^{\prime})))$.
\end{proof}

For each fixed $n \in \mathbb{N}$, let $b_{1}$ be a uniformly distributed ball on the set $\{1,
\dots, n\}$ of balls in $T_{n}^{{\rm sp}}$. Recall that we denote by $D_{n}(b_{1})$ the height (or
depth) of the ball $b_{1}$ in $T_{n}^{{\rm sp}}$, i.e., the number of edges of $T_{n}^{{\rm sp}}$ which are between the root and the vertex where the ball $b_{1}$ is stored. 

\begin{lemma} \label{lemma5}
Assume that Condition \ref{Cond1} is fulfilled. 
\begin{itemize}
\item[(i)] Recall that $\varpi: \mathbb{R} \rightarrow \mathbb{R}$ denotes the function in (\ref{eq10}). Then $ \displaystyle 
\mathbb{E}[D_{n}(b_{1})] = \mu^{-1} \ln n + \varpi(\ln n) +o(1)$. 

\item[(ii)] We also have $ \displaystyle 
\mathbb{E}[(D_{n}(b_{1}) - \mu^{-1}\ln n)^{2}] = \mu^{-3} \sigma^{2} \ln n +o(\ln n)$. 

\item[(iii)] Furthermore, $ \displaystyle 
 \mathbb{E} \left[  \left| D_{n}(b_{1}) - \mu^{-1} \ln n  \right|^{3}\right]  = O (\ln^{\frac{3}{2}}n )$. 

\item[(iv)] As a consequence, we conclude that
$ \displaystyle 
\lim_{n \rightarrow \infty} D_{n}(b_{1}) (\ln n)^{-1} = 1/ \mu$, in probability.
\end{itemize} 
\end{lemma}

\begin{proof}
We observe that $\mathbb{E}[D_{n}(b_{1})] = n^{-1}\mathbb{E} [  \sum_{i=1}^{n} D_{n}(i) ] = n^{-1} \mathbb{E}\left[  \Psi(T_{n}^{{\rm sp}}) \right]$. Then (i) follows immediately from the result in (\ref{eq10}). Turning our attention to the proof of (ii), we write
\begin{eqnarray} \label{eq23}
\mathbb{E}[(D_{n}(b_{1})- \mu^{-1}\ln n)^{2}] =  n^{-1} \mathbb{E}\Big[ \sum_{i=1}^{n}(D_{n}(i)  - \mu^{-1}\ln n)^{2} \Big].
\end{eqnarray}

\noindent Holmgren \cite[Proposition 1.1]{Holm2012} has shown that for $j \leq j^{\prime}$ we have that $D_{n}(j) \leq D_{n}(j^{\prime})$ in the stochastic sense. Moreover, $D_{j}(j) \leq D_{n}(j)$, for $n \geq j$, since a ball with label $j$ only move downward during the splitting process when new balls are added to the tree. Furthermore, it follows from \cite[Theorem 1.3]{Holm2012} that
\begin{eqnarray} \label{eq18}
 \mathbb{E}\left[ (D_{n}(j)  - \mu^{-1}\ln n)^{2} \right] = \mu^{-3} \sigma^{2}  \ln n + o(\ln n), \hspace*{5mm} \text{uniformly for} \hspace*{5mm} \left \lceil n \ln^{-1}n \right \rceil \leq j \leq n.
\end{eqnarray}

\noindent Since $D_{n}(j)$ can be stochastically dominated from above and below by $D_{n}(n)$ and $D_{j}(j)$, for $1 \leq j \leq n$, respectively, we deduce that
\begin{eqnarray} \label{eq20}
\mathbb{E}\left[ (D_{n}(j)  - \mu^{-1}\ln n)^{2} \right] & \leq &  \mathbb{E}\left[ (D_{n}(n)  - \mu^{-1}\ln n)^{2} \right] +  \mathbb{E}\left[ (D_{j}(j)  - \mu^{-1}\ln n)^{2} \right] \nonumber \\
& \leq &  \mathbb{E}\left[ (D_{n}(n)  - \mu^{-1}\ln n)^{2} \right] +  4\mathbb{E}\left[ (D_{j}(j)  - \mu^{-1}\ln j)^{2} \right]  + 4 \mu^{-2} \left|  \ln j-  \ln n   \right|^{2} \nonumber \\
& = & o(\ln^{2} n), 
\end{eqnarray}

\noindent uniformly for $\left \lceil  n \ln^{-2}n \right \rceil \leq j < \left \lceil n \ln^{-1} n \right \rceil$; We have used the inequality $|x-y|^{2} \leq 4x^{2} + 4y^{2}$  for $x,y \geq 0$. On the other hand, Lemma \ref{lemma7} implies that
\begin{eqnarray} \label{eq21}
\mathbb{E}\left[ (D_{n}(j)  - \mu^{-1}\ln n)^{2} \right] & \leq &  4 \mathbb{E}\left[H_{n}^{2} \right] + 4\mu^{-2} \ln^{2}n = o(\ln^{3}n)
\end{eqnarray}

\noindent uniformly for $1 \leq j < \left \lceil n \ln^{-2} n \right \rceil$. Then the combination (\ref{eq23}), (\ref{eq18}), (\ref{eq20}) and (\ref{eq21}) imply (ii). 

We now prove (iii). We observe that
\begin{eqnarray} \label{eq24}
\mathbb{E} \left[  \left| D_{n}(b_{1}) - \mu^{-1} \ln n   \right|^{3}\right] & = & n^{-1} \mathbb{E}\Big[ \sum_{i=1}^{n}\left| D_{n}(i)  - \mu^{-1}\ln n \right|^{3} \Big].
\end{eqnarray}

\noindent We also observe that 
\begin{eqnarray*}
 \mathbb{E} \left[  \left| D_{n}(j) - \mu^{-1} \ln n    \right|^{3}\right] & \leq &  \mathbb{E} \left[  \left| D_{n}(n) - \mu^{-1} \ln n   \right|^{3}\right] +  \mathbb{E} \left[  \left| D_{j}(j) - \mu^{-1} \ln n   \right|^{3}\right] \\
& \leq &  \mathbb{E} \left[  \left| D_{n}(n) - \mu^{-1} \ln n   \right|^{3}\right] +  8\mathbb{E} \left[  \left| D_{j}(j) - \mu^{-1} \ln j   \right|^{3}\right]  + 8 \mu^{-3} \left|  \ln j-  \ln n  \right|^{3},
\end{eqnarray*}

\noindent for $1 \leq j \leq n$; we have used the inequality $|x-y|^{3} \leq 8 x^{3} + 8 y^{3}$  for $x,y \geq 0$. From Holmgren \cite[equation (3.62)]{Holm2012} we deduce that
\begin{eqnarray} \label{eq15}
 \mathbb{E} \left[  \left| D_{n}(j) - \mu^{-1} \ln n   \right|^{3}\right] = O \left(\ln^{\frac{3}{2}} n \right) , \hspace*{4mm} \text{uniformly for} \hspace*{4mm} \left \lceil n \ln^{-2} n  \right \rceil \leq j \leq n.
\end{eqnarray}

We observe that $\mathbb{E}\left[ \left| D_{n}(j)  - \mu^{-1}\ln n \right|^{3} \right] \leq  8 \mathbb{E}\left[H_{n}^{3} \right] + 8\mu^{-3} \ln^{3}n$, uniformly for $1 \leq j < \left \lceil n \ln^{-2} n \right \rceil$. Then Lemma \ref{lemma7} implies that
\begin{eqnarray} \label{eq22}
 \mathbb{E} \left[  \left| D_{n}(j) - \mu^{-1} \ln n   \right|^{3}\right] = O \left( \ln^{3} n \right), \hspace*{4mm} \text{uniformly for} \hspace*{4mm} 1 \leq j < \left \lceil n \ln^{-2} n  \right \rceil.
\end{eqnarray}

\noindent Therefore, (iii) follows from (\ref{eq24}), (\ref{eq15}) and (\ref{eq22}).

The point (iv) follows immediately from (ii)  and a standard application of Chebyshev's inequality. 
\end{proof}

We turn our attention to the height of a random chosen vertex in $T_{n}^{{\rm sp}}$. For each fixed $n
\in \mathbb{N}$, let $u_{1}$ be a uniformly distributed vertex on the random split tree $T_{n}^{{\rm sp}}$ with $n$ balls. Recall that we denote by $d_{n}(u_{1})$ the height of the vertex
$u_{1}$ in $T_{n}^{{\rm sp}}$, i.e., the minimal number of edges of $T_{n}^{{\rm sp}}$ which are needed to connect the root and $u_{1}$. 

\begin{lemma} \label{lemma1}
Assume that Conditions \ref{Cond1} and \ref{Cond2} are fulfilled. 
\begin{itemize}
\item[(i)] Recall that $\zeta \in \mathbb{R}$ is the constant in Condition \ref{Cond2}. Then $ \displaystyle 
\mathbb{E}[d_{n}(u_{1})] = \mu^{-1} \ln n + \zeta \alpha^{-1}+o(1)$. 

\item[(ii)] We also have $ \displaystyle 
\mathbb{E}[(d_{n}(u_{1}) - \mu^{-1}\ln n)^{2}] = \mu^{-3} \sigma^{2} \ln n +o(\ln n)$. 

\item[(iii)] Furthermore, for $\delta > 1/2 - \varepsilon$, $ \displaystyle 
\mathbb{E} \left[  \left| d_{n}(u_{1}) - \mu^{-1} \ln n  \right|^{3}\right] = O (  \ln^{\frac{3}{2}+\delta} n )$. where $\varepsilon >0$ is the constant that appears in Condition \ref{Cond2}.

\item[(iv)] As a consequence, we conclude that
$ \displaystyle 
\lim_{n \rightarrow \infty} d_{n}(u_{1})(\ln n)^{-1} = 1/\mu$, in probability.
\end{itemize}
\end{lemma}

\begin{proof}
We observe that 
\begin{eqnarray*}
\mathbb{E}[d_{n}(u_{1})] = \mathbb{E}\left[ \frac{1}{N} \sum_{u \in T_{n}^{{\rm sp}}} d_{n}(u) \right] = \frac{1}{\mathbb{E}[N]} \mathbb{E}\left[  \Upsilon(T_{n}^{{\rm sp}}) \right] + \mathbb{E}\left[ \left(\frac{1}{N} - \frac{1}{\mathbb{E}[N]} \right) \Upsilon(T_{n}^{{\rm sp}}) \right].
\end{eqnarray*}

\noindent It should be clear that (i) follows from Condition \ref{Cond2} and the result in (\ref{eq5}) by showing that
\begin{eqnarray} \label{eq6}
 \mathbb{E}\left[ \left(\frac{1}{N} - \frac{1}{\mathbb{E}[N]} \right) \Upsilon(T_{n}^{{\rm sp}}) \right] = o(1).
\end{eqnarray}

\noindent Therefore, we focus on the proof of (\ref{eq6}). 

We notice that 
\begin{eqnarray*}
\left| \frac{1}{N} - \frac{1}{\mathbb{E}[N]} \right| \Upsilon(T_{n}^{{\rm sp}}) = \left | \frac{N - \mathbb{E}[N]}{N \mathbb{E}[N] } \right| \sum_{u \in T_{n}^{{\rm sp}}} d_{n}(u) \leq   \frac{ \left| N - \mathbb{E}[N] \right| H_{n}}{\mathbb{E}[N]},
\end{eqnarray*}

\noindent where we recall that $H_{n}$ denotes the height of $T_{n}^{{\rm sp}}$. An application of the Cauchy–Schwarz inequality shows that
\begin{eqnarray*}
 \mathbb{E}\left[ \left(\frac{1}{N} - \frac{1}{\mathbb{E}[N]} \right) \Upsilon(T_{n}^{{\rm sp}}) \right]  \leq  \mathbb{E}^{-1}[N] (Var(N))^{\frac{1}{2}} \mathbb{E}^{1/2}[H_{n}^{2}] = o(1),
\end{eqnarray*}

\noindent where in the last step we used Remark \ref{remark1}, Condition \ref{Cond2} and Lemma \ref{lemma7}.

We turn our attention to the proof of (ii). We notice that 
\begin{align*}
\mathbb{E}[(d_{n}(u_{1}) - \mu^{-1}\ln n)^{2}] & =  \mathbb{E}\left[\frac{1}{N} \sum_{u \in T_{n}^{{\rm sp}}}(d_{n}(u)  - \mu^{-1}\ln n)^{2} \right] \\
& =  \frac{1}{\mathbb{E}[N]}  \mathbb{E}\left[\sum_{u \in T_{n}^{{\rm sp}}}(d_{n}(u)  - \mu^{-1}\ln n)^{2} \right] \\
& \hspace*{10mm} + \mathbb{E}\left[ \left(\frac{1}{N} - \frac{1}{\mathbb{E}[N]}\right) \sum_{u \in T_{n}^{{\rm sp}}}(d_{n}(u) - \mu^{-1}\ln n)^{2} \right].
\end{align*}

\noindent Holmgren \cite[Corollary 2.1]{Ceciliarxiv2010} has shown that 
\begin{eqnarray*}
\mathbb{E}\left[\sum_{u \in T_{n}^{{\rm sp}}}(d_{n}(u)  - \mu^{-1}\ln n)^{2} \right] = \alpha n \mu^{-3} \sigma^{2} \ln n + o(n\ln n).
\end{eqnarray*}

\noindent Then (ii) follows from Condition \ref{Cond2} and Remark \ref{remark1} by providing that 
\begin{eqnarray*}
 \mathbb{E}\left[ \left(\frac{1}{N} - \frac{1}{\mathbb{E}[N]}\right) \sum_{u \in T_{n}^{{\rm sp}}}(d_{n}(u)  - \mu^{-1}\ln n)^{2} \right] = o(\ln n).
\end{eqnarray*}

\noindent This is proved from similar arguments as in the proof of (\ref{eq6}). The details are omitted. 

We continue with the proof of (iii). We have that
\begin{align*}
\mathbb{E} \left[|d_{n}(u_{1}) - \mu^{-1}\ln n|^{3} \right] & =  \mathbb{E}\left[\frac{1}{N} \sum_{u \in T_{n}^{{\rm sp}}}|d_{n}(u)  - \mu^{-1}\ln n|^{3} \right] \\
& =  \frac{1}{\mathbb{E}[N]}  \mathbb{E}\left[\sum_{u \in T_{n}^{{\rm sp}}}|d_{n}(u)  - \mu^{-1}\ln n|^{3} \right] \\
& \hspace*{10mm} + \mathbb{E}\left[ \left(\frac{1}{N} - \frac{1}{\mathbb{E}[N]}\right) \sum_{u \in T_{n}^{{\rm sp}}}|d_{n}(u) - \mu^{-1}\ln n|^{3} \right].
\end{align*}

\noindent Suppose that we have proven that 
\begin{eqnarray} \label{eq63}
 \mathbb{E}\left[ \frac{1}{n}\sum_{u \in T_{n}^{{\rm sp}}}|d_{n}(u)  - \mu^{-1}\ln n|^{3} \right] = O \left( \ln^{ \frac{3}{2} + \delta} n \right),
\end{eqnarray}

\noindent for $\delta > 1/2 - \varepsilon$. Then (iii) follows from Condition \ref{Cond2} and by showing that
\begin{eqnarray*}
 \mathbb{E}\left[ \left(\frac{1}{N} - \frac{1}{\mathbb{E}[N]}\right) \sum_{u \in T_{n}^{{\rm sp}}}|d_{n}(u)  - \mu^{-1}\ln n|^{3} \right] \ln^{- \frac{3}{2} - \delta} n = o(1), \hspace*{5mm} \text{for} \hspace*{3mm} \delta > 1/2 - \varepsilon. 
\end{eqnarray*}

\noindent This can be proved by using similar arguments as in the proof of (\ref{eq6}) and the details are omitted. 

Finally, we check that (\ref{eq63}) holds. For $\delta > 1/2 - \varepsilon$ and $C > 0$, we notice that 
\begin{align*}
& \mathbb{E}\left[ \sum_{u \in T_{n}^{{\rm sp}}}|d_{n}(u)  - \mu^{-1}\ln n|^{3} \mathds{1}_{\left\{ |d_{n}(u)  - \mu^{-1}\ln n| > \ln^{\frac{1}{2}+\frac{\delta}{3}}n \right\}} \right] \\
& \hspace*{10mm} \leq 8 \mathbb{E}\left[ \left( H_{n}^{3}+ \mu^{-3}\ln^{3}n \right) \sum_{u \in T_{n}^{{\rm sp}}}\mathds{1}_{\left\{ |d_{n}(u)  - \mu^{-1}\ln n| > \ln^{\frac{1}{2}+\frac{\delta}{3}}n \right\}} \right] \\
& \hspace*{10mm} \leq 8(C^{3}+\mu^{-3})(\ln^{3}n) \mathbb{E}\left[ \sum_{u \in T_{n}^{{\rm sp}}}\mathds{1}_{\left\{ |d_{n}(u)  - \mu^{-1}\ln n| > \ln^{\frac{1}{2}+\frac{\delta}{3}}n \right\}} \right] + 8 n^{4} \mathbb{P}(H_{n} \geq C \ln n).
\end{align*}

\noindent On the one hand, Holmgren \cite[Theorem 1.2]{Holm2012} has shown that 
\begin{eqnarray*}
(C^{3}+\mu^{-3})(\ln^{3}n) \mathbb{E}\left[ \sum_{u \in T_{n}^{{\rm sp}}}\mathds{1}_{\left\{ |d_{n}(u)  - \mu^{-1}\ln n| > \ln^{\frac{1}{2}+\frac{\delta}{3}}n \right\}} \right] = o\left(n \ln^{\frac{3}{2} + \delta} n\right)
\end{eqnarray*}

\noindent (It is important to point out that the sum inside the expectation is what Holmgren \cite[Theorem 1.2]{Holm2012} calls the number of bad vertices in $T_{n}^{{\rm sp}}$). On the other hand, by (\ref{eq42}), we can choose $C > 0$ such that $8 n^{4} \mathbb{P}(H_{n} \geq C \ln n) = o (n \ln^{\frac{3}{2} + \delta} n )$. Hence,
\begin{eqnarray} \label{eq64}
\mathbb{E}\left[ \sum_{u \in T_{n}^{{\rm sp}}}|d_{n}(u)  - \mu^{-1}\ln n|^{3} \mathds{1}_{\left\{ |d_{n}(u)  - \mu^{-1}\ln n| > \ln^{\frac{1}{2}+\frac{\delta}{3}}n \right\}} \right] =  o\left(n \ln^{\frac{3}{2} + \delta} n\right).
\end{eqnarray}

\noindent We also note that 
\begin{eqnarray*}
\mathbb{E}\left[ \sum_{u \in T_{n}^{{\rm sp}}}|d_{n}(u)  - \mu^{-1}\ln n|^{3} \mathds{1}_{\left\{
            |d_{n}(u)  - \mu^{-1}\ln n| \leq \ln^{\frac{1}{2}+\frac{\delta}{3}}n \right\}} \right] = O \left(
n \ln^{\frac{3}{2} + \delta} n \right),
\end{eqnarray*}

\noindent which combined with (\ref{eq64}) implies (\ref{eq63}). 

The point (iv) follows immediately from (ii)  and a standard application of Chebyshev's inequality.
\end{proof}

Recall the labeling of the balls induced by the split tree generating algorithm explained in Section \ref{sec1}. Let $v$ and $v^{\prime}$ be the vertices in $T_{n}^{{\rm sp}}$ where the balls labeled $j$ and $j^{\prime}$ are located, respectively. We call the vertex $v \wedge v^{\prime}$ at which the paths in $T_{n}^{{\rm sp}}$ from the vertices $v$ and $v^{\prime}$ to the root intersect the last common ancestor of the balls with labels $j$ and $j^{\prime}$.
For simplicity, we denote by $j \wedge j^{\prime}$ a last common ancestor of the balls $j$ and $j^{\prime}$ (notice that $j \wedge j^{\prime}$ is not necessary unique). Let $D_{n}(j \wedge j^{\prime})$ be the height of $j \wedge j^{\prime}$ when all $n$ balls have been inserted. 

\begin{lemma} \label{lemma2}
Assume that Condition \ref{Cond1} is fulfilled. For $n \in \mathbb{N}$ fixed, let $b_{1}$ and $b_{2}$ denote two independent uniformly distributed random ball labels in $T_{n}^{{\rm sp}}$. Let $h: \mathbb{N} \rightarrow \mathbb{R}_{+}$ be some function such that $\lim_{n \rightarrow \infty} h(n) = \infty$. We have that
\begin{eqnarray*}
\lim_{n \rightarrow \infty} \frac{D_{n}(b_{1} \wedge b_{2})}{h(n)} = 0, \hspace*{5mm} \text{in probability}.
\end{eqnarray*}
\end{lemma}

\begin{proof}
For $\delta > 0$, we notice that $D_{n}(b_{1} \wedge b_{2}) \geq \delta h(n)$ when both balls $b_{1}$ and $b_{2}$ lie in the same sub-tree and the height of the last common ancestor related to this sub-tree has to be greater than $\delta h(n)$. For $1 \leq i \leq b^{\lceil \delta h(n) \rceil}$, let $v_{i}$ be a vertex in $T_{n}^{{\rm sp}}$ at height $\lceil \delta h(n) \rceil$ and let $n_{i}$ be the number of balls stored at the sub-tree rooted at $v_{i}$; note that those balls have depth greater than $\delta h(n)$. Since $b_{1}$ and $b_{2}$ denote two independent uniformly distributed random ball in $T_{n}^{{\rm sp}}$, we have that
\begin{eqnarray} \label{eq8}
\mathbb{P}(D_{n}(b_{1} \wedge b_{2}) \geq \delta h(n)) \leq \mathbb{E}\left[\sum_{i=1}^{b^{\lceil \delta h(n) \rceil}} \left(\frac{n_{i}}{n} \right)^{2} \right] = n^{-2} \sum_{i=1}^{b^{\lceil \delta h(n) \rceil}} \mathbb{E}\left[n_{i}^{2} \right].
\end{eqnarray}

\noindent On the other hand, Condition \ref{Cond1} and the standard inequality \cite[equation (1.10)]{Holm2011} for subtrees sizes in split-trees (we refer to the estimation (\ref{eq7}) for a formal proof) imply that
\begin{eqnarray} \label{eq107}
\mathbb{E}\left[n_{i}^{2} \right] = n^{2} \mathbb{E}^{\lceil \delta h(n) \rceil} \left[
    V_{1}^{2}\right] + o(n^{2}\ln^{-k}n),
\end{eqnarray}

\noindent for an arbitrary $k \geq 0$ and where $\mathbb{E}[ V_{1}^{2}] < 1/b$. This identity combined with (\ref{eq8}) clearly implies our claim. 
\end{proof}

Let $v$ and $v^{\prime}$ be two vertices in the split tree $T_{n}^{{\rm sp}}$. We denote by $d_{n}(v \wedge v^{\prime})$ the height of the last common ancestor $v \wedge v^{\prime}$ of the vertices $v$ and $v^{\prime}$ in the tree $T_{n}^{{\rm sp}}$. 

\begin{lemma} \label{cor1}
Assume that Conditions \ref{Cond1} and \ref{Cond2} are fulfilled. For $n \in \mathbb{N}$ fixed, let $u_{1}$ and $u_{2}$ denote two independent uniformly distributed random vertices in $T_{n}^{{\rm sp}}$. Let $h: \mathbb{N} \rightarrow \mathbb{R}_{+}$ be some function with $\lim_{n \rightarrow \infty} h(n) = \infty$. We have that
\begin{eqnarray*}
\lim_{n \rightarrow \infty} \frac{d_{n}(u_{1} \wedge u_{2})}{h(n)} = 0, \hspace*{5mm} \text{in probability}.
\end{eqnarray*}
\end{lemma}

\begin{proof}
We follow a similar argument as in the proof Lemma \ref{lemma2}.  For $\delta > 0$, note that $d_{n}(u_{1} \wedge u_{2}) \geq \delta h(n)$ when both vertices lie in the same sub-tree and the height of the last common ancestor related to this sub-tree has to be greater than $\delta h(n)$. For $1 \leq i \leq b^{\lceil \delta h(n) \rceil}$, let $v_{i}$ be a vertex in $T_{n}^{{\rm sp}}$ at height $\lceil \delta h(n) \rceil$ and let $N_{i}$ be the number of vertices of the sub-tree rooted at $v_{i}$. Since $u_{1}$ and $u_{2}$ are two independent uniformly distributed random vertices in $T_{n}^{{\rm sp}}$, we have that
\begin{eqnarray} \label{eq105}
\mathbb{P}(d_{n}(u_{1} \wedge u_{2}) \geq \delta h(n)) & \leq & \mathbb{E}\left[\sum_{i=1}^{b^{\lceil \delta h(n) \rceil}} \left(\frac{N_{i}}{N} \right)^{2} \right] \nonumber \\
& = & \mathbb{E}\left[ \frac{N^{2}- \mathbb{E}^{2}[N]}{N^{2} \mathbb{E}^{2}[N]}  \sum_{i=1}^{b^{\lceil \delta h(n) \rceil}} N_{i}^{2} \right] + \mathbb{E}\left[\sum_{i=1}^{b^{\lceil \delta h(n) \rceil}} \left(\frac{N_{i}}{\mathbb{E}[N]} \right)^{2} \right].
\end{eqnarray}

We analyze the first term at the right-hand side of (\ref{eq105}). Note that $\sum_{i=1}^{b^{\lceil \delta h(n) \rceil}} N_{i}^{2} \leq N^{2}$. Then Condition \ref{Cond2} and Remark \ref{remark1} imply that
\begin{eqnarray} \label{eq106}
\mathbb{E}\left[ \frac{N^{2}- \mathbb{E}^{2}[N]}{N^{2} \mathbb{E}^{2}[N]}  \sum_{i=1}^{b^{\lceil \delta h(n) \rceil}} N_{i}^{2} \right] \leq \frac{Var(N)}{\mathbb{E}^{2}[N]} = o(1). 
\end{eqnarray}

\noindent We now focus in the second term at the right-hand side of (\ref{eq105}). Note that Condition \ref{Cond2} and Remark \ref{remark1} imply that $\mathbb{E}[N_{i}^{2}] = \mathbb{E}\left[ Var(N_{i} | n_{i}) + \mathbb{E}^{2}[N_{i}|n_{i}] \right] = O(\mathbb{E}[n_{i}^{2}])$, where we have used the well-known formula $Var(N_{i}) = \mathbb{E}[Var(N_{i}| n_{i})] + Var(\mathbb{E}[N_{i}|n_{i}])$. Hence the previous estimate, the inequality (\ref{eq107}) and Condition \ref{Cond2} allow us to conclude that
\begin{eqnarray} \label{eq108}
\mathbb{E}\left[\sum_{i=1}^{b^{\lceil \delta h(n) \rceil}} \left(\frac{N_{i}}{\mathbb{E}[N]} \right)^{2} \right] = o(1).
\end{eqnarray}

\noindent Finally, our claim follows by applying (\ref{eq106}) and (\ref{eq108}) into (\ref{eq105}).  
\end{proof}

We complete this section by stating a corollary of the previous lemmas. Let $u_{1}$ and $u_{2}$ be two independent uniformly chosen vertices in $T_{n}^{{\rm sp}}$. We
write $d_{n}(u_{1},u_{2})$ for the number of edges of $T_{n}^{{\rm sp}}$ which are needed to connect
the root, $u_{1}$ and $u_{2}$. Similarly, let $b_{1}$ and $b_{2}$ be two independent uniformly chosen balls in $T_{n}^{{\rm sp}}$. We write $D_{n}(b_{1}, b_{2})$ for the number of edges of $T_{n}^{{\rm sp}}$ which are needed to connect the root, and vertices where the balls $b_{1}$ and $b_{2}$ are stored.

\begin{corollary} \label{cor2}
Assume that Condition \ref{Cond1} is fulfilled. We have that
\begin{eqnarray*}
\lim_{n \rightarrow \infty} \frac{D_{n}(b_{1}, b_{2})}{\ln n} = \frac{2}{ \mu} \hspace*{5mm} \text{in probability.}
\end{eqnarray*}

\noindent If we further assume that Condition \ref{Cond2} is also satisfied. We have that
\begin{eqnarray*}
\lim_{n \rightarrow \infty} \frac{d_{n}(u_{1},u_{2})}{\ln n} = \frac{2}{ \mu} \hspace*{5mm} \text{in probability.}
\end{eqnarray*}
\end{corollary}

\begin{proof}
We note that $D_{n}(b_{1}, b_{2}) = D_{n}(b_{1}) + D_{n}(b_{2}) - D_{n}(b_{1} \wedge b_{2})$, where $D_{n}(b_{1})$ has the same distribution as $D_{n}(b_{2})$. Therefore, the first result is a direct consequence of Lemma \ref{lemma5} and Lemma \ref{lemma2}. The proof of the second claim follows from a similar argument by using Lemma \ref{lemma1} and Lemma \ref{cor1}. 
\end{proof}

\section{Proof of Theorem \ref{Theo3}} \label{appendix1}

In this section, we prove Theorem \ref{Theo3} which is an important ingredient in the proof of Theorem \ref{Theo2}. For $1 \leq i \leq m_{n}$, we denote by $\mathcal{F}_{i}$ the $\sigma$-field generated by $(n_{v}, d_{n}(v) \leq i)$. Recall from the beginning of Section \ref{sec2} that for a vertex $v \in T_{n}^{{\rm sp}}$ that is at height $d_{n}(v) = i$, we write  $(W_{v,k}, k=1, \dots, i)$ for a sequence of i.i.d.\ random variables on $[0,1]$ given by the split vectors associated with the vertices on the unique path from $v$ to the root. We denote by $\mathcal{G}_{i}$ the $\sigma$-field generated by $((W_{v,k}, k=1, \dots, i): d_{n}(v) = i)$. Recall the notation $\varepsilon_{v}$ in (\ref{eqnew1}) and write 
\begin{eqnarray} \label{eq54}
\hat{n}_{v} \coloneqq n \prod_{k=1}^{i} W_{v,k}, \hspace*{5mm}
\text{and} \hspace*{5mm} \hat{\xi}_{v}:=e^{-\frac{c}{\mu} }\frac{ \ln n}{n} \hat{n}_{v} \varepsilon_{v}.
\end{eqnarray} 

\noindent Note that $\mathcal{G}_{i}$ is equivalent to the $\sigma$-field generated by $(\hat{n}_{v},d_{n}(v) \leq i)$. 

We present now some crucial lemmas that are used in the proof of Theorem \ref{Theo3}. Recall the notation $m_{n} = \lfloor \beta \log_{b} \ln n  \rfloor$ for $\beta > 0$. Furthermore, through this section we assume that $\beta$ is large enough. For the sake of simplicity, we introduce the following notation. For any constants $\theta, x >0$,
\begin{eqnarray*}
 \alpha_{n}^{\prime}\coloneqq\frac{\ln n}{n} \sum_{d_{n}(v) = m_{n}} \hat{n}_{v} e^{- \frac{c}{\mu} \frac{\ln \hat{n}_{v}}{\ln n}} - c e^{- \frac{c}{\mu}} \sum_{d_{n}(v) = m_{n}} \frac{\hat{n}_{v}}{n} \varpi(\ln \hat{n}_{v}), \hspace*{5mm}  \Delta_{n,1}^{\prime}\coloneqq\sum_{1 \leq  d_{n}(v) \leq m_{n}} \mathbb{P}(\hat{\xi}_{v} \geq x| \mathcal{G}_{m_{n}}), 
\end{eqnarray*}

\begin{eqnarray*}
 \Delta_{n,2}^{\prime}\coloneqq\sum_{1 \leq  d_{n}(v) \leq m_{n}} \mathbb{E} \left[ \hat{\xi}_{v} \mathds{1}_{\left \{  \hat{\xi}_{v} \leq \theta \right\} }| \mathcal{G}_{m_{n}} \right] \hspace*{5mm} \text{and} \hspace*{5mm}  \Delta_{n,3}^{\prime}\coloneqq\sum_{1 \leq  d_{n}(v) \leq m_{n}} Var \left( \hat{\xi}_{v} \mathds{1}_{\left \{  \hat{\xi}_{v} \leq \theta \right\} }| \mathcal{G}_{m_{n}} \right).
\end{eqnarray*}

\noindent Recall also the notation $\Delta_{n,i}$, for $i=1,2,3$, in the statement of Theorem \ref{Theo3}. 

\begin{lemma} \label{lemma11}
Suppose that Condition \ref{Cond1} holds. Furthermore, if 
$\ln V_{1}$ is lattice with span $d$ defined in (\ref{eq51}), we also assume that Condition \ref{NewCond1} holds for some $\varrho \in [0,1)$. We have that
\begin{itemize}
\item[(i)] $ \displaystyle \Delta_{n,1} = \Delta_{n,1}^{\prime} + o_{{\rm p}}(1)$.

\item[(ii)] $ \displaystyle \sum_{1 \leq  d_{n}(v) \leq m_{n}} \mathbb{E} \left[ \xi_{v} \mathds{1}_{\left \{ \xi_{v} \leq \theta \right\} }| \mathcal{F}_{m_{n}} \right] = \Delta_{n,2}^{\prime} + o_{{\rm p}}(1)$.

\item[(iii)] $ \displaystyle \frac{\ln n}{n} \sum_{d_{n}(v) = m_{n}} n_{v} e^{- \frac{c}{\mu} \frac{\ln n_{v}}{\ln n}} - ce^{-\frac{c}{\mu} }  \sum_{d_{n}(v)=m_{n}} \frac{n_{v} \varpi(\ln n_{v})}{n} = \alpha_{n}^{\prime} + o_{{\rm p}}\left(1\right)$.

\item[(iv)] $ \displaystyle \Delta_{n,3} = \Delta_{n,3}^{\prime} + o_{{\rm p}}(1)$.
\end{itemize}

\end{lemma}

\begin{lemma} \label{lemma12}
Suppose that Condition \ref{Cond1} holds. Furthermore, if 
$\ln V_{1}$ is lattice with span $d$ defined in (\ref{eq51}), we also assume that Condition \ref{NewCond1} holds for some $\varrho \in [0,1)$. We have that
\begin{itemize}
\item[(i)] $\displaystyle \mathbb{E}[\Delta_{n,1}^{\prime}] = \nu([x,\infty)) + o(1)$ for every $x > 0$.

\item[(ii)] $\displaystyle \mathbb{E}[\Delta_{n,2}^{\prime}] =  \left( \mu m_{n}+  \frac{2c - \sigma^{2}+\mu^{2}}{2\mu} + \ln \theta - \mu \phi \left(\ln \left( \theta^{-1} e^{-\frac{c}{\mu}} \ln n \right) \right) + \psi_{\theta}(d) - \ln \ln n \right) \frac{c}{\mu}e^{-\frac{c}{\mu}}  + o(1)$.

\item[(iii)] $ \displaystyle \mathbb{E}[ \alpha_{n}^{\prime}] = e^{-\frac{c}{\mu}} \ln n + c e^{-\frac{c}{\mu}} m_{n} - ce^{-\frac{c}{\mu}} \varpi(\ln n) + o(1)$.

\item[(iv)] $\displaystyle \mathbb{E}[ \Delta_{n,3}^{\prime}] = \theta  \left(1 + \psi_{\theta}(d) \right) \frac{c}{\mu}e^{-\frac{c}{\mu}}  + o(1)$.
\end{itemize}
\end{lemma}

For any constants $\theta, x >0$ and $\beta$ large enough, we define $m_{n}^{\prime} \coloneqq \lfloor \frac{1}{2} \log_{b} \ln n \rfloor$ and  we write
\begin{eqnarray*}
\Delta_{n,1}^{\prime \prime} \coloneqq  \sum_{m_{n}^{\prime} \leq  d_{n}(v) \leq m_{n}} \mathbb{P}(\hat{\xi}_{v} \geq x| \mathcal{G}_{m_{n}}), \hspace*{5mm}  \Delta_{n,2}^{\prime \prime} \coloneqq  \sum_{m_{n}^{\prime}  \leq  d_{n}(v) \leq m_{n}} \mathbb{E} \left[ \hat{\xi}_{v} \mathds{1}_{\left \{  \hat{\xi}_{v} \leq \theta \right\} }| \mathcal{G}_{m_{n}} \right]  - \alpha_{n}^{\prime}
\end{eqnarray*}

\noindent and
\begin{eqnarray*}
 \Delta_{n,3}^{\prime \prime} \coloneqq  \sum_{m_{n}^{\prime} \leq  d_{n}(v) \leq m_{n}} Var \left( \hat{\xi}_{v} \mathds{1}_{\left \{  \hat{\xi}_{v} \leq \theta \right\} }| \mathcal{G}_{m_{n}} \right).
\end{eqnarray*}

\begin{lemma} \label{lemma13}
Suppose that Condition \ref{Cond1} holds. Furthermore, if 
$\ln V_{1}$ is lattice with span $d$ defined in (\ref{eq51}), we also assume that Condition \ref{NewCond1} holds for some $\varrho \in [0,1)$. We have that $Var \left( \mathbb{E} \left[ \Delta_{n,i}^{\prime \prime}  \Big| \mathcal{G}_{m_{n}^{\prime}} \right] \right) = o(1)$,  for $i =1,2,3$. 
\end{lemma}

\begin{lemma} \label{lemma14}
Suppose that Condition \ref{Cond1} holds. Furthermore, if 
$\ln V_{1}$ is lattice with span $d$ defined in (\ref{eq51}), we also assume that Condition \ref{NewCond1} holds for some $\varrho \in [0,1)$. We have that $E \left[ Var \left( \Delta_{n,i}^{\prime \prime}  \Big| \mathcal{G}_{m_{n}^{\prime}} \right) \right]  =o(1)$, for $i =1,2,3$. 
\end{lemma}

\begin{proof}[Proof of Theorem \ref{Theo3}]
For $v \in T_{n}^{{\rm sp}}$ such that $1 \leq  d_{n}(v) \leq m_{n}$, we observe that
\begin{eqnarray} \label{eq44}
\mathbb{P}\left(\xi_{v} \geq x| \mathcal{F}_{m_{n}} \right) = \mathbb{P}\left(\varepsilon_{v} \geq x e^{\frac{c}{\mu} } \frac{n}{\ln n}\frac{1}{n_{v}}   \Big| \mathcal{F}_{m_{n}} \right) = \left(1-p_{n} \right) \mathds{1}_{\left\{x e^{\frac{c}{\mu}}  \frac{n}{\ln n}\frac{1}{n_{v}} \leq 1 \right \}} \leq \frac{c}{\ln n}, 
\end{eqnarray}

\noindent for $x > 0$. Thus, 
\begin{eqnarray*}
\lim_{n\rightarrow \infty} \sup_{1 \leq  d_{n}(v) \leq m_{n}}\mathbb{P}\left(\xi_{v} \geq x| \mathcal{F}_{m_{n}} \right) = 0, \hspace*{5mm} \text{almost surely},
\end{eqnarray*}

\noindent for every $x > 0$, which proves (i). 

We deduce from Lemma \ref{lemma11} that $\Delta_{n,1} = \Delta_{n,1}^{\prime} + o_{{\rm p}}(1)$, 
\begin{align*}
\Delta_{n,2} & = \Delta_{n,2}^{\prime}- \alpha_{n}^{\prime} + e^{-\frac{c}{\mu}} \ln n + \frac{c}{\mu}e^{-\frac{c}{\mu}} \ln \ln n - c e^{-\frac{c}{\mu}} \varpi(\ln n) \\
& \hspace*{10mm} + \left( \frac{c \mu^{2} - c \sigma^{2}}{2 \mu^{2}} +  \mu \phi \left( \ln \left( \theta^{-1}e^{-\frac{c}{\mu}} \ln n \right) \right) \right) \frac{c}{\mu} e^{-\frac{c}{\mu}} + o_{{\rm p}}(1)
\end{align*}

\noindent and $\Delta_{n,3} = \Delta_{n,3}^{\prime} + o_{{\rm p}}(1)$. Furthermore, Lemma \ref{lemma12} shows that the expected value of the previous quantities converge to the right-hand sides of Theorem \ref{Theo3} (ii), (iii) and (iv). We complete the proof of Theorem \ref{Theo3} by showing that
\begin{eqnarray} \label{eq60}
 Var(\Delta_{n,1}^{\prime}) = o(1) \hspace*{3mm} \text{for every} \hspace*{3mm} x >0, \hspace*{5mm}  Var(\Delta_{n,2}^{\prime} - \alpha^{\prime}_{n}) = o(1) \hspace*{5mm} \text{and} \hspace*{5mm}  Var(\Delta_{n,3}^{\prime})=o(1).
\end{eqnarray}

\noindent Then an application of the Chebyshev's inequality implies Theorem \ref{Theo3} (ii), (iii) and (iv). 

Thus, we prove (\ref{eq60}). A similar argument as in the proof of Lemma \ref{lemma11} implies that 
\begin{eqnarray*}
\Delta_{n,1}^{\prime} = \Delta_{n,1}^{\prime \prime} + o(1), \hspace*{5mm} \Delta_{n,2}^{\prime} - \alpha_{n}^{\prime} = \Delta_{n,2}^{\prime \prime} + o(1) \hspace*{5mm} \text{and} \hspace*{5mm} \Delta_{n,3}^{\prime} = \Delta_{n,3}^{\prime  \prime}+o(1).
\end{eqnarray*}

Recall the well-known variance formula $Var(X) = \mathbb{E}[Var(X| \mathcal{G})] + Var(\mathbb{E}[X|\mathcal{G}])$, where $X$ is a random variable and $\mathcal{G}$ is a sub-$\sigma$-field. Consequently, a combination of the variance formula with $\mathcal{G} = \mathcal{G}_{m_{n}^{\prime}}$, Lemma \ref{lemma13} and Lemma \ref{lemma14} show (\ref{eq60}). This concludes our proof.
\end{proof}

Finally, it only remains to prove Lemmas \ref{lemma11}, \ref{lemma12}, \ref{lemma13} and \ref{lemma14}. Their proofs are close those of Lemmas 2.5, 2.6, 2.7 and 2.8 in \cite{Holm2011}, respectively. However, they are not exactly same due to the nature of the problem. Therefore, we have decided to give only complete proofs of Lemmas \ref{lemma11} and \ref{lemma12} where the main differences appear, and moreover, the key estimations for the proofs of Lemmas \ref{lemma13} and \ref{lemma14} are developed. Then, to avoid unnecessary repetitions, the interested reader can verify that Lemmas \ref{lemma13} and \ref{lemma14} follows along the lines of the proofs of Lemma 2.7 and 2.8 in \cite{Holm2011} (see also \cite[Lemmas 2.7 and 2.8]{Ceciliarxiv20102}) together with estimations used in the proof of Lemma \ref{lemma12}. 

\subsection{Proof of Lemma \ref{lemma11}}

Recall the definition of $(\hat{n}_{v}, 1 \leq  d_{n}(v) \leq m_{n}$) in (\ref{eq54}). The following result shows that $n_{v}$ is close to $\hat{n}_{v}$.

\begin{proposition} \label{proposition1}
Suppose that Condition \ref{Cond1} holds. For $0 \leq i \leq m_{n}$, let $v \in T_{n}^{{\rm sp}}$ such that $d_{n}(v) = i$.  For large enough $n$, we have that
\begin{eqnarray*}
\mathbb{P} \left( \left| n_{v} - \hat{n}_{v} \right| > n^{0.6} \right) \leq n^{-0.19}.
\end{eqnarray*}
\end{proposition}

\begin{proof}
See \cite[Lemma 1.1]{Holm2012} (which holds also in the lattice case).
\end{proof}

Recall the definition of $(\hat{\xi}_{v}, 1 \leq  d_{n}(v) \leq m_{n})$ in (\ref{eq54}). It is not difficult to deduce that
\begin{eqnarray} \label{eq45}
\mathbb{P}\left(\hat{\xi}_{v} \geq x| \mathcal{G}_{m_{n}} \right) = \left(1-p_{n} \right) \mathds{1}_{\left\{x e^{\frac{c}{\mu}}  \frac{n}{\ln n}\frac{1}{\hat{n}_{v}} \leq 1 \right \}}, \hspace*{5mm} x >0.
\end{eqnarray}

\begin{proof}[Proof of Lemma \ref{lemma11}]
We first show (i) for the non-lattice case. The lattice case follows from exactly the same argument. From (\ref{eq44}), (\ref{eq45}) and the triangle inequality, we notice that
\begin{align*}
\mathbb{E} \left[ |\Delta_{n,1} - \Delta_{n,1}^{\prime}| \right]  & \leq  (1-p_{n})\sum_{i=1}^{m_{n}} \sum_{d_{n}(v) = i} \mathbb{E} \left[ \left| \mathds{1}_{\left\{ n_{v} \geq x e^{\frac{c}{\mu}}  \frac{n}{\ln n} \right \}} - \mathds{1}_{\left\{ \hat{n}_{v} \geq x e^{\frac{c}{\mu}}  \frac{n}{\ln n} \right \}} \right|  \right] \\
& =  \frac{c}{\ln n} \sum_{i=1}^{m_{n}} \sum_{d_{n}(v) = i} \mathbb{P} \left(  n_{v} \geq x e^{\frac{c}{\mu}}  \frac{n}{\ln n}, \hat{n}_{v} < x e^{\frac{c}{\mu}}  \frac{n}{\ln n}  \right) \\
& \hspace*{10mm} + \frac{c}{\ln n} \sum_{i=1}^{m_{n}} \sum_{d_{n}(v) = i} \mathbb{P}  \left(  n_{v} < x e^{\frac{c}{\mu}}  \frac{n}{\ln n}, \hat{n}_{v} \geq x e^{\frac{c}{\mu}}  \frac{n}{\ln n} \right).
\end{align*}

\noindent Denote the first term on the right-hand side by $I_{n}^{1}$ and the second term by $I_{n}^{2}$. We first deal with $I_{n}^{1}$ and show that $I_{n}^{1} = o(1)$. For $\delta_{1} \in (0,1)$, we observe that 
\begin{align} \label{eq92}
I_{n}^{1} & \leq \frac{c}{\ln n} \sum_{i=1}^{m_{n}} \sum_{d_{n}(v) = i} \mathbb{P} \left(  n_{v} \geq  x e^{\frac{c}{\mu}}  \frac{n}{\ln n}, \hat{n}_{v} < \delta_{1} x e^{\frac{c}{\mu}}  \frac{n}{\ln n}  \right) \nonumber \\
& \hspace*{10mm}  + \frac{c}{\ln n} \sum_{i=1}^{\infty} \sum_{d_{n}(v) = i} \mathbb{P} \left(  \delta_{1} x e^{\frac{c}{\mu}}  \frac{n}{\ln n}  \leq \hat{n}_{v} <  x e^{\frac{c}{\mu}}  \frac{n}{\ln n}  \right).
\end{align}

\noindent If $d_{n}(v) = i$ for $1 \leq i \leq m_{n}$, the relationship (\ref{eq19}) implies that 
\begin{eqnarray}  \label{eq93}
\mathbb{P} \left(  n_{v} \geq  x e^{\frac{c}{\mu}}  \frac{n}{\ln n}, \hat{n}_{v} < \delta_{1} x e^{\frac{c}{\mu}}  \frac{n}{\ln n}  \right) & \leq & \mathbb{P} \left(  {\rm binomial}(n, \hat{n}_{v}/n) \geq  x e^{\frac{c}{\mu}}  \frac{n}{\ln n} - s_{1}i, \hat{n}_{v} <\delta_{1} x e^{\frac{c}{\mu}}  \frac{n}{\ln n}  \right)  \nonumber \\
& \leq & \mathbb{P} \left(  {\rm binomial} \left(n, \delta_{1} x e^{\frac{c}{\mu}}  \frac{1}{\ln n} \right) \geq  x e^{\frac{c}{\mu}}  \frac{n}{\ln n} - s_{1}i  \right) \nonumber \\
& = & \mathbb{P} \left(  {\rm binomial} \left(n, \delta_{1} x e^{\frac{c}{\mu}}  \frac{1}{\ln n} \right) - x e^{\frac{c}{\mu}}  \frac{\delta_{1}n}{\ln n} \geq  x e^{\frac{c}{\mu}}  \frac{(1- \delta_{1})n}{\ln n} - s_{1}i  \right) \nonumber \\
& \leq & C_{1} (\ln n)/n,
\end{eqnarray}

\noindent for $t \geq 0$ and some constant $C_{1} >0$; where we have used Chebyshev’s inequality and the fact that the variance of a $\text{binomial}(m, q)$ is $mq(1 -q)$, for the last inequality. On the other hand, Lemma \ref{ExtraLemma} (i) implies that 
\begin{eqnarray*} 
\lim_{n \rightarrow \infty}  \frac{c}{\ln n} \sum_{i=1}^{\infty} \sum_{d_{n}(v) = i} \mathbb{P} \left(  \delta_{1} x e^{\frac{c}{\mu}}  \frac{n}{\ln n}  \leq \hat{n}_{v} <  x e^{\frac{c}{\mu}}  \frac{n}{\ln n}  \right)= (\delta^{-1}_{1} - 1) c\mu^{-1}x^{-1}e^{-\frac{c}{\mu}}.  
\end{eqnarray*}

\noindent By combining the previous limit and the estimate (\ref{eq93}) into (\ref{eq92}), we obtain that
\begin{eqnarray*}
\limsup_{n \rightarrow \infty} I_{n}^{1} = (\delta^{-1}_{1} - 1) c\mu^{-1}x^{-1}e^{-\frac{c}{\mu}}.  
\end{eqnarray*}

\noindent By the arbitrariness of $\delta_{1} \in (0,1)$, we deduce that $I_{n}^{1} = o(1)$. We complete the proof of (i) by showing that $I_{n}^{2} = o(1)$. For $\delta_{2} > 1$, we observe that 
\begin{eqnarray*}
I_{n}^{2} \leq \frac{c}{\ln n} \sum_{i=1}^{m_{n}} \sum_{d_{n}(v) = i} \mathbb{P} \left(  n_{v} <  x e^{\frac{c}{\mu}}  \frac{n}{\ln n}, \hat{n}_{v} \geq \delta_{2} x e^{\frac{c}{\mu}}  \frac{n}{\ln n}  \right) + \frac{c}{\ln n} \sum_{i=1}^{\infty} \sum_{d_{n}(v) = i} \mathbb{P} \left(   x e^{\frac{c}{\mu}}  \frac{n}{\ln n}  \leq \hat{n}_{v} < \delta_{2} x e^{\frac{c}{\mu}}  \frac{n}{\ln n}  \right).
\end{eqnarray*}

\noindent But one can show via similar arguments that $I_{n}^{2} = o(1)$; details are left to the reader. Then, an application of the Markov's inequality combined with the previous estimates concludes the proof of (i). 

We next establish (ii). We observe that
\begin{eqnarray*}
\mathbb{E} \left[ \xi_{v} \mathds{1}_{\left \{ \xi_{v} \leq \theta \right\} } | \mathcal{F}_{m_{n}} \right] = (1-p_{n}) \frac{\ln n}{n} e^{-\frac{c}{\mu}} n_{v}  \mathds{1}_{\left\{ n_{v} \leq  \theta e^{\frac{c}{\mu}} \frac{n}{\ln n}\right \}}
\end{eqnarray*}

\noindent and 
\begin{eqnarray*}
 \Delta_{n,2}^{\prime} = \mathbb{E} \left[ \hat{\xi}_{v} \mathds{1}_{\left \{ \hat{\xi}_{v} \leq \theta \right\} } | \mathcal{G}_{m_{n}} \right] = (1-p_{n}) \frac{\ln n}{n} e^{-\frac{c}{\mu}} \hat{n}_{v}  \mathds{1}_{\left\{ \hat{n}_{v} \leq  \theta e^{\frac{c}{\mu}} \frac{n}{\ln n}\right \}}.
\end{eqnarray*}

\noindent Then triangle inequality implies that
\begin{align*}
\mathbb{E} \left[ \left|\sum_{1 \leq  d_{n}(v) \leq m_{n}} \mathbb{E} \left[ \xi_{v}
         \mathds{1}_{\left \{ \xi_{v} \leq \theta \right\} } | \mathcal{F}_{m_{n}} \right]-
        \Delta_{n,2}^{\prime} \right| \right] &  \leq  \frac{c}{n} e^{-\frac{c}{\mu}} \sum_{i=1}^{m_{n}} \sum_{d_{n}(v) = i} \mathbb{E} \left[\left | n_{v} \mathds{1}_{\left\{ n_{v} \leq  \theta e^{\frac{c}{\mu}} \frac{n}{\ln n}\right \}} - \hat{n}_{v} \mathds{1}_{\left\{ \hat{n}_{v} \leq \theta e^{\frac{c}{\mu}} \frac{n}{\ln n} \right \}}\right|  \right] \\
& \leq \frac{c}{n} e^{-\frac{c}{\mu}} \sum_{i=1}^{m_{n}}  \sum_{d_{n}(v) = i}  \mathbb{E} \left[\left | n_{v} - \hat{n}_{v} \right| \mathds{1}_{\left\{ n_{v} \leq  \theta e^{\frac{c}{\mu}} \frac{n}{\ln n}\right \}} \right] \\
& \hspace*{5mm} + \frac{c}{n} e^{-\frac{c}{\mu}} \sum_{i=1}^{m_{n}}  \sum_{d_{n}(v) = i}  \mathbb{E} \left[ \hat{n}_{v}  \left | \mathds{1}_{\left\{ n_{v} \leq  \theta e^{\frac{c}{\mu}} \frac{n}{\ln n}\right \}} - \mathds{1}_{\left\{ \hat{n}_{v} \leq \theta e^{\frac{c}{\mu}} \frac{n}{\ln n} \right \}}\right|  \right].
\end{align*}

\noindent On the one hand, Proposition \ref{proposition1}  implies that
\begin{eqnarray*}
 \frac{c}{n} e^{-\frac{c}{\mu}} \sum_{i=1}^{m_{n}} \sum_{d_{n}(v) = i}   \mathbb{E} \left[\left | n_{v} - \hat{n}_{v} \right| \mathds{1}_{\left\{ n_{v} \leq  \theta e^{\frac{c}{\mu}} \frac{n}{\ln n}\right \}} \right] \leq \frac{c}{n} e^{-\frac{c}{\mu}} \sum_{i=1}^{m_{n}} \sum_{d_{n}(v) = i}   \mathbb{E} \left[\left | n_{v} - \hat{n}_{v} \right| \right] = o(1).
\end{eqnarray*}

\noindent On the other hand, a similar computation as in the proof of point (i) shows that 
\begin{eqnarray*}
\frac{c}{n} e^{-\frac{c}{\mu}} \sum_{i=1}^{m_{n}} \sum_{d_{n}(v) = i}  \mathbb{E} \left[ \hat{n}_{v}  \left | \mathds{1}_{\left\{ n_{v} \leq  \theta e^{\frac{c}{\mu}} \frac{n}{\ln n}\right \}} - \mathds{1}_{\left\{ \hat{n}_{v} \leq \theta e^{\frac{c}{\mu}} \frac{n}{\ln n} \right \}}\right|  \right] = o(1). 
\end{eqnarray*}

\noindent Thus, a combination of the previous estimates with the Markov inequality shows (ii). 

We continue with the proof of (iii). An application of the triangle inequality implies that
\begin{align} \label{eq66}
& \mathbb{E} \left[ \left | \frac{\ln n}{n} \sum_{d_{n}(v) = m_{n}} n_{v} e^{- \frac{c}{\mu} \frac{\ln n_{v}}{\ln n}} -  ce^{-\frac{c}{\mu} }  \sum_{d_{n}(v)=m_{n}} \frac{n_{v} \varpi(\ln n_{v})}{n} - \alpha_{n}^{\prime} \right| \right] \nonumber \\
& \hspace*{10mm} \leq \frac{\ln n}{n} b^{m_{n}}  \mathbb{E} \left[  \left |  n_{v} e^{- \frac{c}{\mu} \frac{\ln n_{v}}{\ln n}} - \hat{n}_{v} e^{- \frac{c}{\mu} \frac{\ln \hat{n}_{v}}{\ln n}}  \right| \right] + \frac{c}{n}e^{-\frac{c}{\mu} } b^{m_{n}}  \mathbb{E} \left[ \left | n_{v} \varpi(\ln n_{v}) - \hat{n}_{v} \varpi(\ln \hat{n}_{v}) \right| \right]. 
\end{align}

\noindent By using Proposition \ref{proposition1}, a similar argument as in the proof of point (ii) shows that 
\begin{align} \label{eq67}
& \frac{\ln n}{n} b^{m_{n}}  \mathbb{E} \left[  \left |  n_{v} e^{- \frac{c}{\mu} \frac{\ln n_{v}}{\ln n}} - \hat{n}_{v} e^{- \frac{c}{\mu} \frac{\ln \hat{n}_{v}}{\ln n}}  \right| \right] \nonumber \\
& \hspace*{10mm} \leq \frac{\ln n}{n} b^{m_{n}} \left(  \mathbb{E} \left[  \left |  n_{v} - \hat{n}_{v} \right| \right] +  \mathbb{E} \left[ \hat{n}_{v}  \left |  e^{- \frac{c}{\mu} \frac{\ln n_{v}}{\ln n}} -  e^{- \frac{c}{\mu} \frac{\ln \hat{n}_{v}}{\ln n}}  \right| \right]\right) = o(1). 
\end{align}

\noindent On the other hand, the triangle inequality and  Proposition \ref{proposition1} imply that 
\begin{eqnarray*}
\mathbb{E} \left[ \left | n_{v} \varpi(\ln n_{v}) - \hat{n}_{v} \varpi(\ln \hat{n}_{v}) \right| \right] & \leq &  \mathbb{E} \left[ \varpi(\ln n_{v}) \left | n_{v}  - \hat{n}_{v} \right| \right] + \mathbb{E} \left[ \hat{n}_{v} \left | \varpi(\ln n_{v}) -  \varpi(\ln \hat{n}_{v}) \right| \right] \\
& = & \mathbb{E} \left[ \hat{n}_{v} \left | \varpi(\ln n_{v}) -  \varpi(\ln \hat{n}_{v}) \right| \right] + o(n b^{-m_{n}}), 
\end{eqnarray*}

\noindent where we have used that $\varpi$ is a continuous $d$-periodic function, with $d$ defined in (\ref{eq51}), and thus it is bounded. We notice that 
\begin{align} \label{eq109}
\mathbb{E} \left[ \hat{n}_{v} \left | \varpi(\ln n_{v}) -  \varpi(\ln \hat{n}_{v}) \right| \right] & = \mathbb{E} \left[ \hat{n}_{v} \left | \varpi(\ln n_{v}) -  \varpi(\ln \hat{n}_{v}) \right| \mathds{1}_{\{ |n_{v} - \hat{n}_{v}| \leq \hat{n}_{v}^{2/3} \}} \right] \nonumber \\ 
& \hspace*{10mm} + \mathbb{E} \left[ \hat{n}_{v} \left | \varpi(\ln n_{v}) -  \varpi(\ln \hat{n}_{v}) \right| \mathds{1}_{\{ |n_{v} - \hat{n}_{v}| > \hat{n}_{v}^{2/3} \}} \right].
\end{align}

\noindent It is not difficult to see that in the event $\{ |n_{v} - \hat{n}_{v}| < \hat{n}_{v}^{2/3} \}$, we can make $|\ln n_{v} - \ln \hat{n}_{v}|$ arbitrary small by taking $n$ large enough. Hence the continuity of the function $\varpi$ allows us to deduce that
\begin{eqnarray} \label{eq110}
\mathbb{E} \left[ \hat{n}_{v} \left | \varpi(\ln n_{v}) -  \varpi(\ln \hat{n}_{v}) \right| \right]  = \mathbb{E} \left[ \hat{n}_{v} \left | \varpi(\ln n_{v}) -  \varpi(\ln \hat{n}_{v}) \right| \mathds{1}_{\{ |n_{v} - \hat{n}_{v}| > \hat{n}_{v}^{2/3} \}} \right] + o(n b^{-m_{n}}). 
\end{eqnarray}

\noindent Recall that a binomial random variable with parameters $(n,q)$ has expected value $nq$ and variance $nq(1-q)$. Following the same reasoning as in the proof of Proposition \ref{proposition1}, we deduce from an application of (\ref{eq19}) and the conditional version of Chebyshev's inequality that 
\begin{eqnarray} \label{eq111}
\mathbb{E} \left[ \hat{n}_{v} \mathds{1}_{\{ |n_{v} - \hat{n}_{v}| > \hat{n}_{v}^{2/3} \}} \right] = 4\mathbb{E}[\hat{n}_{v}^{2/3}] = o(n b^{-m_{n}}).
\end{eqnarray}

\noindent By recalling that the function $\varpi$ is continuous and thus bounded, the estimations (\ref{eq109}), (\ref{eq110}) and (\ref{eq111}) imply that
\begin{eqnarray} \label{eq69}
b^{m_{n}}  \mathbb{E} \left[ \left | n_{v} \varpi(\ln n_{v}) - \hat{n}_{v} \varpi(\ln \hat{n}_{v}) \right| \right]= o(n).
\end{eqnarray}

\noindent Therefore, the combination of (\ref{eq67}) and (\ref{eq69}) into (\ref{eq66}) implies 
\begin{eqnarray*}
\mathbb{E} \left[ \left | \frac{\ln n}{n} \sum_{d_{n}(v) = m_{n}} n_{v} e^{- \frac{c}{\mu} \frac{\ln n_{v}}{\ln n}} -  ce^{-\frac{c}{\mu} }  \sum_{d_{n}(v)=m_{n}} \frac{n_{v} \varpi(\ln n_{v})}{n} - \alpha_{n}^{\prime} \right| \right]= o(1)
\end{eqnarray*}

\noindent which together with the Markov inequality proves (iii). 

Finally, point (iv) follows from a similar argument as in the proof of (ii) by using Proposition \ref{proposition1}.
\end{proof}

\subsection{Proof of Lemma \ref{lemma12}}

We observe that $(n_{v}, d_{n}(v) = i)$ is a sequence of identically distributed random variables, for $1 \leq i \leq m_{n}$. Moreover, the distribution of $n_{v}$ for $v \in T_{n}^{{\rm sp}}$ such that $d_{n}(v) = i$ is determined by the sequence $(W_{v,k}, k=1, \dots, i)$ of i.i.d.\ random variables on $[0,1]$ given by the split vectors associated with the vertices on the unique path from $v$ to the root. We introduce the notation $Y_{v, i} \coloneqq - \sum_{k=1}^{i} \ln W_{v, k}$. We sometimes omit the vertex index of $(W_{v,k}, k=1, \dots, i)$ and we just write $(W_{k}, k=1, \dots, i)$ when it is free of ambiguity. Similarly, we write $Y_{i}$ instead of $Y_{v, i}$. 

\begin{proof}[Proof of Lemma \ref{lemma12}]
Recall our assumption (\ref{eq9}) in the percolation parameter, i.e., $p_{n} = 1 - c/\ln n$, where $c > 0$ is fixed. We first show (i) in the non-lattice case. From the identity (\ref{eq45}), we deduce that
\begin{eqnarray*}
\mathbb{E} \left[ \Delta_{n,1}^{\prime} \right] = (1-p_{n})\sum_{i=1}^{m_{n}} \mathbb{E} \left[ \sum_{d_{n}(v)=i}  \mathds{1}_{\left\{ x e^{\frac{c}{\mu}}  \frac{n}{\ln n}\frac{1}{\hat{n}_{v}} \leq 1 \right \}}   \right] =  (1-p_{n})\sum_{i=1}^{m_{n}} b^{i}\mathbb{P} \left( Y_{i} \leq \ln \left( x^{-1} e^{-\frac{c}{\mu}}  \ln n \right)  \right).
\end{eqnarray*}

\noindent By Lemma \ref{ExtraLemma}, we obtain that
\begin{eqnarray} \label{eq53}
\sum_{i=1}^{\infty} b^{i}\mathbb{P} \left( Y_{i} \leq \ln \left( x^{-1} e^{-\frac{c}{\mu}}  \ln n \right)  \right) = \left(\mu^{-1} + o(1) \right) x^{-1} e^{-\frac{c}{\mu}}  \ln n = \mu^{-1}e^{-\frac{c}{\mu}} x^{-1} \ln n +o(\ln n).
\end{eqnarray}

\noindent Thus (i) follows from (\ref{eq53}) by providing that
\begin{eqnarray} \label{eq47}
(1-p_{n})\sum_{i=m_{n}+1}^{\infty} b^{i}\mathbb{P} \left( Y_{i} \leq \ln \left( x^{-1} e^{-\frac{c}{\mu}}  \ln n \right)  \right) = o(1).
\end{eqnarray}

\noindent Choose an arbitrary $t >0$. By an application of the Markov
inequality and the fact that $(W_{k}, k=1, \dots, i)$ are i.i.d.\ random variables, we obtain that
\begin{eqnarray} \label{eq46}
 \mathbb{P} \left( Y_{i} \leq \delta  \right) = \mathbb{P} \left( e^{-t Y_{i}} \geq e^{-\delta t}  \right) \leq m(t)^{i}e^{\delta t}, 
\end{eqnarray}

\noindent for $\delta > 0$, where we define $m(t) \coloneqq \mathbb{E}[V_{1}^{t}]$ for $t >0$.
Then,
\begin{eqnarray} \label{eq52}
(1-p_{n})\sum_{i=m_{n}+1}^{\infty} b^{i}\mathbb{P} \left( Y_{i} \leq \ln \left( x^{-1} e^{-\frac{c}{\mu}}  \ln n \right)  \right) \leq cx^{-t} e^{-\frac{c}{\mu}} (\ln n)^{t-1} \sum_{i = m_{n}+1}^{\infty} (m(t)b)^{i}.
\end{eqnarray}

\noindent Thus our claim (\ref{eq47}) follows after some computations by taking $t >0$ such that $bm(t) <1$ (this is possible by Condition \ref{Cond1}) and $\beta > \max( (1-t)/ \log_{b}(bm(t)), -2/(1+\log_{b} \mathbb{E}[V_{1}^{2}]))$.

In the lattice case, we see that (\ref{eq53}) becomes
\begin{eqnarray*}
\sum_{i=1}^{\infty} b^{i}\mathbb{P} \left( Y_{i} \leq \ln \left( x^{-1} e^{-\frac{c}{\mu}}  \ln n \right)  \right) & =  &\left(\frac{d}{\mu}\frac{1}{1-e^{-d}} + o(1) \right) e^{d \lfloor d^{-1}  \ln(  x^{-1} e^{-\frac{c}{\mu}}  \ln n )  \rfloor} \nonumber \\
& = & \frac{d}{\mu}\frac{1}{1-e^{-d}} e^{d \lfloor d^{-1}  \ln(  x^{-1} e^{-\frac{c}{\mu}})+ \{ d^{-1} \ln \ln n \}   \rfloor - d \{ d^{-1} \ln \ln n \}  }  \ln n + o(\ln n),
\end{eqnarray*}
\noindent and the results follows exactly as in the non-lattice case.

We next establish (ii) only in the non-lattice case. The lattice case follows from exactly the same argument. We observe that
\begin{eqnarray*}
\mathbb{E}[ \Delta_{n,2}^{\prime}] = \frac{\ln n}{n}(1-p_{n})e^{-\frac{c}{\mu}} \sum_{i=1}^{m_{n}} \mathbb{E} \left[ \sum_{d_{n}(v) = i} \hat{n}_{v} \mathds{1}_{\left \{ \hat{n}_{v} \leq \theta e^{\frac{c}{\mu}} \frac{n}{\ln n}  \right\}}  \right] = c e^{-\frac{c}{\mu}} \sum_{i=1}^{m_{n}} b^{i} \mathbb{E} \left[ e^{-Y_{i}} \mathds{1}_{\left \{ Y_{i} \geq \ln \left( \theta^{-1} e^{-\frac{c}{\mu}} \ln n \right)  \right\}}  \right].
\end{eqnarray*}

\noindent By noticing that $\mathbb{E}[e^{-Y_{i}}] = b^{-i}$, we use integration by parts to obtain that
\begin{eqnarray*}
\mathbb{E}[ \Delta_{n,2}^{\prime}]   
& = &  c e^{-\frac{c}{\mu}}  m_{n} -  \frac{c \theta}{\ln n} \sum_{i=1}^{m_{n}} b^{i}
\mathbb{P}\left(Y_{i} \leq \ln \left( \theta^{-1} e^{-\frac{c}{\mu}} \ln n \right) \right) \\
&   & \quad -  c e^{-\frac{c}{\mu}} \int_{0}^{\ln \left( \theta^{-1} e^{-\frac{c}{\mu}} \ln n \right)} e^{-z} \sum_{i=1}^{m_{n}} b^{i} \mathbb{P}(Y_{i} \leq z) {\rm d} z \\
& = & c e^{-\frac{c}{\mu}}  m_{n} - \frac{c}{\mu} e^{- \frac{c}{\mu}} - c e^{-\frac{c}{\mu}} \int_{0}^{\ln \left( \theta^{-1} e^{-\frac{c}{\mu}} \ln n \right)} e^{-z}\sum_{i=1}^{m_{n}} b^{i} \mathbb{P}(Y_{i} \leq z) {\rm d} z + o(1),
\end{eqnarray*}

\noindent where we have used (\ref{eq53}) and (\ref{eq47}), with $t >0$ such that $bm(t) < 1$ and $\beta > \max( (1-t)/ \log_{b}(bm(t)), -2/(1+\log_{b} \mathbb{E}[V_{1}^{2}]))$,  in order to get the last equality. 

On the other hand, we deduce from (\ref{eq52}) that 
\begin{eqnarray*} 
\int_{0}^{\ln \left( \theta^{-1} e^{-\frac{c}{\mu}} \ln n \right)} e^{-z}\sum_{i=m_{n}+1}^{\infty} b^{i}\mathbb{P} \left( Y_{i} \leq z  \right) {\rm d} z \leq  \theta^{-t} e^{-\frac{c}{\mu}} (\ln n)^{t} \left( \sum_{i = m_{n}+1}^{\infty} (m(t)b)^{i} \right) \ln \left( \theta^{-1} e^{-\frac{c}{\mu}} \ln n \right) = o(1),
\end{eqnarray*}

\noindent when $t >0$ such that $bm(t) <1$ (this is possible by Condition \ref{Cond1}) and $\beta > \max( -t/ \log_{b}(bm(t)),  (1-t)/ \log_{b}(bm(t)), -2/(1+\log_{b} \mathbb{E}[V_{1}^{2}]))$. Hence 
\begin{eqnarray*}
\mathbb{E}[ \Delta_{n,2}^{\prime}]  =  c e^{-\frac{c}{\mu}} m_{n} - \frac{c}{\mu} e^{-\frac{c}{\mu}}  -  c e^{-\frac{c}{\mu}} \int_{0}^{\ln \left( \theta^{-1} e^{-\frac{c}{\mu}} \ln n \right)} e^{-z} \sum_{i=1}^{\infty} b^{i} \mathbb{P}(Y_{i} \leq z) {\rm d} z + o(1).
\end{eqnarray*}

\noindent By the result in (\ref{eq33}), we know that
\begin{eqnarray*}
 \int_{0}^{\ln \left( \theta^{-1} e^{-\frac{c}{\mu}} \ln n \right)} e^{-z}\left( \sum_{i=1}^{\infty} b^{i} \mathbb{P}(Y_{i} \leq z) -\mu^{-1} e^{z}\right){\rm d} z =\frac{\sigma^{2}-\mu^{2}}{2\mu^{2}}- \mu^{-1} + \phi \left(\ln \left( \theta^{-1} e^{-\frac{c}{\mu}} \ln n \right) \right) +  o(1),
\end{eqnarray*}

\noindent where $\phi: \mathbb{R} \rightarrow \mathbb{R}$ is the $d$-periodic continuous function in (\ref{eq33}). Therefore,
\begin{eqnarray*}
\mathbb{E}[ \Delta_{n,2}^{\prime}]  =  c e^{-\frac{c}{\mu}} m_{n} +  \frac{2c^{2} -c\sigma^{2}+c\mu^{2}}{2\mu^{2}}  e^{-\frac{c}{\mu}}  - ce^{-\frac{c}{\mu}}\phi \left(\ln \left( \theta^{-1} e^{-\frac{c}{\mu}} \ln n \right) \right) + \frac{c}{\mu}e^{-\frac{c}{\mu}} \ln \theta  - \frac{c}{\mu}e^{-\frac{c}{\mu}}  \ln  \ln n + o(1) 
\end{eqnarray*}

\noindent which proves point (ii).

We continue with the proof of (iii). Recall the function $m(t) = \mathbb{E}[V_{1}^{t}]$ for $t >0$. From the definition of $\hat{n}_{v}$ in (\ref{eq54}), we deduce that
\begin{eqnarray*}
\mathbb{E}[\alpha_{n}^{\prime}] & = & b^{m_{n}} e^{-\frac{c}{\mu}} m\left( 1- \frac{c}{\mu} \frac{1}{\ln n} \right)^{m_{n}} \ln n - ce^{-\frac{c}{\mu}} b^{m_{n}} \mathbb{E} \left[ \prod_{k=1}^{m_{n}} W_{k} \varpi \left( \ln n + \sum_{k=1}^{m_{n}} \ln W_{k} \right) \right] \\
& = & b^{m_{n}} e^{-\frac{c}{\mu}} m\left( 1- \frac{c}{\mu} \frac{1}{\ln n} \right)^{m_{n}} \ln n - ce^{-\frac{c}{\mu}} \varpi(\ln n),
\end{eqnarray*}

\noindent  since $\varpi$ is $d$-periodic, with $d$ defined in (\ref{eq51}), and  $\ln W_{k} \in d \mathbb{Z}$. We notice that $m(1) = \mathbb{E}[V_{1}] = 1/b$ and $m^{\prime}(1) = \mathbb{E}[V_{1} \ln V_{1}] = - \mu/b$.  Then a simple Taylor's expansion calculation shows that
\begin{eqnarray*}
m\left( 1- \frac{c}{\mu} \frac{1}{\ln n} \right) = \frac{1}{b} + \frac{c}{b\ln n} + o\left(\frac{1}{b \ln^{2} n} \right)
\end{eqnarray*}

\noindent which implies that 
\begin{eqnarray*}
\mathbb{E}[\alpha_{n}^{\prime}] = e^{-\frac{c}{\mu}} \ln n + c e^{-\frac{c}{\mu}} m_{n} - ce^{-\frac{c}{\mu}} \varpi(\ln n) + o(1),
\end{eqnarray*}

\noindent and completes the proof of (iii). 

We finally show (iv) only in the non-lattice case. The lattice case follows from exactly the same argument. We notice that
\begin{eqnarray*}
\mathbb{E}[ \Delta_{n,3}^{\prime}] & = & \frac{\ln^{2}n}{n^{2}}(1-p_{n})p_{n}e^{-2\frac{c}{\mu}} \sum_{i=1}^{m_{n}} \mathbb{E} \left[ \sum_{d_{n}(v) = i} \hat{n}_{v}^{2} \mathds{1}_{\left \{ \hat{n}_{v} \leq \theta e^{\frac{c}{\mu}} \frac{n}{\ln n}  \right\}}   \right] \\
&= & c e^{-2\frac{c}{\mu}}  (\ln n)  p_{n} \sum_{i=1}^{m_{n}} b^{i} \mathbb{E} \left[ e^{-2Y_{i}} \mathds{1}_{\left \{ Y_{i} \geq \ln \left( \theta^{-1} e^{-\frac{c}{\mu}} \ln n \right)  \right\}}  \right].
\end{eqnarray*}

\noindent By integration by parts,  we obtain that
\begin{align*}
\mathbb{E}[ \Delta_{n,3}^{\prime}]  &=   - \frac{c \theta^{2} p_{n}}{\ln n}  \sum_{i=1}^{m_{n}} b^{i} \mathbb{P}\left(Y_{i} \leq \ln \left( \theta^{-1} e^{-\frac{c}{\mu}} \ln n \right) \right)  \\
& \hspace*{10mm} + 2c e^{-2\frac{c}{\mu}} (\ln n)  p_{n} \int_{\ln \left( \theta^{-1} e^{-\frac{c}{\mu}} \ln n \right)}^{\infty} e^{-2z} \sum_{i=1}^{m_{n}} b^{i} \mathbb{P}(Y_{i} \leq z) {\rm d} z \\
& =  c \mu^{-1} e^{-\frac{c}{\mu}} \theta + o(1),
\end{align*}

\noindent where we have used (\ref{eq53}) and (\ref{eq47}), with $t >0$ such that $bm(t) < 1$ and $\beta > \max( -t/ \log_{b}(bm(t)), (1-t)/ \log_{b}(bm(t)), -2/(1+\log_{b} \mathbb{E}[V_{1}^{2}]))$, in order to get the last equality.  This concludes the proof of (iv).
\end{proof}

\end{appendices}

\end{document}